\documentclass[final,reqno]{amsart}
\usepackage{pdfsync}
\usepackage{amsmath,amssymb}
\usepackage{enumerate}
\usepackage{xspace}
\usepackage{array}
\usepackage{tabularx}
\usepackage{subfigure}
\usepackage{caption}
\usepackage{fullpage}
\usepackage{james}
\newcommand{\revise}[1]{#1}

\usepackage{mathtools}
\mathtoolsset{showonlyrefs=true}

\author{
 Elena Celledoni
}
\address{
  Elena Celledoni
  \thanks{
    Department of Mathematical Sciences, NTNU, 7491 Trondheim, Norway
    {\tt{elena.celledoni@ntnu.no}}.
}}

\author{
  James Jackaman
}
\address{
  James Jackaman
  \thanks{
    Department of Mathematics and Statistics, Memorial University of
    Newfoundland, St.\ John's, NL, A1C 5S7, Canada
    {\tt{jjackaman@mun.ca}}.
}}

\title[Discrete conservation laws]
      {Discrete conservation laws for finite element discretisations
        of multisymplectic PDEs}
      \date{\today}

\begin{document}
\maketitle

\begin{abstract}

  In this work we propose a new, arbitrary order space-time finite
  element discretisation for Hamiltonian PDEs in multisymplectic
  formulation. We show that the new method which is obtained by using
  both continuous and discontinuous discretisations in space, admits a
  local and global conservation law of energy. We \revise{also} show
  existence and uniqueness of \revise{solutions} of the discrete
  equations. Further, we illustrate the error behaviour and the
  conservation properties of the proposed discretisation in extensive
  numerical experiments on the linear and nonlinear wave equation and
  on the nonlinear Schr{\"o}dinger equation.
  
\end{abstract}


\section{Introduction}
\label{sec:introduction}

Finite element discretisations of space-time variational problems have
seen a revival of interest in the recent literature
\cite{henning19mor, perugia20tpa, antonietti18aho, urban14aie}, with
their origin going back to the work of \cite{lions68pal,
  AzizMonk:1989}. The focus of the present work is the
structure-preserving discretisation of variational problems using
finite element methods. The point of departure is the variational
space-time formulation of PDE problems arising as the Euler-Lagrange
equations of a space-time action functional. Formally via the Legendre
transform one obtains the Hamiltonian formulation of these partial
differential equations \cite{Olver}.  However a space-time analogue of
the Legendre transform gives rise to the so called multisymplectic
formulation of these PDEs as originally proposed by Bridges
\cite{Bridges:1997}, and intrinsically generalised in
\cite{marsden97mgc, marsden98mgv}.

There are two main proposed discretisation approaches to the
multisymplectic formulation of Hamiltonian PDEs. The first is
inspired by a technique proposed by Veselov to discretise Hamiltonian
ODEs and consists of {\it discretising} the Lagrangian density to
yield a discrete analogue of the variational principle, and {\it then
  extremising} to obtain discrete Euler Lagrange equations,
\cite{marsden98mgv}. The second is obtained by first {\it extremising}
the variational principle, writing the multisymplectic PDEs in the
strong multisymplectic formulation, and {\it then discretising} these
equations with structure preserving approximations,
\cite{BridgesReich:2001}.  The first approach leads automatically to
the conservation of a discrete multisymplectic conservation law and
the corresponding numerical methods typically show as a side effect
good local conservation of a modified energy and momentum in numerical
experiments. The second approach is more flexible, while
discretisations with similar properties can be obtained using
appropriate multisymplectic schemes \cite{BridgesReich:2001}, this
formulation can be also easily adopted to obtain methods satisfying
local conservation laws of energy or momentum
\cite{GongCaiWang:2014}. In the sequel we shall follow the latter
approach. We note for Hamiltonian ODEs that symplectic schemes are often
preferred over conservative schemes as they allow for backward error
analysis, and the resulting schemes will preserve a modified conserved
quantity \cite{HairerLubichWanner:2006}. A backward error analysis for
multisymplectic schemes in the PDE setting is not yet fully
developed, \cite{MooreReich:2003, IslasSchober:2005}. On the other
hand, being able to bound the approximation by the energy often proves
to be an invaluable property for convergence studies of both geometric
numerical integrators \cite{CourantFriedrichsLewy:1967,
  BuchholzEtAl:2018, BuchholzDoerichHochbruck:2020} and finite element
schemes via energy arguments \cite{Thomee:2006}.

Most of the proposed discretisations of multisymplectic PDEs are
proposed in the framework of finite differences with restrictions to
rectangular domains, are not easy to use on irregular domains, and are
often of low order with some exceptions \cite{zhen03sam, tang17dgm,
  McLachlanStern:2017, reich00msr, mclachlan14hom, Beirao:2017}. Due
to the nature of their formulation, finite element methods may
overcome these issues. It has also been remarked \cite{ryland08omo}
that some of the proposed multisymplectic discretisations might not
be well defined locally and or globally or not have solutions/unique
solutions, \cite{AscherMcLachlan:2004}, \cite{mclachlan14hom}.

Standard finite element methods, when arising from boundary value
problems, do not historically lead to discretisations which may be interpreted
locally. Indeed, we may on a case-by-case basis move out of the finite
element framework through the assembling of the associated algebraic
system which may then be interpreted locally similarly to a spatial
finite difference discretisation, see for example
\cite{KarakashianMakridakis:1999}. The recent work
\cite{McLachlanStern:2017} utilises the \emph{hybridisable
  discontinuous Galerkin} framework
\cite{CockburnGopalakrishnan:2004}, in which the solution inside an
element and on the element boundaries are considered independently and
coupled to ensure global communication of the solution. Fortuitously,
standard finite element methods also fit within this framework
\cite{cockburn09uho}, and as such global solution properties may be
naturally restricted to the patch surrounding a single element.

Another novel approach which allows for the incorporation of spatially
local conservation laws are local discontinuous Galerkin \emph{(ldG)}
methods \cite{CockburnShu:1998, CastilloEtAl:2000, XuShu:2005}, in
which a given PDE is reduced to a first order system through the
incorporation of auxiliary variables. The resulting discretisation is
comprised of discontinuous approximations of first derivatives
utilising either upwind or downwind flux contributions, which allow
for the development of schemes which preserve conservation laws for
PDEs with even order spatial derivatives \cite{XingChouShu:2013,
  CastilloGomez:2018}. A crucial implementational benefit of these
methods is that due to the discontinuous nature of the approximation
the discrete system may be rewritten as a single equation, and as such
the resulting algorithms are of a competitive complexity. In this
work, we shall consider similar spatial discretisations, however, due
to the general multisymplectic framework considered we choose an
\emph{average} spatial flux, as opposed to the upwind/downwind
considered in the ldG setting. Our choice here is similar to that made
in \cite{self:defocus}. This average spatial flux choice allows us to
utilise a discrete integration by parts which is fundamental to the
preservation of conservation laws in a general setting.  If we
restrict ourselves to consider multisymplectic PDEs with even order
spatial derivatives, appropriate modifications of the spatial
derivative may be designed to preserve local conservation laws for
this class of multisymplectic PDEs with upwind/downwind fluxes, which
often yields an optimal rate of convergence in the resulting numerical
scheme.

Temporally, finite element methods are a competitive, well studied,
class of methods \cite{EstepFrench:1994, FrenchSchaeffer:1990,
  Estep:1995} dating back to 1969 \cite{Fried:1969}. Further to this,
the conforming approximation is well known to preserve energy for
Hamiltonian problems \cite{BetchSteinmann:2000, Hansbo:2001,
  TangSun:2012, self:thesis}. For space-time finite element
approximations the temporal are typically nonconforming, with the
standard choice of a discontinuous upwind flux introducing
artificial diffusion \cite{EstepStuart:2002} which facilitates stability,
even when the space-time algorithm is adapted. Here we shall focus our
attention on the temporally conforming approximation, with a view to
extend this work to incorporate a new form of \emph{conservative}
space-time adaptive algorithm with inherent stability building on
ideas developed in \cite{self:thesis}.

The outline of the paper is as follows, in Section~\ref{sec:multisym}
we briefly review the main features of multisymplectic PDEs including
examples; in Section~\ref{sec:cg} we present the continuous space-time
finite element discretisation and prove existence and uniqueness of
solutions of the discrete equations at the end of the section; in
Section~\ref{sec:dg} we extend the method to the spatially
discontinuous case; in Section~\ref{sec:numerics} we conduct numerical
experiments; and in Section~\ref{sec:conclusion} we conclude.

\section{Multisymplectic PDEs} \label{sec:multisym}

Let $u=u(t,x)$ where $t \in [0,T]$ and $x \in S^1 = [0,1)$ periodic
and consider the space-time variational problem
\begin{equation}
  \label{eq:0}
  0
  =
  \delta \int_{[0,T] \times S^1} L(u,u_t,u_x)\, \di{x} \,\di{t}
  ,
\end{equation}
where $L(u,u_t,u_x)$ is some given Lagrangian density
function. Multisymplectic PDEs may arise naturally from such problems
through the following methodology. Introducing the auxiliary variables
\begin{equation} \nonumber
  v:=\frac{\partial L}{\partial u_t},
  \qquad
  w:=\frac{\partial L}{\partial u_x}
  ,
\end{equation}
and assuming that $u_t=u_t(v),$ $u_x=u_x(w)$ are invertible functions,
we can define the Hamiltonian density
\begin{equation} \nonumber
  \S{u,v,w}:=vu_t+wu_x-L(u,u_t(v),u_x(w))
  .
\end{equation}
We may then express the variational principle by means of $\S{u,v,w}$
in the new variables as
\begin{equation} \nonumber
  0
  =
  \delta \int_{[0,T] \times S^1}
  \bc{vu_t+wu_x-\S{u,v,w} }\,
  \di{x} \,\di{t}
  ,
\end{equation}
and through taking variations
$\bc{\psi_u,\psi_v,\psi_w}^T =: \vec{\psi}$ we obtain
\begin{equation}
  \label{eq:2}
  0
  =
  \int_{[0,T]\times S^1}
  \left( K\,\vec{z}_t+L\,\vec{z}_x-\nabla \S{\vec{z}}\right)
  \cdot \vec{\psi} \quad \di{x}\,\di{t},
  \qquad \forall \vec{\psi}
  ,
\end{equation}
where
\begin{equation} \nonumber
  \vec{z}:=\left[  
    \begin{array}{c}
      u\\
      v\\
      w
    \end{array}
  \right], \quad \vec{\psi}:=\left[  
    \begin{array}{c}
      \psi_{u}\\
      \psi_{v}\\
      \psi_{w}
    \end{array}
  \right], \qquad K=\left[
    \begin{array}{ccc}
      0 & -1 & 0\\
      1 & 0 & 0\\
      0 & 0& 0
    \end{array}
  \right], \qquad L=\left[
    \begin{array}{ccc}
      0 & 0 & -1\\
      0 & 0 & 0\\
      1 & 0& 0
    \end{array}
  \right]
  .
\end{equation}
We note that this integral yields an example of what we refer to as a
\emph{multisymplectic PDE} as discussed in \cite{Bridges:1997}, and
includes examples such as wave equations. Additionally, through
modifying the dependencies of the Lagrangian density function we may
obtain more examples of multisymplectic PDEs.  This underlying
variational nature of multisymplectic PDEs advocates space-time finite
elements as a desirable discretisation methodology.


\begin{definition}[Multisymplectic PDEs] \label{def:multisym}

  Let $t \in [0,T]$ for some fixed $T$ and let $x \in S^1:=[0,1)$ be
  periodic. Further let $K, L \in \real^{D \times D}$ be
  \emph{constant} skew-symmetric matrices for integer $D \ge 2$. Then
  a multisymplectic PDE is described by seeking
  $\vec{z}: (t,x) \to \real^D$ such that
  \begin{equation} \label{eqn:multisym}
    K \vec{z}_t + L \vec{z}_x  = \nabla \S{\vec{z}}
    ,
  \end{equation}
  where $\S{} : \real^D \to \real$. 

\end{definition}

\begin{remark}[Skew-symmetric inner product] \label{rem:skewinner}

  \revise{A skew-symmetric matrix induces a skew-symmetric product,
    that is to say that} if we have $\vec{u},\vec{v} \in \real^D$ then
  \revise{for any skew-symmetric $K \in \real^{D\times D}$}
  \begin{equation}
    \vec{u} \cdot K \vec{v} = - K \vec{u} \cdot \vec{v}
    ,
  \end{equation}
  \revise{where $\cdot$ denotes the Euclidean inner product.}
  
\end{remark}

The namesake of a multisymplectic PDE is its multisymplectic
structure, which falls within the form of a conservation law. To fully
describe this structure we utilise the Cartan derivative and the
skew-symmetric wedge product. For a detailed analysis of the
multisymplectic structure see \cite{BridgesReich:2001}.

\begin{theorem}[Multisymplectic structure] \label{thm:multisymstruct}

  The PDE \eqref{eqn:multisym} possesses the multisymplectic
  conservation law
  \begin{equation}
    \bc{ \d \vec{z} \wedge K \vec{z} }_t 
    + \bc{ \d \vec{z} \wedge L \vec{z} }_x
    =
    0
    .
  \end{equation}

\end{theorem}

\begin{proof}

  Applying the Cartan exterior derivative to \eqref{eqn:multisym} we find
  \begin{equation} \label{p:sym1}
    K \d \vec{z}_t + L \d \vec{z}_x 
    =
    \nabla^T \nabla \S{\vec{z}} \d \vec{z}
    .
  \end{equation}
  Acting on \eqref{p:sym1} with $\wedge \d \vec{z}$ from the right we
  can write 
  \begin{equation}
    K \d \vec{z}_t \wedge \d \vec{z} + L \d \vec{z}_x \wedge \d \vec{z}
    =
    \nabla^T \nabla \S{\vec{z}} \d \vec{z} \wedge \d \vec{z}
    .
  \end{equation}
  Through the skew-symmetry of the wedge product we find that
  \begin{equation}
    \frac12\bc{ K \d \vec{z} \wedge \vec{z}}_t
    =  K \d \vec{z}_t \wedge \d \vec{z}
    ,
  \end{equation}
  and similarly for the spatial derivative. Additionally, through the
  skew-symmetry of the wedge product and the symmetry of $\nabla^T
  \nabla \S{\vec{z}}$ we have that
  \begin{equation}
    \nabla^T \nabla \S{\vec{z}} \d \vec{z} \wedge \d \vec{z} = 0
    ,
  \end{equation}
  allowing us to conclude.
  
\end{proof}

\begin{remark}[The discrete preservation of the multisymplectic
  structure]

  The multisymplectic structure discussed in Theorem
  \ref{thm:multisymstruct} is possible to preserve on the discrete
  level, and leads to desirable numerical results, see
  \cite{BridgesReich:2001,McLachlanStern:2017, McDonald:2016,
    Cano:2006, Beirao:2017}.

  On a fundamental level, multisymplectic conservation laws treat time
  and space equivalently, which is natural for elliptic and some
  hyperbolic problems. However, for evolution problems this is not
  always ideal. Indeed, even boundary conditions are treated
  independently in space and time. In this contribution we shall
  discretise space and time non-homogeneously, in the sense that our
  discretisations in each variable are fundamentally
  different. Indeed, in the lowest order conforming case our proposed
  discretisation will be equivalent to discretising with the Average
  Vector Field method in time and conforming finite elements in space
  \cite{CelledoniEtAl:2012}. As such, we do not expect the methods
  proposed here to exactly preserve a discrete multisymplectic
  conservation law.
  
\end{remark}

Further to the multisymplectic conservation laws, all multisymplectic
PDEs also preserve momentum and energy conservation laws, which
correspond to spatial and temporal translations respectively via
Noether's theorem.

\begin{theorem}[Classical conservation laws in multisymplectic PDEs]

  Let $\vec{z}$ be the solution of the multisymplectic PDE
  \eqref{eqn:multisym} as set out in Definition \eqref{def:multisym}.
  Further define the momentum flux and density as
  \begin{align} \label{eqn:momentumstrong}
    \mflux = \frac12 \vec{z} \cdot K \vec{z}_t - \S{\vec{z}}
    \qquad \qquad 
    \mdensity = \frac12 \vec{z}_x \cdot K \vec{z}
  \end{align}
  respectively, and energy flux and density as
  \begin{align} \label{eqn:energystrong}
    \eflux = \frac12 \vec{z}_t \cdot L \vec{z}
    \qquad \qquad
    \edensity = \frac12 \vec{z} \cdot L \vec{z}_x - \S{\vec{z}}
  \end{align}
  respectively. Then the momentum conservation law
  \begin{equation} \label{eqn:mlaw}
    \mdensity_t + \mflux_x = 0
  \end{equation}
  holds. Additionally, $\vec{z}$ possesses the energy conservation law
  \begin{equation} \label{eqn:elaw}
    \edensity_t + \eflux_x = 0
    .
  \end{equation}

\end{theorem}

\begin{proof}
  
  The proofs of \eqref{eqn:mlaw} and \eqref{eqn:elaw} follow from the
  application of the PDE \eqref{eqn:multisym} and the constant
  skew-symmetric structures of the $K$ and $L$. As we utilise this
  methodology for the proof of momentum and energy conservation laws
  on the discrete level we shall present the proof of \eqref{eqn:mlaw}
  fully to inform the sequel. 
  
  Through \eqref{eqn:momentumstrong} we can write
  \begin{equation} \label{p:cclaw1}
    \mflux_x
    =
    \frac12 \vec{z}_x \cdot K \vec{z}_t
    + \frac12 \vec{z} \cdot K \vec{z}_{tx}
    - \nabla\S{\vec{z}} \cdot \vec{z}_x
    .
  \end{equation}
  Acting on the PDE \eqref{eqn:multisym} with $\cdot \vec{z}_x$ from
  the right we have
  \begin{equation} \label{p:cclaw2}
    \nabla\S{\vec{z}} \cdot \vec{z}_x
    =
    K \vec{z}_t \cdot \vec{z}_x + L \vec{z}_x \cdot \vec{z}_x 
    = 
    K \vec{z}_t \cdot \vec{z}_x
    ,
  \end{equation}
  through the skew-symmetry of $L$. Applying \eqref{p:cclaw2} to
  \eqref{p:cclaw1} we find
  \begin{equation} \label{p:cclaw3}
    \mflux_x
    =
    - \frac12 \vec{z}_x \cdot K \vec{z}_t
    + \frac12 \vec{z} \cdot K \vec{z}_{tx}
    .
  \end{equation}
  Through \eqref{eqn:momentumstrong} we can also write
  \begin{equation} \label{p:cclaw4}
    \begin{split}
      \mdensity_t
      & =
      \frac12 \vec{z}_{xt} \cdot K \vec{z} 
      + \frac12 \vec{z}_x \cdot K \vec{z}_t \\
      & = 
      -\frac12 \vec{z} \cdot K \vec{z}_{xt}
      + \frac12 \vec{z}_x \cdot K \vec{z}_t
      ,
    \end{split}
  \end{equation}
  through the skew-symmetry of $K$. Summing \eqref{p:cclaw3} and
  \eqref{p:cclaw4} we obtain the desired result for the momentum law
  \eqref{eqn:mlaw}. For the energy law \eqref{eqn:elaw} we follow the
  same methodology but instead act on \eqref{eqn:multisym} with
  $ \cdot \vec{z}_t$.

\end{proof}

\begin{example}[Nonlinear wave equations] \label{ex:nlw}

  Consider the nonlinear wave equation
  \begin{equation}
    u_{tt} - u_{xx} - \Vp{u} = 0
    .
  \end{equation}
  Through the introduction of the auxiliary variables
  \begin{equation}
    \begin{split}
      v & := u_t \\
      w & := u_x
      ,
    \end{split}
  \end{equation}
  we can write the nonlinear wave equation in multisymplectic form
  through defining $\vec{z} = (u,v,w)$, so $D=3$, with
  \begin{equation}
    \begin{split}
      K
      =
      \begin{pmatrix}
        0 & -1 & 0 \\
        1 & 0 & 0 \\
        0 & 0 & 0
      \end{pmatrix}
      & \qquad \qquad \qquad
      L =
      \begin{pmatrix}
        0 & 0 & 1 \\
        0 & 0 & 0 \\
        -1 & 0 & 0 
      \end{pmatrix}
      \\
      \S{\vec{z}} & = 
      \frac12 v^2 - \frac12 w^2 + \V{u}
      .
    \end{split}
  \end{equation}
  We additionally observe that the multisymplectic form of the
  momentum and energy density are compatible with the standard forms
  which are given by $\frac12 u_{xt} u - \frac12 u_x u_t$ and
  $\frac12 u u_{xx} - \frac12 u_t^2 - \V{u}$ respectively.
  Note that this multisymplectic PDE falls within the framework
  discussed at the beginning of this section, with a Lagrangian
  density of
  \begin{equation}
    L\bc{u,u_t,u_x}
    =
    \frac12 u_t^2 - \frac12 u_x^2 - \V{u}
    .
  \end{equation}
  
\end{example}


  

\begin{example}[A nonlinear {Schr\"odinger} equation] \label{ex:nls}

  Let us consider the nonlinear Schr\"odinger \emph{(NLS)} equation
  \begin{equation}
    i \xi_{t} + \xi_{xx} + \abs{\xi}^2 \xi = 0
    ,
  \end{equation}
  where $\xi(t,x)$ takes complex values. Through splitting into real
  and complex components, and introducing auxiliary variables for the
  spatial derivatives, we may rewrite the NLS equation in
  multisymplectic form. In particular, we define
  $\vec{z} = (u, v, p, q)$ through the introduction of the variables
  \begin{equation}
    \begin{split}
      u & = \Re{\xi} \\
      v & = \Im{\xi} \\
      p & = u_x \\
      q & = v_x
      ,
    \end{split}
  \end{equation}
  the multisymplectic formulation may then be expressed through the
  skew-symmetric operators
  \begin{equation}
    \begin{split}
      K
      =
      \begin{pmatrix}
        0 & -1 & 0 & 0 \\
        1 & 0 & 0 & 0 \\
        0 & 0 & 0 & 0 \\
        0 & 0 & 0 & 0
      \end{pmatrix}
      & \qquad \qquad \qquad
      L =
      \begin{pmatrix}
        0 & 0 & 1 & 0 \\
        0 & 0 & 0 & 1 \\
        -1 & 0 & 0 & 0 \\
        0 & -1 & 0 & 0
      \end{pmatrix}
      ,
    \end{split}
  \end{equation}
  and the symmetric nonlinear function
  \begin{equation}
    \S{\vec{z}}
    =
    - \frac14 \bc{u^2 + v^2}^2 - \frac12 \bc{p^2 + q^2}
    .
  \end{equation}
    
\end{example}

\section{Continuous space-time finite element
  approximations} \label{sec:cg}

We shall now focus on the design of a conforming space-time finite
element method which is constructed through the tensor product of
spatial and temporal finite element discretisations. This decoupling
of the spatial and temporal discretisations allows us to combine a
conservative spatial method with a temporally conservative method,
which for the study of Hamiltonian ODEs and Hamiltonian PDEs are
fundamentally different, see \cite{self:thesis}.

Before defining our finite element space we must give consideration to
the discretisation of the temporal and spatial domains. We partition
our temporal domain $\bs{0,T}$ such that
$0:=t_0 < t_1 < \cdots < t_N =:T$, and define a temporal finite
element as $\telement{n} := \left[t_n,t_{n+1}\right)$ which possesses an
element length $\dt{n} := t_{n+1} - t_n$. Similarly, we partition our
periodic spatial domain $S^1$ such that
$0:=x_0 < x_1 < \cdots < x_M := 1$ and define a spatial finite element
as $\selement{m} := \bc{x_m,x_{m+1}}$ which possesses an element
length $\dx{m} := x_{m+1}-x_m$.

To define the spatial and temporal finite element spaces
simultaneously we consider an arbitrary domain $\Omega \subset \real$
which is partitioned such that $y_0 < y_1 < \cdots < y_J$ possessing
an arbitrary element $\aelement{j} := \bc{y_{j},y_{j+1}}$. 

\begin{definition}[Finite element spaces] \label{def:fes}

  The concise definition of a given finite element space depends on
  the discretisation of the domain being considered. Here we consider
  the discretisation of $\Omega$ by the elements $\aelement{j}$.

  Let $\bpoly{r}\bc{\aelement{j}}$ denote the space of polynomials of
  degree $r$ on an interval $\aelement{j} \subset \mathbb{R}$, and
  $\left( \bpoly{r} \bc{\aelement{j}} \right)^{D}$ be the
  $D$-dimensional vector extension of the aforementioned space, then
  the discontinuous finite element space is
  \begin{equation} \label{def:dpoly}
    \dpoly{r}\bc{\Omega}
    =
    \{ \vec{W} : \left. \vec{W} \right|_{\aelement{j}} 
    \in \left( \bpoly{r} \bc{\aelement{j}} \right)^{D}, j=0,...,J-1 \},
  \end{equation}
  further to this the continuous finite element space is defined
  analogously with global continuity enforced, i.e.,
  \begin{equation} \label{def:cpoly}
    \cpoly{r} (\Omega)
    =
    \dpoly{r} (\Omega) \cap \bc{\mathcal{C}^0\bc{\Omega}}^D
    .
  \end{equation}

\end{definition}

\begin{remark}[Space-time finite element
  space] \label{rem:spacetimespace} 

  In this work we define our space-time finite element spaces through
  the tensor product of one dimensional \revise{temporal and spatial}
  finite element spaces, allowing us, in some sense, to use
  \revise{temporal or spatial} finite element techniques where
  required. We shall seek a discrete solution in the trial space
  \begin{equation}
    \cpoly{q+1}\bc{\bs{0,T}} \times \cpoly{p}\bc{S^1}
  \end{equation}
  which we test against all functions in
  \begin{equation}
    \dpoly{q}\bc{\bs{0,T}} \times \cpoly{p}\bc{S^1}
    .
  \end{equation}
  Note that spatially we choose our trial and test spaces to be the
  same, however, temporally we choose our trial space such that if we
  differentiate in time we enter the test space. This is typical for
  conservative temporal finite element approximations, see
  \cite{TangSun:2012, BetchSteinmann:2000, Hansbo:2001,
    EstepFrench:1994, self:invariant, self:thesis}. \revise{Both the
    spatial and temporal spaces for the trial functions
    are conforming, in the sense that they are finite dimensional
    subspaces of the solution space on the continuous level.}

\end{remark}

We may define our space-time finite element
approximation of multisymplectic PDEs as follows.

\begin{definition}[Continuous space-time finite element method] \label{def:stfem}

  Seek $\vec{Z} \in \cpoly{q+1}\bc{\bs{0,T}} \times \cpoly{p}\bc{S^1}$ such that
  \begin{equation} \label{eqn:stfem}
    \begin{split}
      \stint 
      K\vec{Z}_t \cdot \phi + L\vec{Z}_x \cdot \phi
      - \nabla \S{\vec{Z}} \cdot \phi
      \di{x} \di{t} 
      & = 0 
      \qquad
      \forall \phi \in \dpoly{q}\bc{\bs{0,T}} \times \cpoly{p}\bc{S^1}
      \\
      \vec{Z}(0,x) & = \icproject{ \vec{z}(0,x) }
      ,
    \end{split}
  \end{equation}
  where $\icproject{}$ represents the $L_2$ projection into
  $\cpoly{p}(S^1)$.

\end{definition}

\begin{remark}[Localisation of the finite element
  method] \label{rem:localisation}

  While we have formulated \eqref{eqn:stfem} globally, it is crucial
  that we can write the formulation locally in \emph{time} due to the
  evolutionary nature of the problem. By design, our test function
  $\phi$ is discontinuous which allows us to rewrite the method in a
  time stepping fashion as seeking
  $\vec{Z} \in \cpoly{q+1}\bc{\telement{n}} \times \dpoly{p}\bc{S^1}$
  for $n = 0,...,N-1$ such that
  \begin{equation} \label{eqn:stlfem}
    \begin{split}
      \stlint 
      K\vec{Z}_t \cdot \psi + L\vec{Z}_x \cdot \psi
      - \nabla \S{\vec{Z}} \cdot \psi
      \di{x} \di{t} 
      & = 0 
      \qquad
      \forall \psi \in \bpoly{q}\bc{\telement{n}} \times \dpoly{p}\bc{S^1}
      ,
    \end{split}
  \end{equation}
  where $\vec{Z}(t_n,x)$ is enforced through either global continuity,
  or the initial condition $\vec{Z}(0,x) = \icproject{ \vec{z}(0,x) }$. This
  local formulation can be implemented in a time stepping fashion.

  Further to localising time we may localise the formulation over
  space. However, as the test functions are spatially continuous we
  may not express this localisation concisely, without the
  incorporation of hybridisable dG techniques or through examining the
  underlying nonlinear system of equations for which we solve.

\end{remark}

\begin{remark}[A spatially discontinuous finite element
  approximation] \label{rem:sdg}

  To clarify exposition here we have assumed a spatially continuous
  finite element approximation. In fact, through an appropriate
  discretisation of the spatial derivative we may construct a
  discontinuous geometric space-time finite element discretisation. A
  discontinuous discretisation has two significant benefits: The first
  is the resulting method is more receptive to adaptivity due to the
  spatially discontinuous nature of the solution, although
  complications do arise from the temporal continuity. The second is
  that the discontinuous nature allows us to employ local
  discontinuous Galerkin techniques \cite{CockburnShu:1998} to reduce
  the dimension of the system within the implementation of the method
  through rewriting the method as a single equation in terms of the
  original variable. This is often referred to as the \emph{primal
    form} \cite{Castillo:2006}. In practice, this is standard for the
  implementation of multisymplectic PDEs, such as box schemes
  \cite{BridgesReich:2001, AscherMcLachlan:2004}. We note that due to
  the conforming nature of the finite element method in time, we still
  cannot reduce components of the system with temporal derivatives
  into primal form. The development of novel tools is required to
  design a fully discontinuous approximation and falls beyond the
  scope of this work, although progress has been made to these ends
  \cite[Chapter 3]{self:thesis}. We shall investigate a spatially
  discontinuous approximation in detail in Section \ref{sec:dg}.
  
\end{remark}

We shall now investigate the discrete conservation laws which are
preserved by continuous space-time finite element approximation.

\begin{theorem}[Discrete conservation laws] \label{thm:claws}

  Let the conditions of Definition \eqref{def:stfem} hold, and in
  particular let $\vec{Z}$ be the solution of the space-time finite
  element method \eqref{eqn:stfem}. Additionally define the scalar function
  \begin{equation} \label{eqn:w}
    W :=
    \nabla \S{\vec{Z}} \cdot \project{\vec{Z}_x}
    ,
  \end{equation}
  where $\project{}$ is the $L_2$ projection into the finite element
  space $\dpoly{q}\bc{\telement{n}} \times \cpoly{p}\bc{S^1}$.
  Further define the fluxes $\dmflux,\deflux$ and densities
  $\dmdensity, \dedensity$ such that
  \begin{align} \label{eqn:fluxdensity}
    \dmflux = \frac12 \vec{Z}_t \cdot K\vec{Z} - \S{\vec{Z}}
    \quad & \quad
      \dmdensity = \frac12 \vec{Z}_x \cdot K\vec{Z}
    \\
    \deflux = \frac12 \vec{Z}_t \cdot L \vec{Z}
    \quad & \quad
            \dedensity = \frac12 \vec{Z}_x \cdot L \vec{Z}
            -
    \S{\vec{Z}} 
            .
  \end{align}
  The numerical solution $\vec{Z}$ conserves a consistent discrete momentum
  locally in time, namely
  \begin{equation} \label{eqn:mclaw}
    \stlint {\dmdensity}_t + \bc{\dmflux + \S{\vec{Z}}}_x - W \di{x} \di{t} = 0
    .
  \end{equation}
  Additionally, a discrete energy is conserved locally in time,
  which we can write as
  \begin{equation} \label{eqn:eclaw}
    \stlint \dedensity_t + \deflux_x \di{x} \di{t} = 0
    .
  \end{equation}
 
\end{theorem}

\begin{remark}[Initial remarks on Theorem \ref{thm:claws}]

  $W$ as constructed in \eqref{eqn:w} is a discrete
  approximation of $\S{\vec{Z}}$, which is designed to allow us to
  apply the finite element approximation to the spatial momentum flux
  contribution in a way which mimics the continuous arguments. This
  construction is required as $\vec{Z}_x$ is not in the test space due
  to the Petrov-Galerkin nature of the formulation in time and the
  global continuity of the spatial finite element space.

  As our method is constructed under spatial and temporal integrals,
  the natural discrete conservation laws must also be defined under
  the integral. This weakens the notion of a conservation law. This
  may be strengthened slightly when considering a spatially
  discontinuous approximation by explicitly localising these laws to
  single elements in space-time, as will be seen in Section
  \ref{sec:dg}.

\end{remark}

\begin{remark}[The consistent conserved momentum described in Theorem
  \ref{thm:claws}]

  We note that the consistent discrete momentum conservation
  \eqref{eqn:mclaw} law does not fall concisely within the expected
  format for a conservation law. The lack of conservation here is
  caused by the typically nonlinear term $\nabla\S{\vec{Z}}$ and the
  fact that $Z_x$ does not live within the finite element space
  $\dpoly{q}\bc{\telement{n}} \times \cpoly{p}\bc{S^1}$. On a case by
  case basis, the conservation of momentum may be examined for
  particular examples. We shall see that in situations where schemes
  have been constructed in the literature which preserve both
  invariants, we also may expect conservation of both momentum and
  energy, for example for the linear wave equation, we do indeed
  obtain it. For brevity we shall not present proofs on a case by case
  basis here instead presenting select numerical simulations.

  In cases where we do not obtain a momentum conservation law through
  alternate means, we note that the deviation in momentum may be
  quantified locally by
  \begin{equation}
    \stlint \S{\vec{Z}}_x - \nabla\S{\vec{Z}} \cdot
    \project{\vec{Z}_x}
    \di{x} \di{t}
    ,
  \end{equation}
  which may be written equivalently as
  \begin{equation}
    \stlint \nabla\S{\vec{Z}} \cdot \bc{
      \vec{Z}_x - \project{\vec{Z}_x}}
    \di{x}\di{t}
    ,
  \end{equation}
  so the deviation in momentum is quantified by the difference between
  $\vec{Z}_x$ and its projection into the finite element space. This
  may be thought of as a quantification of how discontinuous the
  spatial derivative of $\vec{Z}$ is. Experimentally, when simulating
  smooth solutions which are not highly oscillatory we expect this
  deviation to be small.
  
\end{remark}

\begin{proof}[Proof of Theorem \ref{thm:claws}]
  
  Through the definition of the momentum density $\dmdensity$ given in
  \eqref{eqn:fluxdensity} we can write
  \begin{equation} \label{p:mclaw0}
    \begin{split}
      \stlint {\dmdensity}_t \di{x}\di{t}
      & =
      \stlint  \frac12 \vec{Z}_{xt} \cdot K \vec{Z} 
      \ \revise{+} \ \frac12 \vec{Z}_x \cdot K \vec{Z}_t \
      \revise{\di{x}\di{t}} \\
      & \ \revise{=
        \stlint \vec{Z}_x \cdot K \vec{Z}_t \di{x}\di{t}  
        ,
      }
    \end{split}
  \end{equation}
  \revise{in view of the skew-symmetric inner product induced by
    $K$ and integration by parts. }
  Further, through the definition of the momentum flux $\dmflux$ in
  \eqref{eqn:fluxdensity} we find
  \begin{equation} \label{p:mclaw1}
    \begin{split}
      \stlint \bc{{\dmflux} + \S{\vec{Z}}}_x - W \di{x} \di{t}
      & =
      \stlint
      - \nabla \S{\vec{Z}} \cdot \project{\vec{Z}_x}
      \di{x} \di{t}
    \end{split}
  \end{equation}
  \revise{in view of the periodic spatial boundary.} We observe
  that $\vec{Z}_x$ is not in the test space. With this in mind we
  choose $\psi = \project{\vec{Z}_x}$ in the temporally local scheme
  \eqref{eqn:stlfem} allowing us to write
  \begin{equation} \label{p:mclaw2}
    \begin{split}
      0 & =
      \stlint
      K \vec{Z}_t \cdot \project{\vec{Z}_x} 
      + L \vec{Z}_x \cdot \project{\vec{Z}_x}
      - \nabla \S{\vec{Z}} \cdot \project{\vec{Z}_x} \di{x} \di{t}
      \\
      & = 
      \stlint
      K \vec{Z}_t \cdot \vec{Z}_x
      + L \project{\vec{Z}_x} \cdot \project{\vec{Z}_x}
      - \nabla \S{\vec{Z}} \cdot \project{\vec{Z}_x} \di{x} \di{t}
      \\
      & = 
      \stlint
      K \vec{Z}_t \cdot \vec{Z}_x
      - \nabla \S{\vec{Z}} \cdot \project{\vec{Z}_x} \di{x} \di{t}
      ,
    \end{split}
  \end{equation}
  through the definition of the $L_2$ projection into the finite
  element space $\dpoly{q}\bc{\telement{n}} \times \cpoly{p}\bc{S^1}$
  and the skew symmetry of the matrix $L$. \revise{ Indeed, that
    $K\vec{Z}_t \in \dpoly{q}\bc{\telement{n}} \times
    \cpoly{p}\bc{S^1}$ is crucial to reach second line of
    \eqref{p:mclaw2} as it allows us to exploit the definition of the
    $L_2$ projection and is fundamental to the design of our finite
    element discretisation.\footnote{\revise{In the second line of
        \eqref{p:mclaw2} we have used that
        $K\vec{Z}_t, \project{\vec{Z}_x} \in
        \dpoly{q}\bc{\telement{n}} \times \cpoly{p}\bc{S^1}$ to choose
      each as a test function in the definition of the $L_2$
      projection, i.e.,
      \begin{equation}
        \stlint \bc{Z_x - \project{Z_x}} \cdot K\vec{Z}_t \di{x}\di{t}
        =
        0
        \qquad
        \text{and}
        \qquad
        \stlint \bc{L\vec{Z}_x - L\project{\vec{Z}_x}} \cdot
        \project{\vec{Z}_x}
        \di{x}\di{t}
        =
        0
        .
      \end{equation}
    }}}
Applying \eqref{p:mclaw2} to
  \eqref{p:mclaw1} we have
  \begin{equation} \label{p:mclaw3}
    \begin{split}
      \stlint \bc{{\dmflux} + \S{\vec{Z}}}_x - W \di{x} \di{t}
      & =
      \ \revise{\stlint - K\vec{Z}_t \cdot \vec{Z}_x \di{x}\di{t}
        }
      .
    \end{split}
  \end{equation}
  Summing \eqref{p:mclaw3} with \eqref{p:mclaw0} we obtain the
  momentum conservation law \eqref{eqn:mclaw}.

  Turning our attention to the proof of local energy conservation, we
  apply the definitions of the energy density
  \eqref{eqn:fluxdensity} finding
  \begin{equation}
    \stlint \dedensity_t \di{x}\di{t}
    =
    \stlint \frac12 \vec{Z}_{xt} \cdot L \vec{Z}
    + \frac12 \vec{Z}_x\cdot L \vec{Z}_t
    - \vec{Z}_t \cdot \nabla \S{\vec{Z}} \di{x}\di{t}
    .
  \end{equation}
  Through choosing $\psi = \vec{Z}_t$ in the temporally local
  scheme \eqref{eqn:stlfem} we can write
  \begin{equation} \label{p:leclaw1}
    \stlint \dedensity_t \di{x}\di{t}
    =
    \stlint \frac12 \vec{Z}_{xt} \cdot L \vec{Z}
    + \frac12 \vec{Z}_x \cdot L \vec{Z}_t
    + K\vec{Z}_t \cdot \vec{Z}_t
    - L\vec{Z}_x \cdot \vec{Z}_t
    \di{x}\di{t}
    .
  \end{equation}
  Through the skew-symmetric inner products induced by $K$ and $L$ we
  can reduce \eqref{p:leclaw1} to
  \begin{equation} \label{p:leclaw2}
    \begin{split}
      \stlint \dedensity_t \di{x}\di{t}
      & =
      0
      .
    \end{split}
  \end{equation}
  Additionally, through the definition of the flux $\deflux$ we observe
  that
  \begin{equation} \label{p:leclaw3}
    \begin{split}
      \stlint \deflux_x \di{x}\di{t}
      & = 
      0
    \end{split}
  \end{equation}
  utilising the fundamental theorem of calculus and the periodic
  boundary conditions. \revise{As, under the spatial integral, the
    energy density and flux are both zero, clearly their sum will be
    zero allowing us to conclude. Note that while for the energy law
    there was no interplay between the density and flux terms of the
    energy in the proof this was not the case for the momentum
    conservation law.}
  
\end{proof}

\begin{remark}[Theorem \ref{thm:claws} in the semi-discrete
  setting] \label{rem:semiclaws}

  While Theorem \ref{thm:claws} is posed in the fully discrete
  setting, the proofs of similar results in either the spatially or
  temporally semi-discrete setting may be obtained through following
  similar arguments. Indeed, if we allow time to be continuous we
  obtain the conservation laws
  \begin{equation} \label{eqn:clawsspace}
    \begin{split}
    \slint {\dmdensity}_t + \bc{\dmflux + \S{\vec{Z}}}_x - W \di{x}
    & = 0 \\
    \slint \dedensity_t + \deflux_x \di{x} & = 0
    ,
    \end{split}
  \end{equation}
  as a point-wise result in time where $\vec{Z}$ is discrete in space
  and continuous in time. Furthermore, if we allow space to be
  continuous and time to be discrete we obtain the temporally local
  conservation laws
  \begin{equation} \label{eqn:clawstime}
    \begin{split}
      \tint {\dmdensity}_t + \dmflux_x \di{t} & = 0 \\
      \tint \dedensity_t + \deflux_x  \di{t} & = 0
      ,
    \end{split}
  \end{equation}
  which are point-wise in space. Note that if we are spatially
  continuous we may \emph{exactly} preserve both the momentum and
  energy conservation laws.

  While the temporal and spatial conservation laws here are deeply
  intertwined, we note that, similarly to \cite{McLachlanQuispel:2014},
  if we were to couple our temporal discretisation with an
  alternative spatial discretisation which preserves conservation laws
  in the semi-discrete setting we would obtain an alternate
  conservative discretisation in the fully discrete setting.

\end{remark}

\begin{remark}[An alternate scheme with an exact momentum conservation
  law for the nonlinear wave equation] \label{rem:mclaw}

  For the nonlinear wave equation described in Example \ref{ex:nlw},
  perturbing the finite element scheme described in Definition
  \ref{def:stfem} preserves a natural discrete momentum. In
  particular, this modified scheme is given by seeking
  $\vec{Z} \in \cpoly{q+1}\bc{\bs{0,T}} \times \dpoly{p}\bc{S^1}$ such
  that
  \begin{equation} \label{eqn:mstfem}
    \begin{split}
      \stint 
      K\vec{Z}_t \cdot \phi + L\vec{Z}_x \cdot \phi
      - \cproject{\nabla \S{\vec{Z}}} \cdot \phi
      \di{x} \di{t} 
      & = 0 
      \qquad
      \forall \phi \in \dpoly{q}\bc{\bs{0,T}} \times \dpoly{p}\bc{S^1}
      \\
      \vec{Z}(0,x) & = \icproject{ \vec{z}(0,x) }
      ,
    \end{split}
  \end{equation}
  where $\icproject{}$ represents the $L_2$ projection into
  $\dpoly{p}(S^1)$ and $\cproject{}$
  the $L_2$ projection into $\cpoly{q+1}{\bc{[0,T]}} \times
  \dpoly{q}\bc{S^1}$. This discretisation has been modified through
  projecting the non-linearity of the PDE into the finite element trial
  space. Such a discretisation may be viewed as the introduction of
  additional auxiliary variables which ensure that $\nabla\S{\vec{Z}}$
  is within the finite element space. 

  For the nonlinear wave equation, this discretisation preserves the
  momentum conservation law
  \begin{equation} \label{eqn:mmclaw}
    \stlint \dmdensity_t + \dmflux_x \di{x}\di{t}
    =
    0
    ,
  \end{equation}
  where $\dmdensity,\dmflux$ are given by \eqref{eqn:fluxdensity}. The
  energy law is no longer conserved when $\nabla\S{\vec{Z}}$ is
  nonlinear, however, the consistent energy law
  \begin{equation}
    \stlint \bc{\dedensity - \S{\vec{Z}}}_t + \cproject{\nabla\S{\vec{Z}}}
    Z_t + \deflux_x \di{x} \di{t} = 0
    ,
    \end{equation}
    where $\cproject{}$ is the $L_2$ projection into
    $\cpoly{q+1}\bc{\telement{n}} \times \dpoly{p}\bc{S^1}$, is
    preserved.

\end{remark}

\subsection{Existence and uniqueness} \label{sec:cg:eu}

Due to skew-symmetric nature of the multisymplectic PDE
\eqref{eqn:multisym} and the possible sparsity of the operators $K,L$
it is not possible to prove existence and uniqueness on the continuous
level for a generic problem of this type. Although, in practice for
physically motivated examples such results are expected, as a physical
PDE typically has a dense non-skew-symmetric formulation, for example
the wave equation. With this in mind, when considering the existence
and uniqueness of the approximation \eqref{eqn:stfem} we shall focus
on the wave equation described in Example \ref{ex:nlw} under the
assumption that $\V{u} \ge 0$ for all $u$. We note that this includes
the linear wave equation, i.e., when $\V{u}=0$.

As our problem is fully discrete, it is possible to prove existence
and uniqueness of the underlying nonlinear system of equations, and
this is typically the approach in the literature, see for example
\cite{KarakashianMakridakis:1999}. For brevity, we shall instead
utilise the temporally conforming nature of the space-time
discretisation to show existence of solutions through the
Picard-Lindel\"of theorem. Throughout we shall consider existence and
uniqueness \emph{only} over the arbitrary temporal interval
$[t_n,t_{n+1})$, which may be naturally extended globally through the
global continuity of the discrete solution.

For the wave equation, we may explicitly write our finite element
discretisation over an arbitrary temporal element as seeking
$U,V,W \in \cpoly{q+1}\bc{\telement{n}} \times \cpoly{p}\bc{S^1}$ such
that
\begin{equation} \label{eqn:stwave}
  \begin{split}
    \stlint
    V_t\phi - W_x\phi + \V{U}\phi \di{x}\di{t}& = 0 \\
    \stlint
    V \psi - U_t \psi \di{x}\di{t} & = 0 \\
    \stlint
    W \chi - U_x \chi \di{x}\di{t} & = 0
    ,
  \end{split}
\end{equation}
$\forall \phi,\psi,\chi \in \dpoly{q}\bc{\telement{n}} \times
\cpoly{p}\bc{S^1}$ where $U(t_n),V(t_n),W(t_n)$ is enforced by either
temporal continuity or the initial data.

Our space-time finite element discretisation \eqref{eqn:stwave}
exactly preserves the energy density under a temporally local
integral, i.e.,
\begin{equation}
  \begin{split}
    0 & = \stlint {\dedensity}_t \di{x}\di{t} = \stlint \bc{ \frac12 \vec{Z} \cdot L \vec{Z}_x -
      \S{\vec{Z}} }_t \di{x} \di{t}
    =
    \stlint \bc{ \frac12 W_x U - \frac12 U_xW - \frac12 V^2 + \frac12 W^2 - \V{U} }_t
    \di{x} \di{t} \\
    & =
    \stlint \bc{ - \frac12 V^2 - \frac 12 \sproject{U_x}^2 - \V{U}
    }_t \di{x} \di{t}
    ,
  \end{split}
\end{equation}
where $\sproject{}$ is the $L_2$ projection into the \emph{spatial}
finite element space, as
\begin{equation}
  \stlint \bc{W-U_x} W_t \di{x}\di{t} = 0
  \qquad
  \textrm{and}
  \qquad
  \stlint \bc{W-U_x}\sproject{U_{x}}_t\di{x}\di{t}=0
  ,
\end{equation}
which follow from \eqref{eqn:stwave}. Crucially, we observe that
conservation of this energy over time yields stability of the
discretisation over time, i.e., that assuming $U_x, \V{U}, V$ are
finite at $t=0$ we have that
\begin{equation} \label{eqn:cg:bounded}
  \begin{split}
    \norm{L_2(S^1)}{V(t_n)}^2 \le C \\
    \norm{L_2(S^1)}{\sproject{U_x(t_n)}}^2 \le C \\
    \norm{\V{S^1}}{U(t_n)} \le C
    ,
  \end{split}
\end{equation}
where
\begin{equation}
  C:= \norm{L_2(S^1)}{V(0)}^2 + \norm{L_2(S^1)}{U_x(0)}^2
  + \norm{\V{S^1}}{U(0)}
  .
\end{equation}
Here $\norm{\V{S^1}}{\cdot}$ represents the norm induced by
$\V{\cdot}$, and in particular
\begin{equation}
  \norm{\V{S^1}}{U(t)} := \int_{S^1} \V{U(t)} \di{x}
  .
\end{equation}
Note that if $\V{U}=0$ then the third inequality of
\eqref{eqn:cg:bounded} does not provide us with any information,
however we may still obtain stability of $U$ through inverse estimates
\cite{Thomee:2006}. Furthermore, spatially the auxiliary variable $W$
is equivalent to the $L_2$ projection of $U_x$, which allows us to
reduce \eqref{eqn:stwave} to write
\begin{equation} \label{eqn:stwavered}
  \begin{split}
    \stlint
    V_t\phi - \sproject{U_x}_x\phi + \V{U}\phi \di{x}\di{t}& = 0 \\
    \stlint
    V \psi - U_t \psi \di{x}\di{t} & = 0
    ,
  \end{split}
\end{equation}
where $\sproject{}$ is the $L_2$ projection into the spatial
finite element space. Indeed, through applying the definition of the
$L_2$ projection to all terms of
\eqref{eqn:stwavered} which are not within the test space we
may interpret this space-time
approximation as the point-wise finite dimensional ODE system
\begin{equation} \label{p:eucg1}
  \begin{split}
    U_t & = \project{V} \\
    V_t & = \project{\sproject{U_x}_x - \V{U}}
    ,
  \end{split}
\end{equation}
where $\project{}$ represents the $L_2$ projection into the space-time
finite element space
$\dpoly{q}\bc{\telement{n}}\times\cpoly{p}(S^1)$. Through inverse
estimates and the stability of the $L_2$ projection we see that the
right hand side of our ODE is continuous in $U$ and $V$, and through
\eqref{eqn:cg:bounded} the solution remains in a bounded set depending
on the initial data. Additionally, the Jacobian of the right hand side
is a uniformly bounded operator, thus through invoking the
Picard-Lindel\"of theorem we yield a unique solution over $\telement{n}$ which
may be extended globally.

\section{Spatially  discontinuous finite element approximation} \label{sec:dg}
\renewcommand{\gfunc}[1]{%
  \ifthenelse{\isempty{#1}}{\ensuremath{\mathcal{G}}}%
  {\ensuremath{\mathcal{G}\left(#1\right)}}}

We shall now focus our attention on a spatially discontinuous
approximation. Such an approximation allows us to concisely localise
our conservation laws spatially, in addition to the benefits discussed
in Remark \ref{rem:sdg}.

As we consider \emph{discontinuous} function spaces we require
the following additional definitions to concisely describe the
method about points with multiple values.

\begin{definition}[Jumps and averages] \label{def:jumpavg}

  Due to the discontinuous nature of the finite element space finite
  element functions are permitted to be multi-valued at the vertices of
  the elements. With this in mind we introduce notation
  \begin{equation} \label{eqn:dg:xlim}
    \vec{U}_m^+ := \vec{U}(x_m^+) := \lim_{x \searrow x_m} \vec{U}(x),
    \qquad
    \vec{U}_m^- := \vec{U}(x_m^-) := \lim_{x \nearrow x_m} \vec{U}(x),
  \end{equation}
  to describe the values of the function on the right and left of the
  discontinuity \revise{at $x_m$} respectively. We further define the
  jump of a function at the point $x_m$ to be
  \begin{equation} \label{eqn:dg:jump}
    \jump{\vec{U}_m} = \vec{U}_m^- - \vec{U}_m^+
  \end{equation} 
  and the average as
  \begin{equation} \label{eqn:dg:avg}
    \avg{\vec{U}_m} = \frac12 \left( \vec{U}_m^- + \vec{U}_m^+ \right).
  \end{equation}
  
\end{definition}

\begin{definition}[Discrete operator for first spatial
  derivatives] \label{def:gfunc}

  Let $\vec{U} \in \dpoly{p}\bc{S^1}$, then $\gfunc{} : \dpoly{p}\bc{S^1} \to \dpoly{p}\bc{S^1}$
  such that
  \begin{equation} \label{eqn:dg:gfunc}
    \int_{S^1} \gfunc{\vec{U}} \cdot \phi  \di{x}
    =
    \sum_{m=0}^{M-1} \int_{\selement{m}} \vec{U}_x \cdot \phi  \di{x}
    - \sum_{m=0}^{M-1} \jump{\vec{U}_m} \cdot \avg{\phi_m}
    \qquad
    \forall \phi \in \dpoly{p}\bc{S^1}
    .
  \end{equation}

\end{definition}

With this definition in mind, we define the spatially discontinuous
space-time finite element approximation of multisymplectic PDEs as
follows.

\begin{definition}[Spatially discontinuous space-time finite element
  method] \label{def:stfemdg}

  Seek $\vec{Z} \in \cpoly{q+1}\bc{\bs{0,T}} \times \dpoly{p}\bc{S^1}$ such that
  \begin{equation} \label{eqn:stfemdg}
    \begin{split}
      \stint 
      K\vec{Z}_t \cdot \phi + L\gfunc{\vec{Z}} \cdot \phi
      - \nabla \S{\vec{Z}} \cdot \phi
      \di{x} \di{t} 
      & = 0 
      \qquad
      \forall \phi \in \dpoly{q}\bc{\bs{0,T}} \times \dpoly{p}\bc{S^1}
      \\
      \vec{Z}(0,x) & = \icproject{ \vec{z}(0,x) }
      ,
    \end{split}
  \end{equation}
  where $\icproject{}$ represents the $L_2$ projection into
  $\dpoly{p}(S^1)$.
\end{definition}

\begin{remark}[Localisation of the spatially discontinuous approximation]

  Due to the local nature of the test functions in both space and
  time, the finite element method proposed in Definition
  \ref{def:stfemdg} may be written locally as
    \begin{equation} \label{eqn:dg:lsltfem}
    \begin{split}
      \sltlint 
      K\vec{Z}_t \cdot \chi + L\gfunc{\vec{Z}} \cdot \chi
      - \nabla \S{\vec{Z}} \cdot \chi
      \di{x} \di{t} 
      & = 0 
      \qquad
      \forall \chi \in
      \dpoly{q}\bc{\telement{n}} \times \dpoly{p}\bc{\selement{m}}
      \\
      \vec{Z}(0,x) & = \icproject{ \vec{z}(0,x) }
      .
    \end{split}
  \end{equation}
  Note that, while this method may be formulated locally in space, it
  must be implemented globally due to the global nature of the
  operator $\gfunc{}$. For readers more familiar with finite
  difference methods, this localisation may be thought of similarly to
  defining a finite difference method over its stencil. That is to say
  that the method may be uniquely defined over its stencil but must be
  implemented globally.

\end{remark}

\begin{remark}[Connection between the spatially discontinuous
  approximation and local discontinuous Galerkin
  methods] \label{rem:ldg}

  Readers well versed in local discontinuous Galerkin \emph{(ldG)}
  methods, \cite{CockburnShu:1998, CastilloEtAl:2000, XuShu:2005}, may
  spot similarities with our spatially discontinuous approximation
  \eqref{def:stfemdg}. Indeed, multisymplectic PDEs may often be
  viewed the reformulation of a PDE as a first order system through
  introducing first derivatives as auxiliary variables, for example in
  Example \ref{ex:nlw}. However, this is not always the
  case. Additionally, ldG methods are formulated in terms of
  \emph{local} spatial derivatives with either upwind or downwind
  fluxes. Here our spatial derivative is a global operator, as when
  restricted to a single element it depends on both upwind and
  downwind information. Here we do obtain one of the core benefits of
  ldG methods: For auxiliary variables of first spatial derivatives,
  we may reduce the computational complexity of our implementations
  through rewriting them in primal form as discussed in Remark
  \ref{rem:sdg}.

\end{remark}

Before examining the conservation laws preserved by the
multisymplectic form we must understand the properties of the discrete
spatial operator $\gfunc{}$, as these will be key in investigating
discrete local conservation laws.

\begin{proposition}[Properties of $\gfunc{}$] \label{prop:gfunc}
  
  Let $\vec{U},\vec{V} \in \dpoly{p}\bc{S^1}$, and $\gfunc{}$ be as given in
  Definition \ref{def:gfunc}, then $\gfunc{}$ is orthogonal to
  constants in the sense that
  \begin{equation} \label{eqn:dg:gorth}
    \int_{S^1} \gfunc{\vec{U}} \cdot \vec{1} \di{x} = 0 
    .
  \end{equation}
  Additionally it respects the skew-symmetric identity
  \begin{equation} \label{eqn:dg:gskew}
    \int_{S^1} \gfunc{\vec{U}} \cdot \vec{V} \di{x}
    =
    - \int_{S^1} \vec{U} \cdot \gfunc{\vec{V}} \di{x}
    ,
  \end{equation}
  which corresponds to integration by parts. Finally, $\gfunc{}$ also
  respects the product rule
  \begin{equation} \label{eqn:dg:gproduct}
    \int_{S^1} \gfunc{\vec{U} \cdot \vec{V}} \di{x} 
    =
    \int_{S^1} \gfunc{\vec{U}} \cdot \vec{V}
    + \vec{U} \cdot \gfunc{\vec{V}} \di{x}
    .
  \end{equation}
    Locally, we can write these three properties as
    \begin{equation} \label{eqn:dg:lgorth}
      \int_{\selement{m}} \gfunc{\vec{U}} \cdot \vec{1} \di{x}
      \revise{\
        = a_m\bc{\vec{U} \cdot \vec{1}}
        }
      ,
    \end{equation}
  \begin{equation} \label{eqn:dg:lgskew}
    \begin{split}
      \int_{\selement{m}} \gfunc{\vec{U}} \cdot \vec{V} \di{x}
      &\revise{ =
        -\int_{\selement{m}} \vec{U} \cdot \gfunc{\vec{V}} \di{x}
        + b_m\bc{\vec{U},\vec{V}}
      }
    \end{split}
  \end{equation}
  and
  \revise{
    \begin{equation} \label{eqn:dg:lgprod}
      \begin{split}
        \int_{\selement{m}} \gfunc{\vec{U} \cdot \vec{V}} \di{x}
        - a_m\bc{\vec{U} \cdot \vec{V}}
        & =
        \int_{\selement{m}} \gfunc{\vec{U}} \cdot \vec{V} 
        + \vec{U} \cdot \gfunc{\vec{V}}
        \di{x} - b_m\bc{\vec{U},\vec{V}}
        ,
      \end{split}
    \end{equation}
    where we have defined
    \begin{equation} \label{eqn:a}
      a_m\bc{\vec{U} \cdot \vec{V}}
      :=
      \frac12 \avg{\vec{U}_{m+1} \cdot \vec{V}_{m+1}} 
      - \frac12 \avg{\vec{U}_m  \cdot \vec{V}_{m}}
    \end{equation}
    and
    \begin{equation} \label{eqn:b}
      b_m\bc{\vec{U},\vec{V}}
      :=
      \frac12 \bc{ \vec{U}_{m+1}^- \cdot \vec{V}_{m+1}^+
        + \vec{U}_{m+1}^+ \cdot \vec{V}_{m+1}^-}
      - \frac12 \bc{ \vec{U}_m^- \cdot \vec{V}_m^+
        + \vec{U}_m^+ \cdot \vec{V}_m^- }
      .
    \end{equation}
  }

\end{proposition}

\revise{Note that here we have introduced variables $a_m$ and $ b_m$ to
  represent the flux contributions arising in the properties of
  $\gfunc{}$. This is done to simplify exposition in the
  sequel, in particular for the proof of Theorem \ref{thm:clawsdg}.}
\revise{Additionally,} the structure of the input arguments of $\gfunc{}$ is not
always consistent within Proposition \ref{prop:gfunc}, i.e., sometimes
$\gfunc{}$ operates on a vector and sometimes a scalar. We have made
this choice \revise{as $\gfunc{}$ shall act on both in the sequel.}

\begin{proof}

  To prove Proposition \ref{prop:gfunc} it suffices to show the local
  properties. All three global properties follow immediately from the
  globalisation of the local properties in view of periodic boundary
  conditions. To show local properties we shall utilise the definition
  of $\gfunc{}$ as described in Definition \ref{def:gfunc}. We may
  restrict this definition to a single element by choosing
  \begin{equation}
    \phi(x) =
    \begin{cases}
      \chi(x) \quad \forall x \in \selement{m} \\
      0 \qquad \textrm{ otherwise,}      
    \end{cases}
  \end{equation}
  for all $\chi \in \dpoly{q}\bc{\selement{m}}$, and can be explicitly
  written as
  \begin{equation} \label{eqn:dg:lgfunc}
    \int_{\selement{m}} \gfunc{\vec{U}} \cdot \chi \di{x}
    =
    \int_{\selement{m}} \vec{U}_x \cdot \chi \di{x}
    - \frac12 \jump{\vec{U}_m} \cdot \chi_m^+
    - \frac12 \jump{\vec{U}_{m+1}} \cdot \chi_{m+1}^-
    \qquad
    \forall \chi \in \dpoly{q}\bc{\selement{m}}
    .
  \end{equation}
  
  Through \eqref{eqn:dg:lgfunc} we observe that
  \begin{equation}
    \begin{split}
      \int_{\selement{m}} \gfunc{\vec{U}} \cdot \vec{1} \di{x}
      & =
      \int_{\selement{m}} \vec{U}_x \cdot \vec{1} \di{x}
      - \frac12 \jump{\vec{U}_m} \cdot\vec{1}
      - \frac12 \jump{\vec{U}_{m+1}} \cdot \vec{1}
      \\
      & =
      \bc{\vec{U}_{m+1}^- - \vec{U}_m^+} \cdot \vec{1}
      - \frac12 \bc{\vec{U}_m^--\vec{U}_m^+} \cdot \vec{1}
      - \frac12 \bc{\vec{U}_{m+1}^--\vec{U}_{m+1}^+} \cdot \vec{1}
      \\
      & = 
      \avg{\vec{U}_{m+1}} \cdot \vec{1} 
      -\avg{\vec{U}_m} \cdot \vec{1}
      ,
    \end{split}
  \end{equation}
  through the fundamental theorem of calculus and Definition
  \ref{def:jumpavg} yielding \eqref{eqn:dg:lgorth}. Again, in view of
  \eqref{eqn:dg:lgfunc}
  \begin{equation}
    \begin{split}
      \int_{\selement{m}} \gfunc{\vec{U}} \cdot \vec{V} \di{x}
      & = 
      \int_{\selement{m}} \vec{U}_x \cdot \vec{V} \di{x}
      - \frac12 \jump{\vec{U}_m} \cdot \vec{V}_m^+
      - \frac12 \jump{\vec{U}_{m+1}} \cdot \vec{V}_{m+1}^-
      \\
      & = 
      - \int_{\selement{m}} \vec{U} \cdot \vec{V}_x \di{x}
      + \frac12 \vec{U}_{m+1}^- \cdot \vec{V}_{m+1}^-
      - \frac12 \vec{U}_m^+ \cdot \vec{V}_m^+
      \\
      & \qquad
      + \frac12 \vec{U}_{m+1}^+ \cdot \vec{V}_{m+1}^-
      - \frac12 \vec{U}_m^- \cdot \vec{V}_m^+
      \\
      & = 
      - \int_{\selement{m}} \vec{U} \cdot \gfunc{\vec{V}} \di{x}
      + \frac12 \vec{U}_{m+1}^- \cdot \vec{V}_{m+1}^+
      + \frac12 \vec{U}_{m+1}^+ \cdot \vec{V}_{m+1}^-
      \\
      & \qquad
      - \frac12 \vec{U}_m^- \cdot \vec{V}_m^+
      - \frac12 \vec{U}_m^+ \cdot \vec{V}_m^-
      ,
    \end{split}
  \end{equation}
  through integration by parts and application of Definition
  \ref{def:jumpavg} yielding \eqref{eqn:dg:lgskew}.

  Note that as \eqref{eqn:dg:lgorth} holds for
  vectors it trivially holds for \revise{any scalar $w$ and
    \begin{equation}
      \int_{\selement{m}} \gfunc{w}
      =
      \frac12 \avg{w_{m+1}}
      -
      \frac12 \avg{w_m}
      .
    \end{equation}
    In particular, when $w=\vec{U}\cdot\vec{V}$ we observe that}
  \begin{equation} \label{p:dg:gprop1}
    \int_{\selement{m}} \gfunc{\vec{U} \cdot \vec{V}} \di{x}
    =
    \revise{
      a_m\bc{\vec{U}\cdot\vec{V}}
      ,
      }
    \end{equation}
    \revise{which we shall be utilised in the sequel.}
    Through summing \eqref{p:dg:gprop1} and
  \eqref{eqn:dg:lgskew} we obtain \eqref{eqn:dg:lgprod} allowing us to
  conclude. 

\end{proof}

\begin{theorem}[Discrete conservation laws] \label{thm:clawsdg}

  Let the conditions of Definition \eqref{def:stfem} hold, and in
  particular let $\vec{Z}$ be the solution of the space-time finite
  element method \eqref{eqn:stfemdg}. Additionally define the scalar function
  \begin{equation} \label{eqn:dg:w}
    W :=
    \nabla \S{\vec{Z}} \cdot \project{\gfunc{\vec{Z}}}
    ,
  \end{equation}
  where $\project{}$ is the $L_2$ projection into the finite element
  space $\bc{\bpoly{q}\bc{\telement{n}}}^D \times \cpoly{p}\bc{S^1}$.
  Further define the fluxes $\dmflux,\deflux$ and densities
  $\dmdensity, \dedensity$ such that
  \begin{align} \label{eqn:dg:fluxdensity}
    \dmflux = \frac12 \vec{Z}_t \cdot K\vec{Z} - \S{\vec{Z}}
    \quad & \quad
      \dmdensity = \frac12 \gfunc{\vec{Z}} \cdot K\vec{Z}
    \\
    \deflux = \frac12 \vec{Z}_t \cdot L \vec{Z}
    \quad & \quad
            \dedensity = \frac12 \gfunc{\vec{Z}} \cdot L \vec{Z}
            \ \revise{+} \ \S{\vec{Z}} 
            .
  \end{align}
  The numerical solution $\vec{Z}$ conserves a consistent discrete momentum
  locally in time, namely
  \begin{equation} \label{eqn:dg:mclaw}
    \stlint \dmdensity_t + \gfunc{\dmflux + \S{\vec{Z}}} - W \di{x} \di{t} = 0
    .
  \end{equation}
  Additionally, a discrete energy is conserved locally in time,
  which we can write as
  \begin{equation} \label{eqn:dg:eclaw}
    \stlint \dedensity_t + \gfunc{\deflux} \di{x} \di{t} = 0
    .
  \end{equation}
  We can restrict these conservation laws locally in space allowing us
  to write
  \begin{equation} \label{eqn:dg:lmclaw}
    \begin{split}
      \sltlint \dmdensity_t + \gfunc{\dmflux + \S{\vec{Z}}} - W \di{x} \di{t} 
      & = 
      \revise{\tint
        \frac12 a_m\bc{\vec{Z}\cdot K \vec{Z}}
        + \frac12 b_m\bc{\vec{Z}, K \vec{Z}}
          \di{t}}
      ,
  \end{split}
  \end{equation}
  and
  \begin{equation} \label{eqn:dg:leclaw}
    \begin{split}
      \sltlint \dedensity_t + \gfunc{\deflux} \di{x} \di{t} & =
      \revise{\tint
        \frac12 a_m\bc{\vec{Z}\cdot L \vec{Z}}
        + \frac12 b_m\bc{\vec{Z}, L\vec{Z}}
        \di{t}
      }
    \end{split} 
  \end{equation}
  respectively, \revise{where $a_m$ and $b_m$ are the flux
    contributions defined in \eqref{eqn:a} and \eqref{eqn:b}.}

\end{theorem}

\begin{proof}[Proof of Theorem \ref{thm:clawsdg}]

  \revise{Before beginning we note that the arguments in this proof
    follow a similar structure to that of the proof of Theorem
    \ref{thm:claws}, but requires us to overcome additional
    technical details.} Through the definition of the momentum density
  $\dmdensity$ given in \eqref{eqn:dg:fluxdensity} we can write
  \begin{equation} \label{p:dg:mclaw00}
    \begin{split}
      \sltlint \dmdensity_t \di{x}\di{t}
      & =
      \sltlint \frac12 \gfunc{\vec{Z}}_t \cdot K \vec{Z} 
      + \frac12 \gfunc{Z} \cdot K \vec{Z}_t
      \revise{\di{x} \di{t}}
      .
    \end{split}
  \end{equation}
  \revise{Through local integration by parts \eqref{eqn:dg:lgskew} and
    the skew-symmetry of $K$ we may rewrite \eqref{p:dg:mclaw00} as
    \begin{equation} \label{p:dg:mclaw0}
      \begin{split}
        \sltlint \dmdensity_t \di{x} \di{t}
        & =
        \sltlint
        -\frac12 \vec{Z}_t \cdot K \gfunc{\vec{Z}}
        + \frac12 \gfunc{\vec{Z}} \cdot K \vec{Z}_t
        \di{x}\di{t}
        +
        \tint \frac12 b_m\bc{\vec{Z}_t,K\vec{Z}} \di{t} \\
        & =
        \sltlint
        \gfunc{\vec{Z}} \cdot K \vec{Z}_t
        \di{x}\di{t}
        + \tint
        \frac12 b_m\bc{\vec{Z}_t,K\vec{Z}}
        \di{t}
        .
      \end{split}
    \end{equation}
    Note that here we have used that
    $\gfunc{K\vec{Z}} = K \gfunc{\vec{Z}}$, which holds as $K$ is a
    \emph{constant} skew-symmetric matrix.}
  Further, through the
  definition of the momentum flux $\dmflux$ in
  \eqref{eqn:dg:fluxdensity} we find
  \begin{equation} \label{p:dg:mclaw1}
    \begin{split}
      \sltlint \gfunc{{\dmflux} + \S{\vec{Z}}} - W \di{x} \di{t}
      & =
      \sltlint
      - \nabla \S{\vec{Z}} \cdot \project{\gfunc{\vec{Z}}}
      \di{x} \di{t}
      \ \revise{
        +\tint \frac12 a_m\bc{\vec{Z}_t \cdot K \vec{Z}}
        \di{t}
        }
      ,
    \end{split}
  \end{equation}
  \revise{after simplification, application of \eqref{p:dg:gprop1}
  and by construction of $W$ \eqref{eqn:dg:w}.} \revise{Note that
  $a_m$ and $b_m$ represent flux contributions as defined in
    Proposition \ref{prop:gfunc}.} We observe that, due to a
  difference in the temporal component of the trial and test spaces,
  $\gfunc{\vec{Z}}$ is not in the test space. With this in mind we
  choose $\chi = \project{\gfunc{\vec{Z}}}$ in our local scheme
  \eqref{eqn:dg:lsltfem}, that
  \begin{equation} \label{p:dg:mclaw2}
    \begin{split}
      0 & =
      \sltlint
      K \vec{Z}_t \cdot \project{\gfunc{\vec{Z}}} 
      + L \gfunc{\vec{Z}} \cdot \project{\gfunc{\vec{Z}}}
      - \nabla \S{\vec{Z}} \cdot \project{\gfunc{\vec{Z}}} \di{x} \di{t}
      \\
      & = 
      \sltlint
      K \vec{Z}_t \cdot \gfunc{\vec{Z}}
      + L \project{\gfunc{\vec{Z}}} \cdot \project{\gfunc{\vec{Z}}}
      - \nabla \S{\vec{Z}} \cdot \project{\gfunc{\vec{Z}}} \di{x} \di{t}
      \\
      & = 
      \sltlint
      K \vec{Z}_t \cdot \gfunc{\vec{Z}}
      - \nabla \S{\vec{Z}} \cdot \project{\gfunc{\vec{Z}}} \di{x} \di{t}
      ,
    \end{split}
  \end{equation}
  through the definition of the $L_2$ projection into the finite
  element space
  $\bpoly{q}\bc{\selement{m}} \times \bpoly{p}\bc{\telement{n}}$ and
  the skew symmetry of the matrix $L$. Applying \eqref{p:dg:mclaw2} to
  \eqref{p:dg:mclaw1} we can write
  \begin{equation} \label{p:dg:mclaw3}
    \begin{split}
      \sltlint \gfunc{{\dmflux} + \S{\vec{Z}}} - W \di{x} \di{t}
      & =
      \revise{ \sltlint - \gfunc{\vec{Z}} \cdot K \vec{Z}_t \di{x}\di{t}
        + \tint \frac12 a_m\bc{\vec{Z}_t \cdot K\vec{Z}}
        \di{t}
        .
      }
    \end{split}
  \end{equation}
  \revise{Summing \eqref{p:dg:mclaw3} with \eqref{p:dg:mclaw0} we
    obtain the} local momentum conservation law \eqref{eqn:dg:lmclaw}.
  Through summing the \eqref{eqn:dg:lmclaw} for $m=0,...,M-1$ then we
  obtain the spatially global momentum conservation law
  \eqref{eqn:dg:mclaw} in view of the periodic boundary conditions.

  Turning our attention to the proof of local energy conservation, we
  apply the definitions of the energy density
  \eqref{eqn:dg:fluxdensity} finding
  \begin{equation}
    \sltlint \dedensity_t \di{x}\di{t}
    =
    \sltlint \frac12 \gfunc{\vec{Z}}_t \cdot L \vec{Z}
    + \frac12 \gfunc{\vec{Z}} \cdot L \vec{Z}_t
    - \vec{Z}_t \cdot \nabla \S{\vec{Z}} \di{x}\di{t}
    .
  \end{equation}
  Through choosing $\chi = \vec{Z}_t$ in the localised space-time
  scheme \eqref{eqn:dg:lsltfem} we can write
  \begin{equation} \label{p:dg:leclaw1}
    \sltlint \dedensity_t \di{x}\di{t}
    =
    \sltlint \frac12 \gfunc{\vec{Z}}_t \cdot L \vec{Z}
    + \frac12 \gfunc{\vec{Z}} \cdot L \vec{Z}_t
    + K\vec{Z}_t \cdot \vec{Z}_t
    - L\gfunc{\vec{Z}} \cdot \vec{Z}_t
    \di{x}\di{t}
    .
  \end{equation}
  Through the skew-symmetric inner products induced by $K$ and $L$,
  and the skew-symmetry of $\gfunc{}$, we
  can reduce \eqref{p:dg:leclaw1} to
  \begin{equation} \label{p:dg:leclaw2}
    \begin{split}
      \sltlint \dedensity_t \di{x}\di{t}
      & =
      \ \revise{
        \sltlint \frac12 \gfunc{\vec{Z}}_t \cdot L \vec{Z}
        - \frac12 \gfunc{\vec{Z}} \cdot L \vec{Z}_t
        \di{x} \di{t}
        =
        \tint \frac12 b_m\bc{\vec{Z}_t,L\vec{Z}} \di{t}
      }
      .
    \end{split}
  \end{equation}
  Additionally we observe that
  \begin{equation} \label{p:dg:leclaw3}
    \begin{split}
      \sltlint \gfunc{\deflux} \di{x}\di{t}
      & =
      \ \revise{
        \tint
        \frac12 a_m\bc{\vec{Z}_t \cdot L \vec{Z}}
        \di{t}
    }
    ,
    \end{split}
  \end{equation}
  utilising \eqref{eqn:dg:lgorth} from Proposition
  \ref{prop:gfunc}. Summing \eqref{p:dg:leclaw2} and \eqref{p:dg:leclaw3} we
  obtain the local energy conservation law \eqref{eqn:dg:leclaw}. Further
  summing the local result over the spatial domain, and observing the
  periodic boundary conditions, we obtain the global result \eqref{eqn:dg:eclaw}.

\end{proof}

\subsection{Existence and uniqueness} \label{sec:dg:eu}

To show existence and uniqueness of the spatially discontinuous scheme
for wave equations with $\V{u}\ge 0$ we may mimic the arguments made
in Section \ref{sec:cg:eu}, as long as our numerical solution remains
within a bounded set. We again utilise the energy conservation law,
which states that
\begin{equation}
  \begin{split}
    0 & =
    \stlint \bc{ \frac12 \vec{Z} \cdot K \gfunc{\vec{Z}} -
      \S{\vec{Z}}}_t
    \di{x} \di{t}
    =
    \stlint \bc{ \frac12 \gfunc{W}U - \frac12 \gfunc{U}W - \frac12 V^2 +
      \frac12 W^2 - \V{U} }_t
    \di{x}\di{t} \\
    & =
    \stlint \bc{-\frac12 \gfunc{U}^2 - \frac12 V^2 - \V{U}}_t
    \di{x}\di{t}
    ,
  \end{split}
\end{equation}
after noting the skew-symmetry of $\gfunc{}$ and that
\begin{equation}
  \stlint \bc{W - \gfunc{U}} W_t \di{x}\di{t} = 0
  \qquad
  \textrm{and}
  \qquad
  \stlint \bc{W-\gfunc{U}}\gfunc{U}_t \di{x}\di{t} = 0
  .
\end{equation}
This immediately yields boundedness of
$\norm{L_2(S^1)}{\gfunc{U(t_n)}}$, $\norm{L_2(S^1)}{V(t_n)}$ and
$\norm{\V{S^1}}{U(t_n)}$ for arbitrary time $t=t_n$, so after mimicking
the arguments made in Section \ref{sec:cg:eu} we may conclude that a
unique numerical solution exists.

\section{Numerical experiments} \label{sec:numerics}

In this section we illustrate the numerical performance of the
numerical methods designed in Section \ref{sec:cg} and Section
\ref{sec:dg} for select examples through summarising extensive
numerical experiments. The brunt of the computational work here has
been conducted by Firedrake \cite{Firedrake:2017, Firedrake:extruded}
and utilises PETsc solvers \cite{petsc-user-ref}. Throughout we shall
use direct linear solvers and set our nonlinear solver tolerance to be
$10^{-12}$. We employ a Gauss quadrature which is either of high
enough order to be able to integrate exactly, or of order $16$,
whichever criteria is met first.

For each benchmark test we fix the temporal and spatial polynomial
degrees $q$ and $p$, respectively, assume uniform spatial and temporal
element sizes, and compute a sequence of solutions with
$\dt{} = \dx{} = h(i) = 2^{-i}$ with corresponding errors $a(i)$ with
the following definition in mind.

\begin{definition}[Experimental order of convergence]
  
  Given two sequences $a(i)$ and $h(i) \searrow 0$ we define the
  \emph{experimental order of convergence} (EOC) to be the local slope
  of the $\log{a(i)}$ vs. $\log{h(i)}$ curve, i.e.,
  \begin{equation} \label{eqn:eoc}
    \textrm{EOC}(a,h;i) = \frac{\log{\frac{a_{i+1}}{a_i}}}{\log{\frac{h_{i+1}}{h_i}}}.
  \end{equation}
  
\end{definition}

Throughout \revise{this section}, we shall graphically show the
propagation of $L_2$ errors over time, i.e., we measure the error
\revise{sequentially up to} the temporal node $t_n$ in the Bochner
norm
\begin{equation} \label{eqn:enorm}
  L_2\bc{\bs{0,t_n}, L_2\bc{S^1}}
  .
\end{equation}

In addition, we examine the preservation of the discrete
momentum and energy conservation laws
\begin{equation} \label{eqn:momentum2}
  \stlint \dmdensity_t + \gfunc{\dmflux} \di{x} \di{t} = 0
\end{equation}
and
\begin{equation} \label{eqn:energy2}
  \stlint \dedensity_t + \gfunc{\deflux} \di{x} \di{t} = 0
\end{equation}
respectively, where $\dmdensity,\dmflux,\dedensity,\deflux$ are as
given by \eqref{eqn:dg:fluxdensity}. Note that when considering the
spatially continuous approximation the action of $\gfunc{}$ is exactly
the continuous spatial derivative. In particular, we shall check that
the value of the momentum on the temporal nodes does not deviate,
i.e., that
\begin{equation} \label{eqn:momentum}
  \int_{S^1} \dmdensity\bc{t_n,x} \di{x}
  =
  \int_{S^1} \dmdensity\bc{0,x} \di{x}
  ,
\end{equation}
and similarly that the value of the energy on the temporal nodes does not deviate
\begin{equation} \label{eqn:energy}
  \int_{S^1} \dedensity\bc{t_n,x} \di{x}
  =
  \int_{S^1} \dedensity\bc{0,x} \di{x}
  .
\end{equation}
By the fundamental theorem of calculus, this is equivalent to checking
\eqref{eqn:momentum2} and \eqref{eqn:energy2}. We note that
numerically, the deviation in these invariants may propagate over
time. In addition, we shall also numerically verify that our
approximations are mass invariant where appropriate.

\subsection{Test 1: The linear wave equation} \label{sec:num:lw}

Recall in Example \ref{ex:nlw} we showed that the nonlinear wave
equation falls within the multisymplectic framework, and indeed also
falls naturally within the variational framework discussed at the
beginning of Section
\ref{sec:multisym}. In particular the multisymplectic formulation of
the linear wave equation, i.e., where $\V{u} = 0$, may be simplified to
\begin{equation} \label{eqn:lw}
  \begin{split}
    v_t - w_x & = 0 \\
    v - u_t & = 0 \\
    w - u_x & = 0
    .
  \end{split}
\end{equation}
We shall begin by considering the spatially continuous approximation
\eqref{eqn:stlfem} which we note is the natural reduction of the
linear wave equation to a first order system. In addition, both the
proposed finite element approximation \eqref{eqn:stfem} and the
alternate momentum preserving approximation \eqref{eqn:mstfem},
described in Remark \ref{rem:mclaw}, are
equivalent in this setting. Assuming a solution in the form of a
harmonic wave, an exact solution for the auxiliary system may be
described by
\begin{equation} \label{eqn:lwexact}
  \begin{split}
    u & = \frac12 \sin{2\pi\bc{x+t}} \\
    v & = \pi \cos{2\pi\bc{x+t}} \\
    w & = \pi \cos{2\pi\bc{x+t}}
    .
  \end{split}
\end{equation}
Here we shall carry out numerical simulations of \eqref{eqn:stlfem} for the linear
wave equation with \eqref{eqn:lwexact} for varying polynomial degree. 
Plotting the errors and corresponding order of convergence we obtain
Figure \ref{fig:lw:eoc}.
\begin{figure}[h]
  \centering
  \subfigure[][$q=0$ and $p=1$]{
    \includegraphics[
    width=0.30\textwidth]{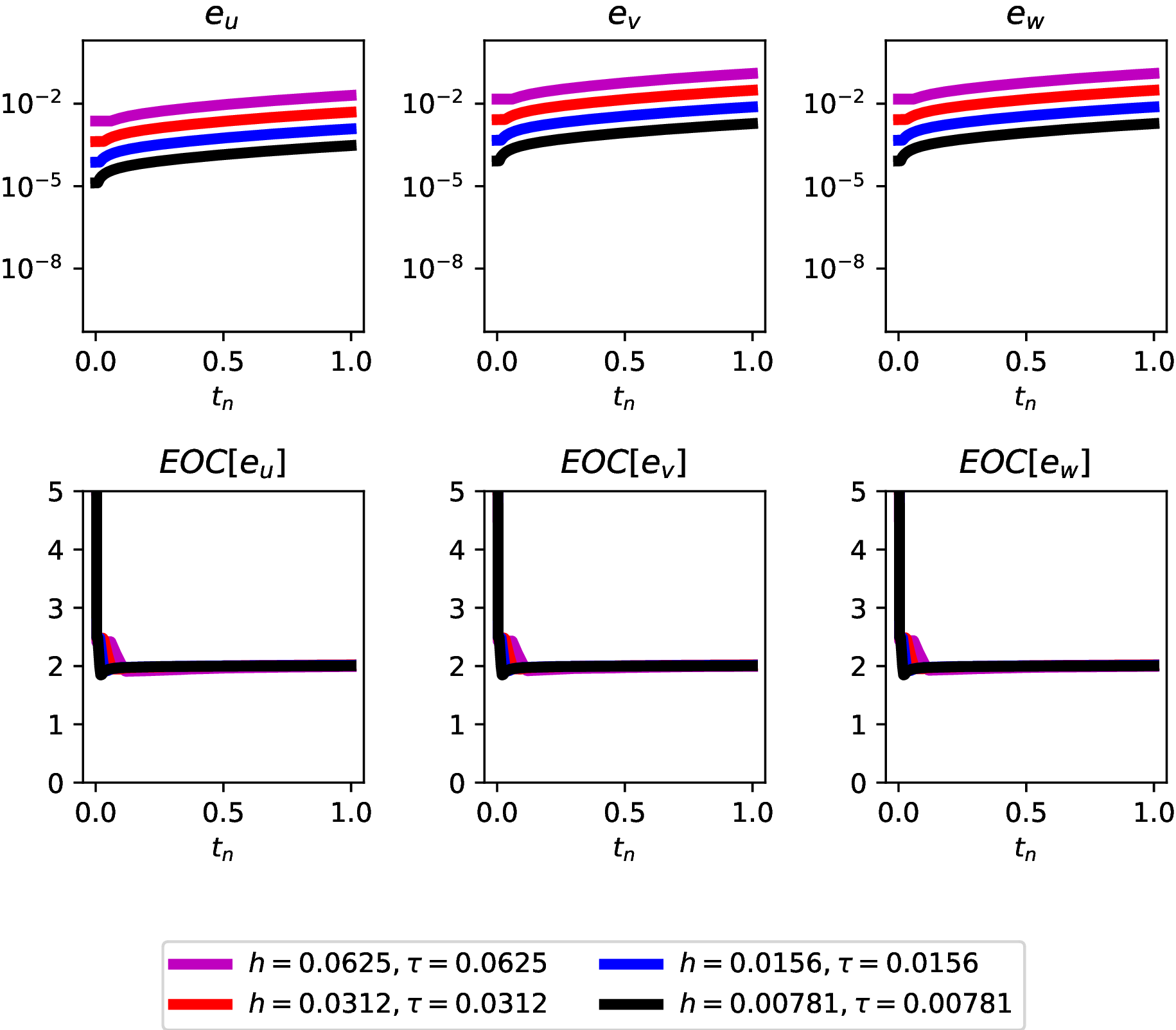}
  } \subfigure[][$q=0$ and $p=2$]{ \includegraphics[
    width=0.30\textwidth]{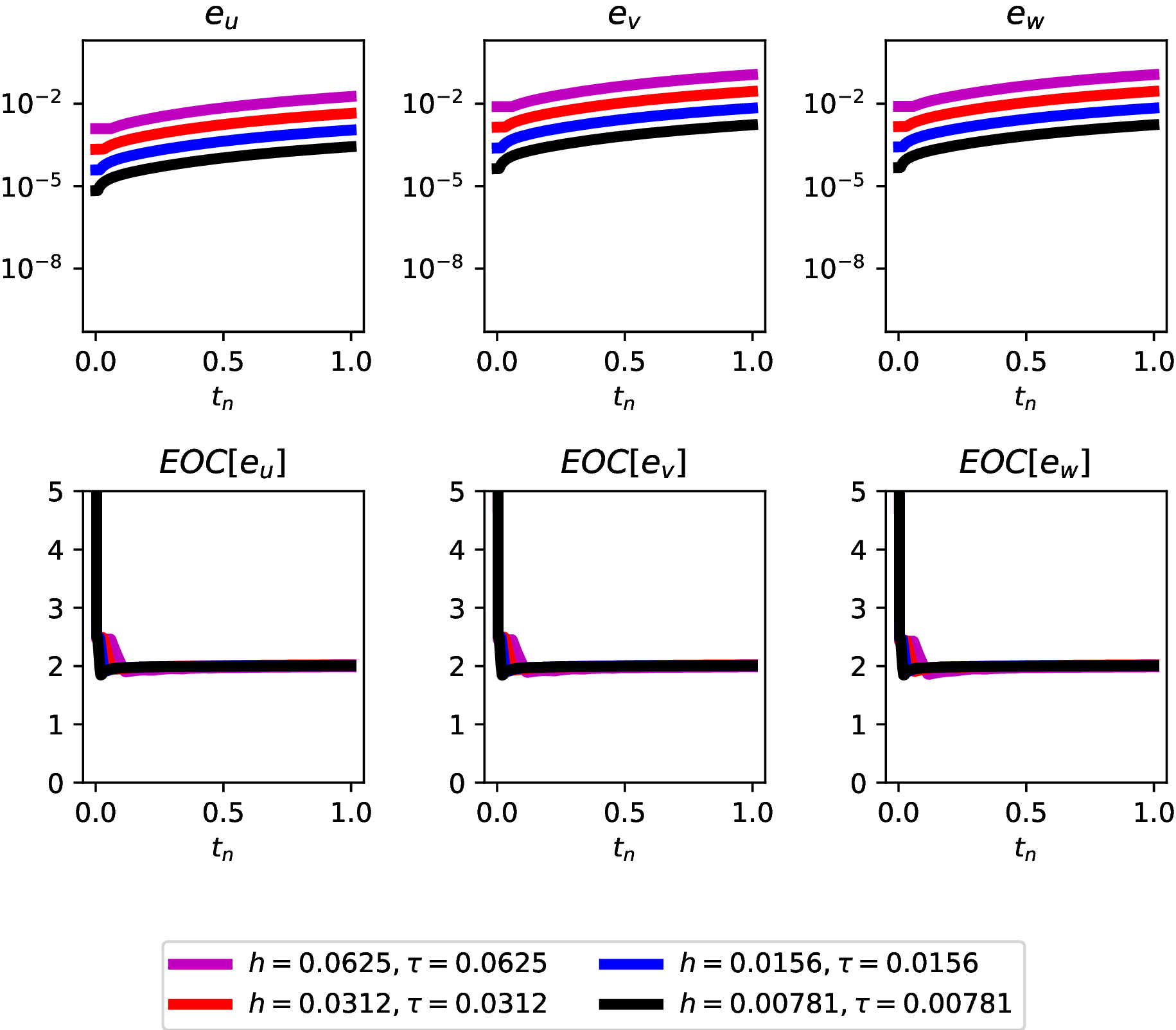}
  } \subfigure[][$q=0$ and $p=3$]{ \includegraphics[
    width=0.30\textwidth]{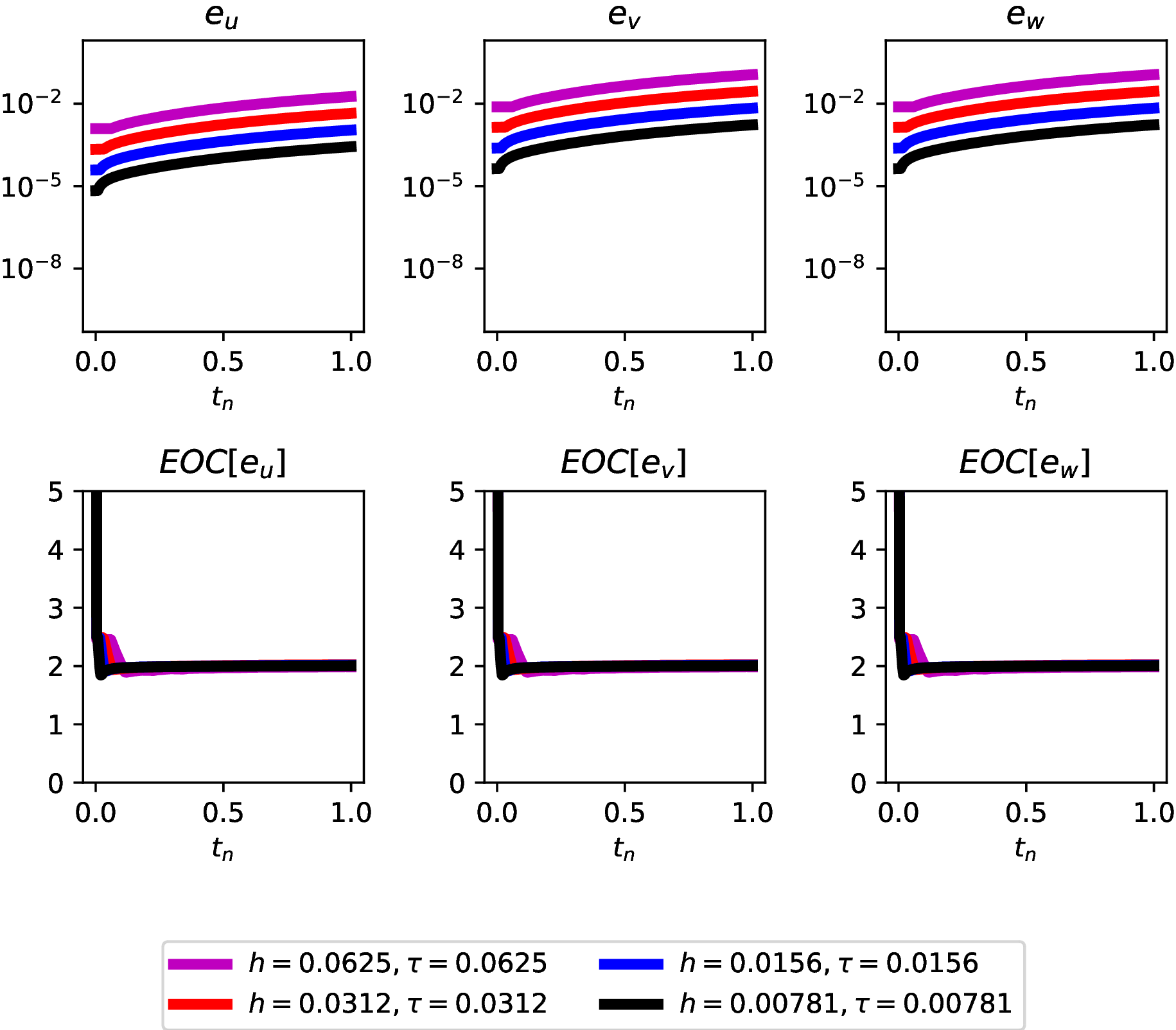}
  }
  \\
    \subfigure[][$q=1$ and $p=1$]{
    \includegraphics[
    width=0.30\textwidth]{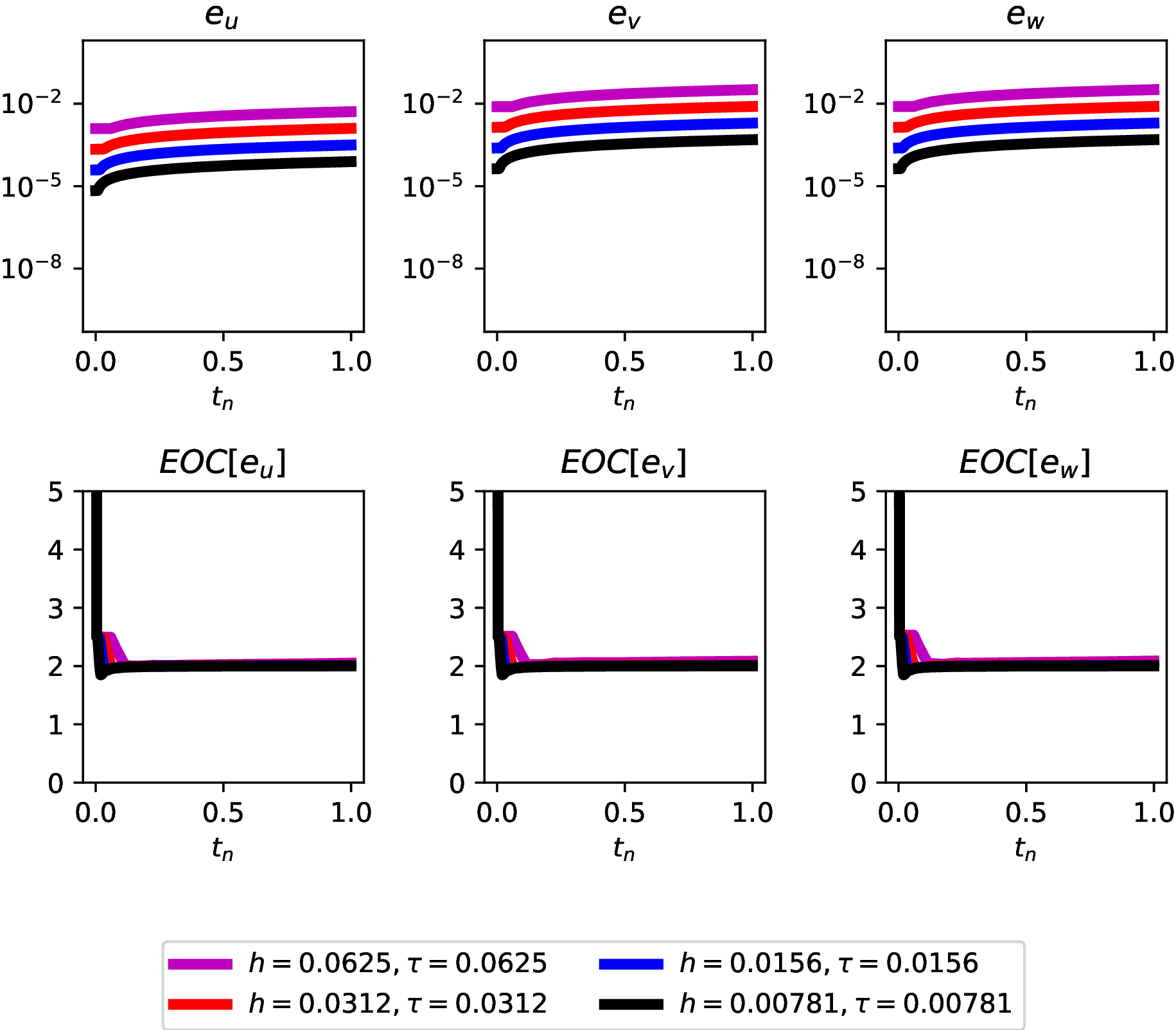}
  } \subfigure[][$q=1$ and $p=2$]{ \includegraphics[
    width=0.30\textwidth]{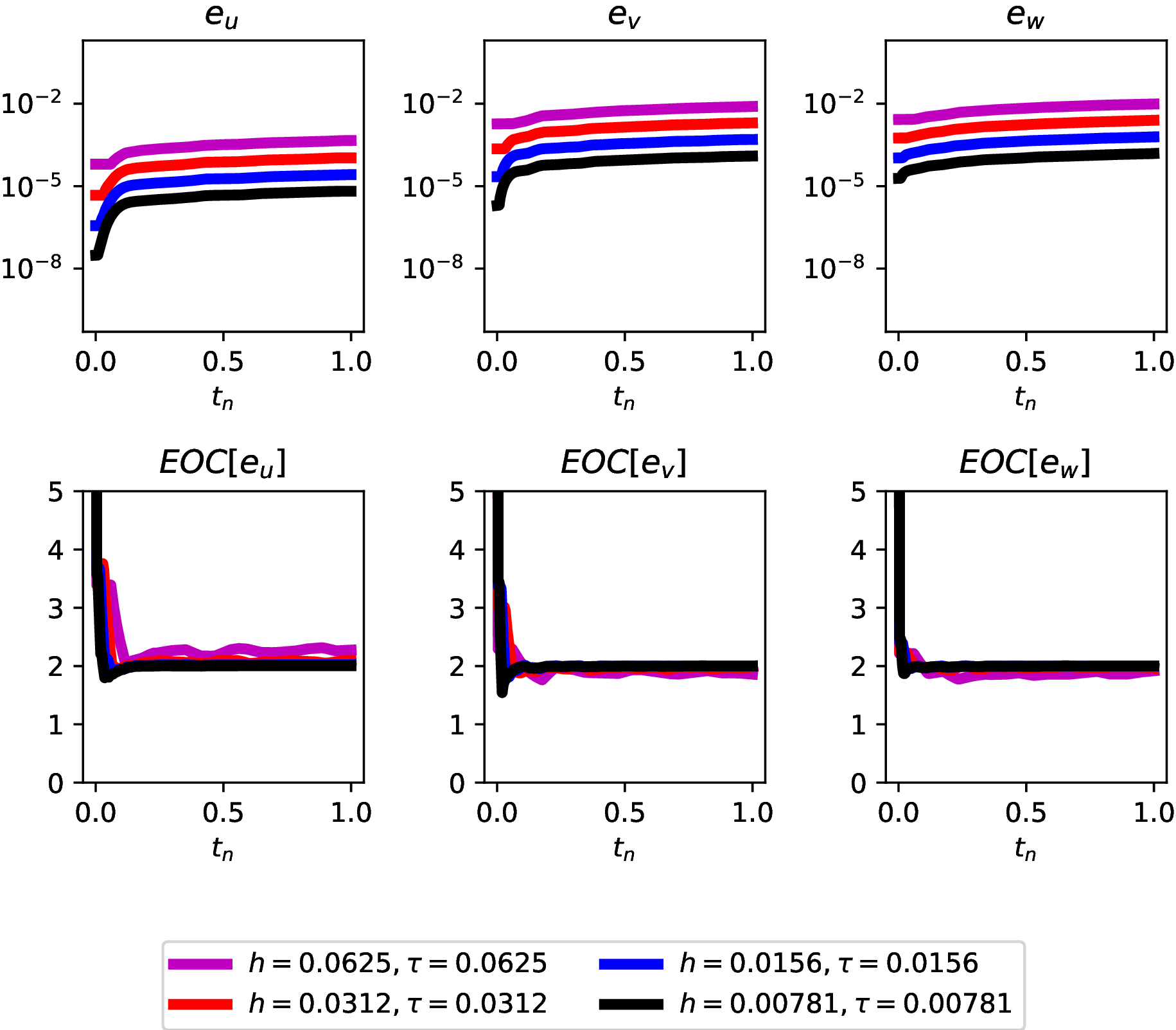}
  } \subfigure[][$q=1$ and $p=3$]{ \includegraphics[
    width=0.30\textwidth]{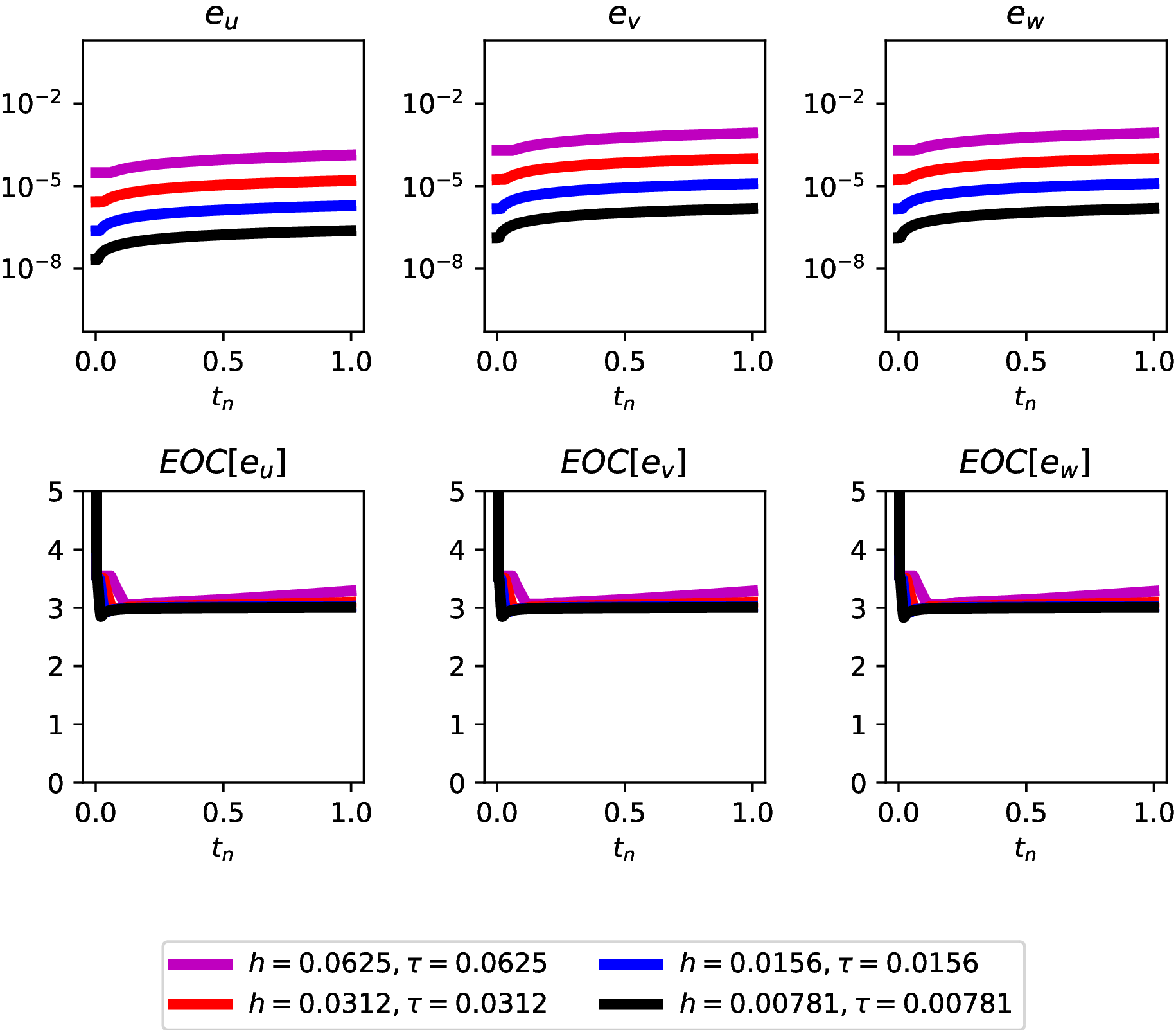}
  }
  \\
    \subfigure[][$q=2$ and $p=1$]{
    \includegraphics[
    width=0.30\textwidth]{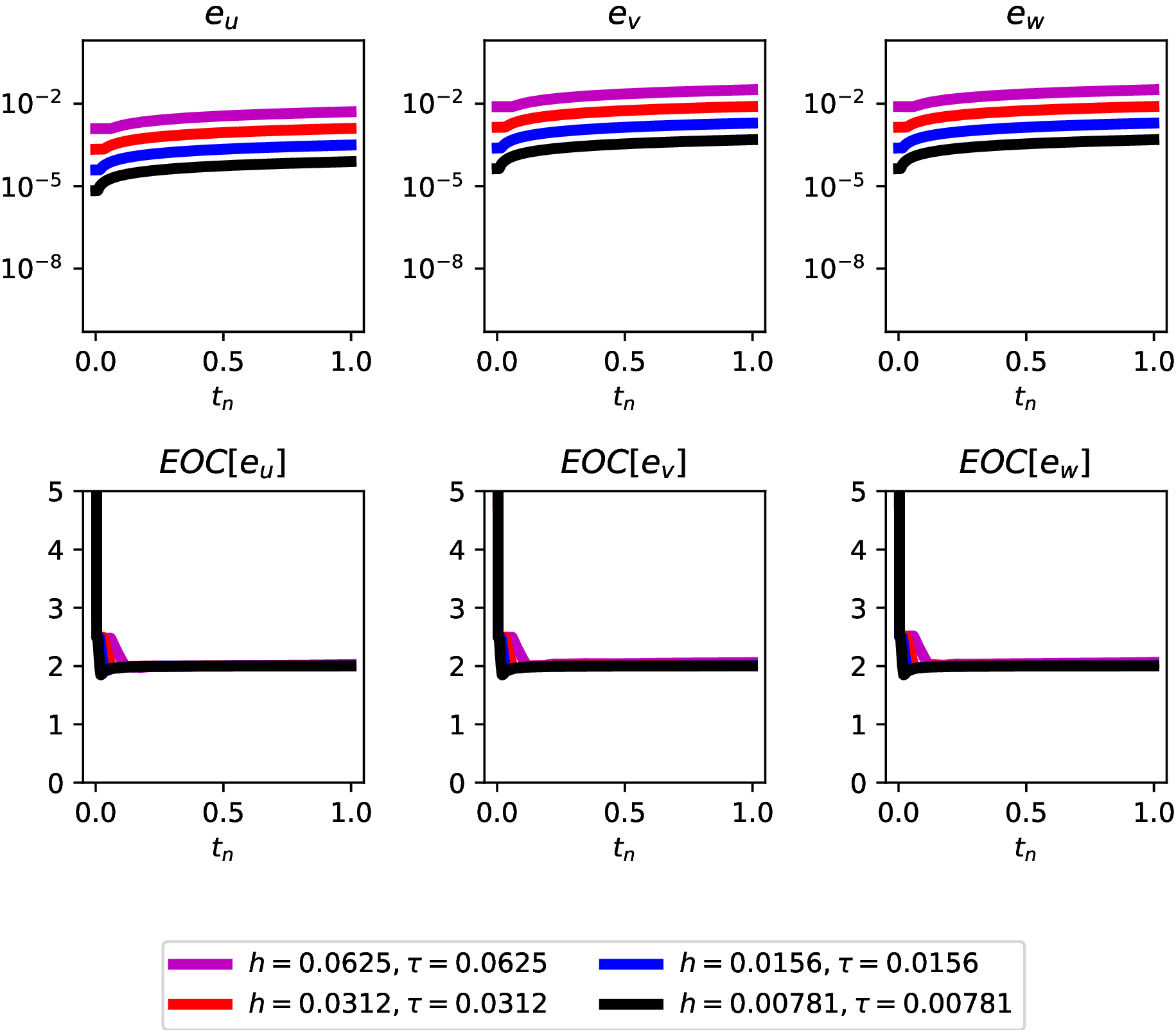}
  } \subfigure[][$q=2$ and $p=2$]{ \includegraphics[
    width=0.30\textwidth]{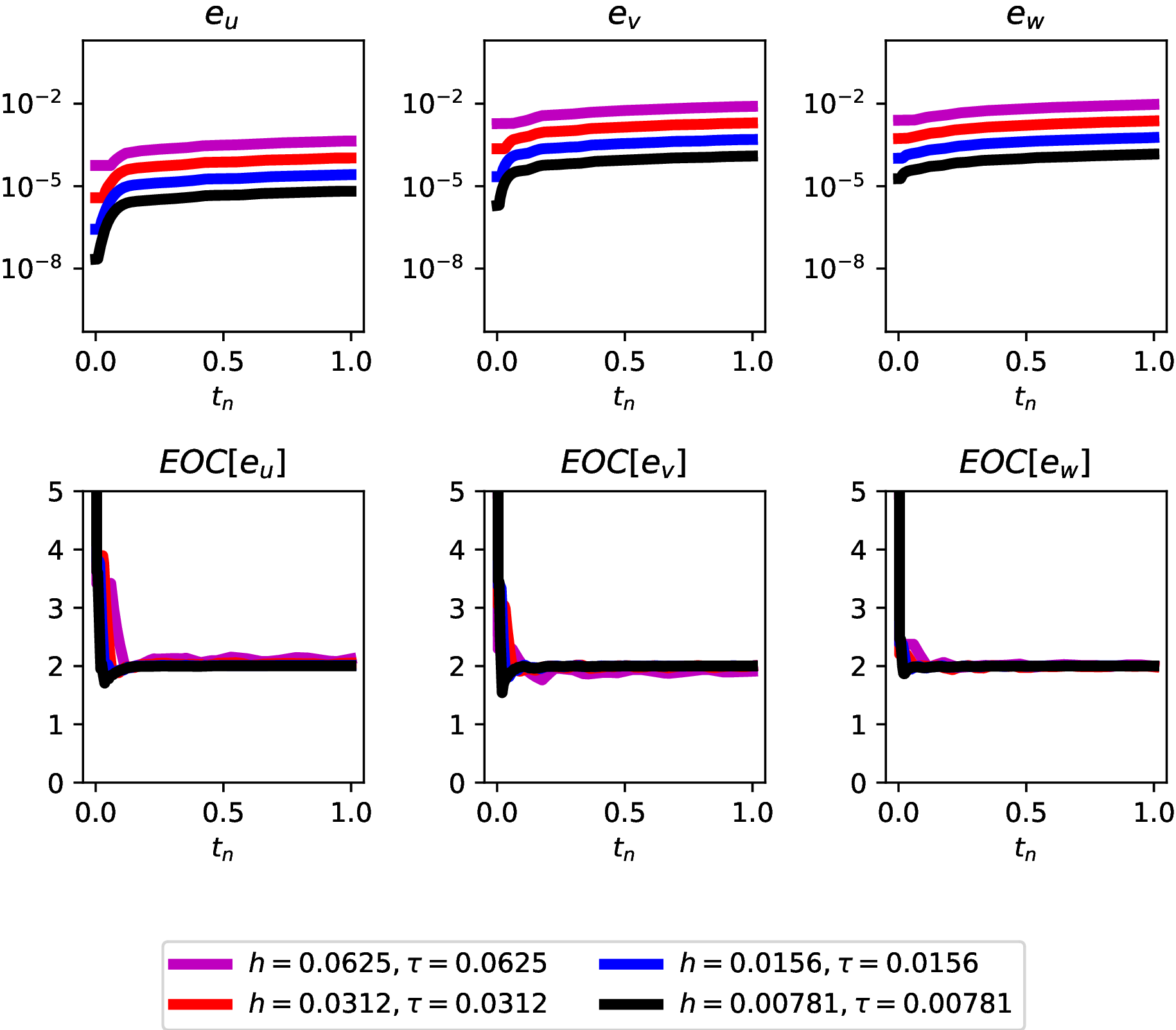}
  } \subfigure[][$q=2$ and $p=3$]{ \includegraphics[
    width=0.30\textwidth]{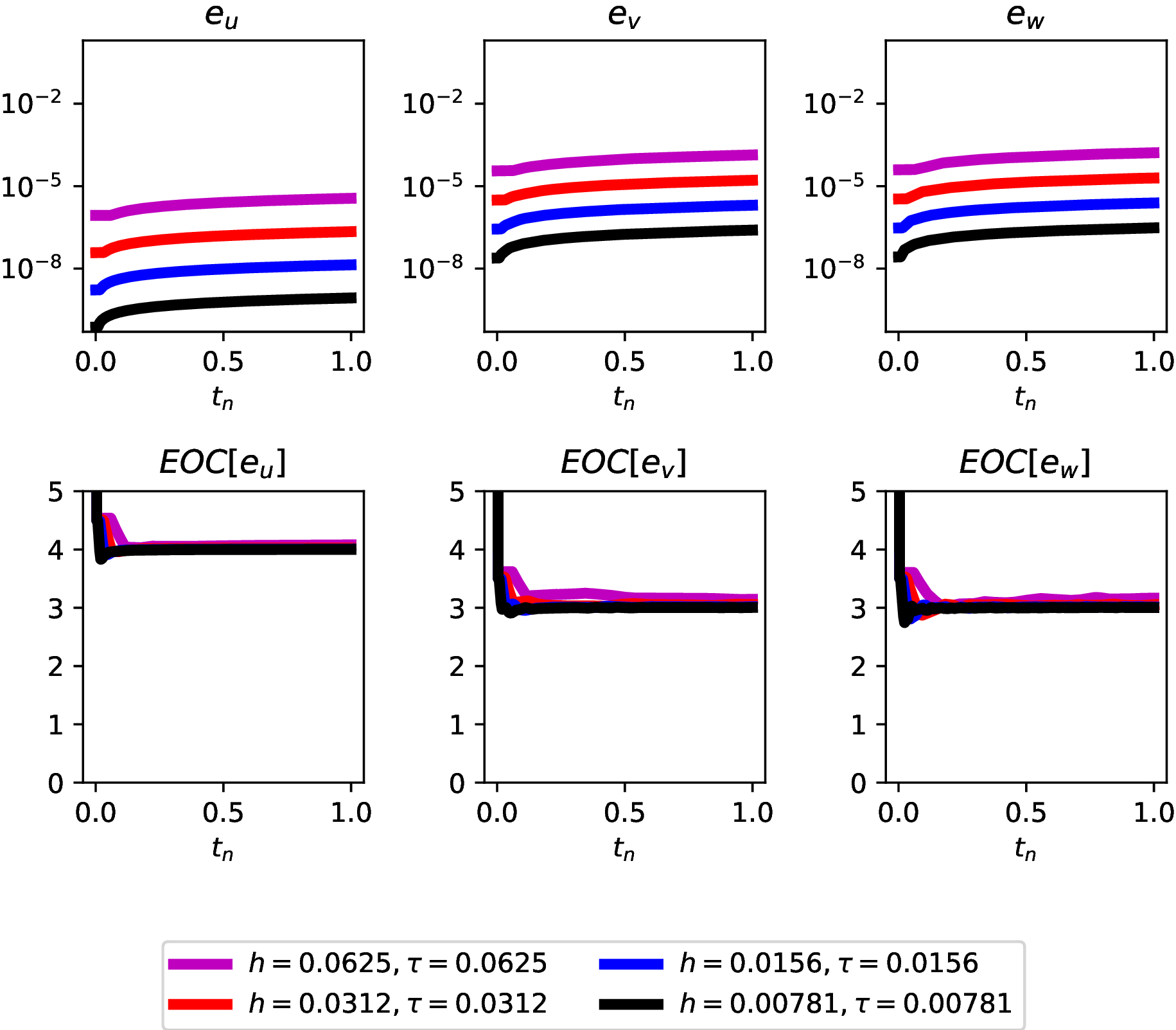}
  }

  \caption{The error of the spatially continuous finite element scheme
    \eqref{eqn:stlfem} for the linear wave equation \eqref{eqn:lw}
    initialised by the exact solution \eqref{eqn:lwexact}. Here
    $e_u,e_v,e_w$ denotes the errors in the components $U,V,W$
    measured in the Bochner norm \eqref{eqn:enorm}. Below we plot the
    EOC \eqref{eqn:eoc} corresponding to each of these errors. \label{fig:lw:eoc}}
\end{figure}
Within this figure we consider the errors for the solution $U$ and the
auxiliary variables independently. We shall now focus our attention to
the error in $U$, denoted $e_u$. The simulations suggest a temporal experimental order of accuracy of
$\mathcal{O}\bc{\dt{}^{q+2}}$, which is optimal in time. However,
spatially we observe that for $p$ even we obtain sub-optimal convergence
at rate $\mathcal{O}\bc{\dx{}^q}$, but for odd $p$ we converge
optimally at rate $\mathcal{O}\bc{\dx{}^{q+1}}$. This phenomenon is
supported by more extensive numerical experimentation.

Instead simulating the spatially discontinuous finite element approximation \eqref{eqn:stfemdg} for the linear
wave equation we obtain the errors and EOCs shown in Figure \ref{fig:lw:eoc:dg}.
\begin{figure}[h]
  \centering
  \subfigure[][$q=0$ and $p=1$]{
    \includegraphics[
    width=0.30\textwidth]{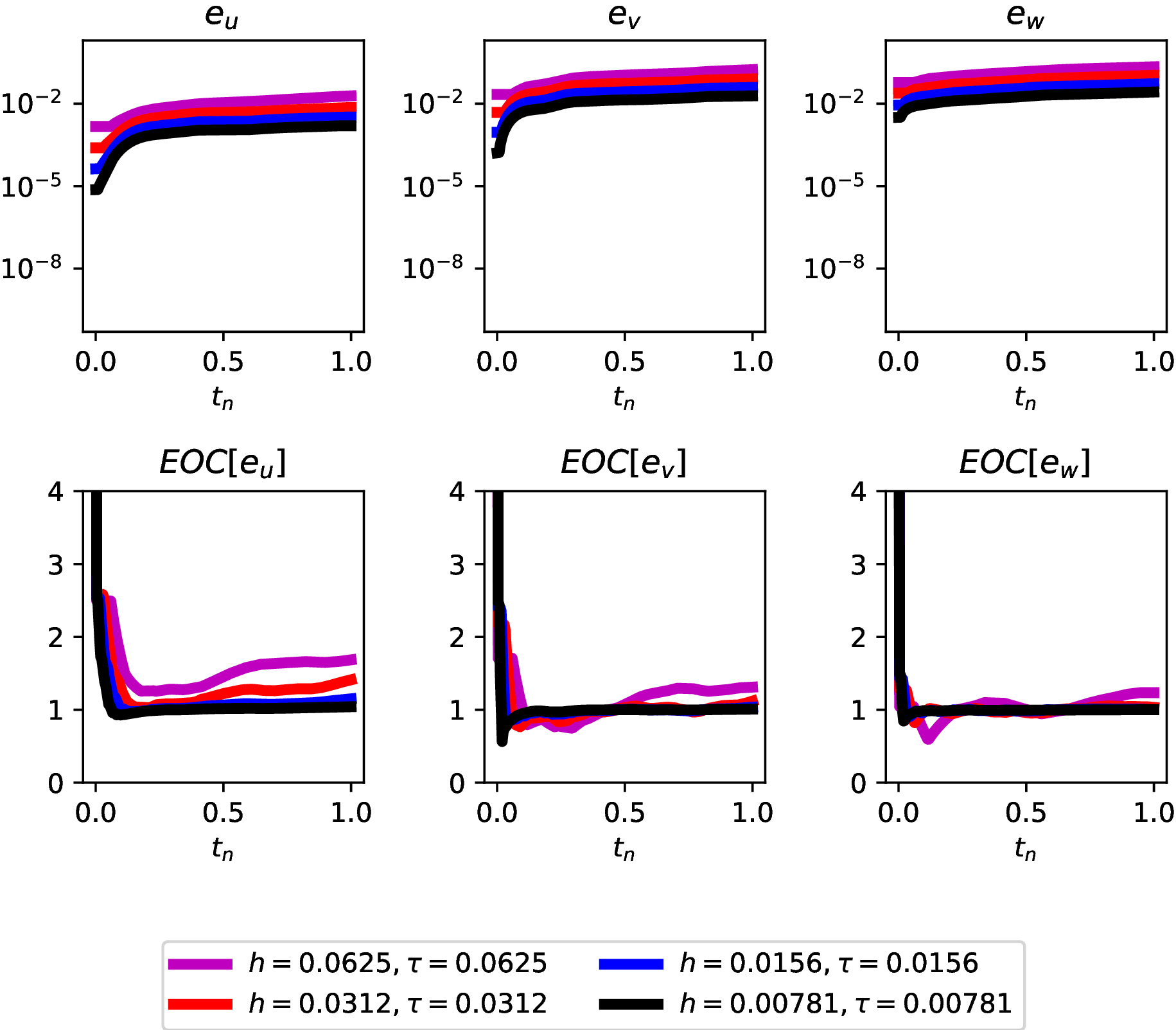}
  } \subfigure[][$q=0$ and $p=2$]{ \includegraphics[
    width=0.30\textwidth]{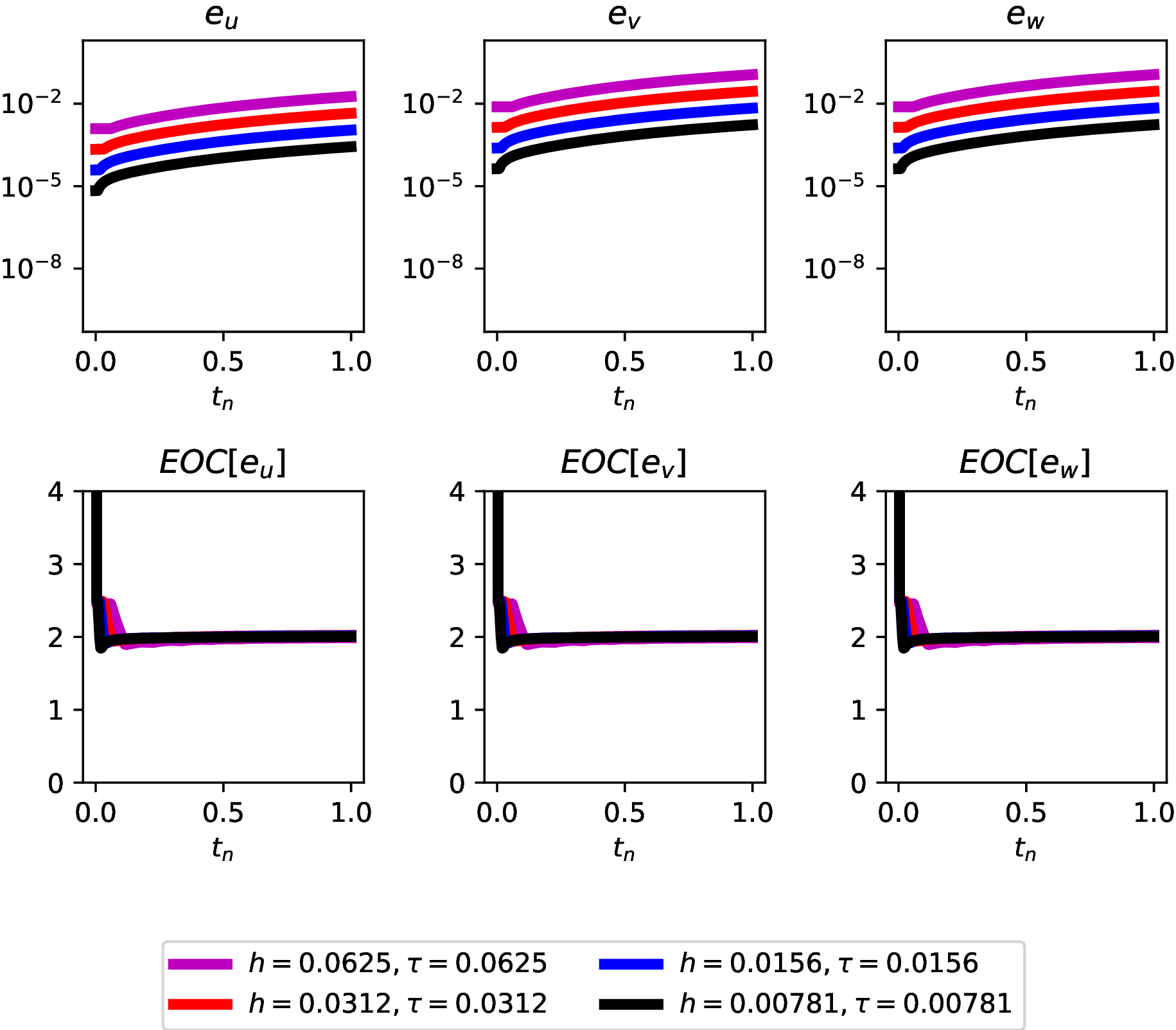}
  } \subfigure[][$q=0$ and $p=3$]{ \includegraphics[
    width=0.30\textwidth]{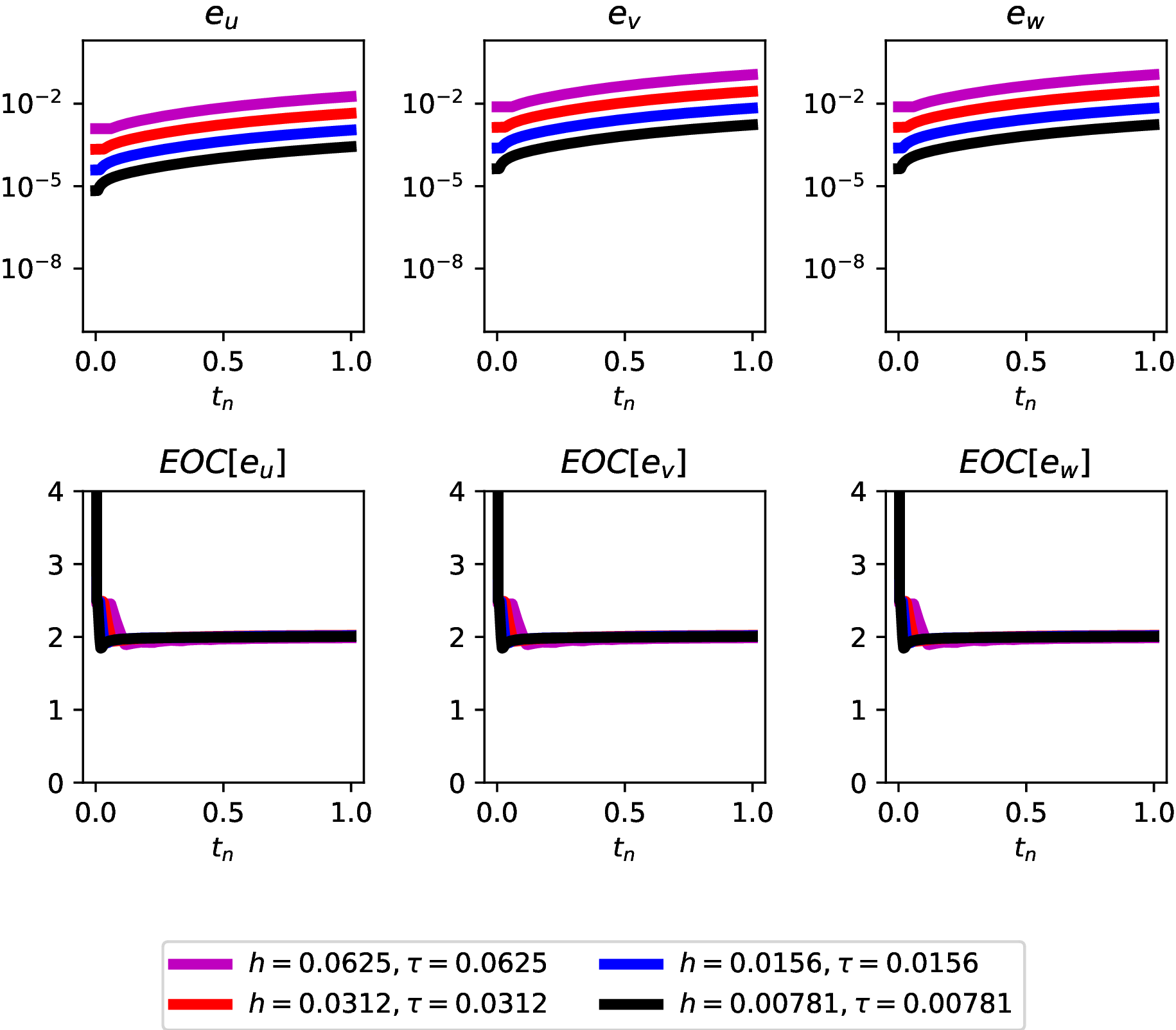}
  }
  \\
    \subfigure[][$q=1$ and $p=1$]{
    \includegraphics[
    width=0.30\textwidth]{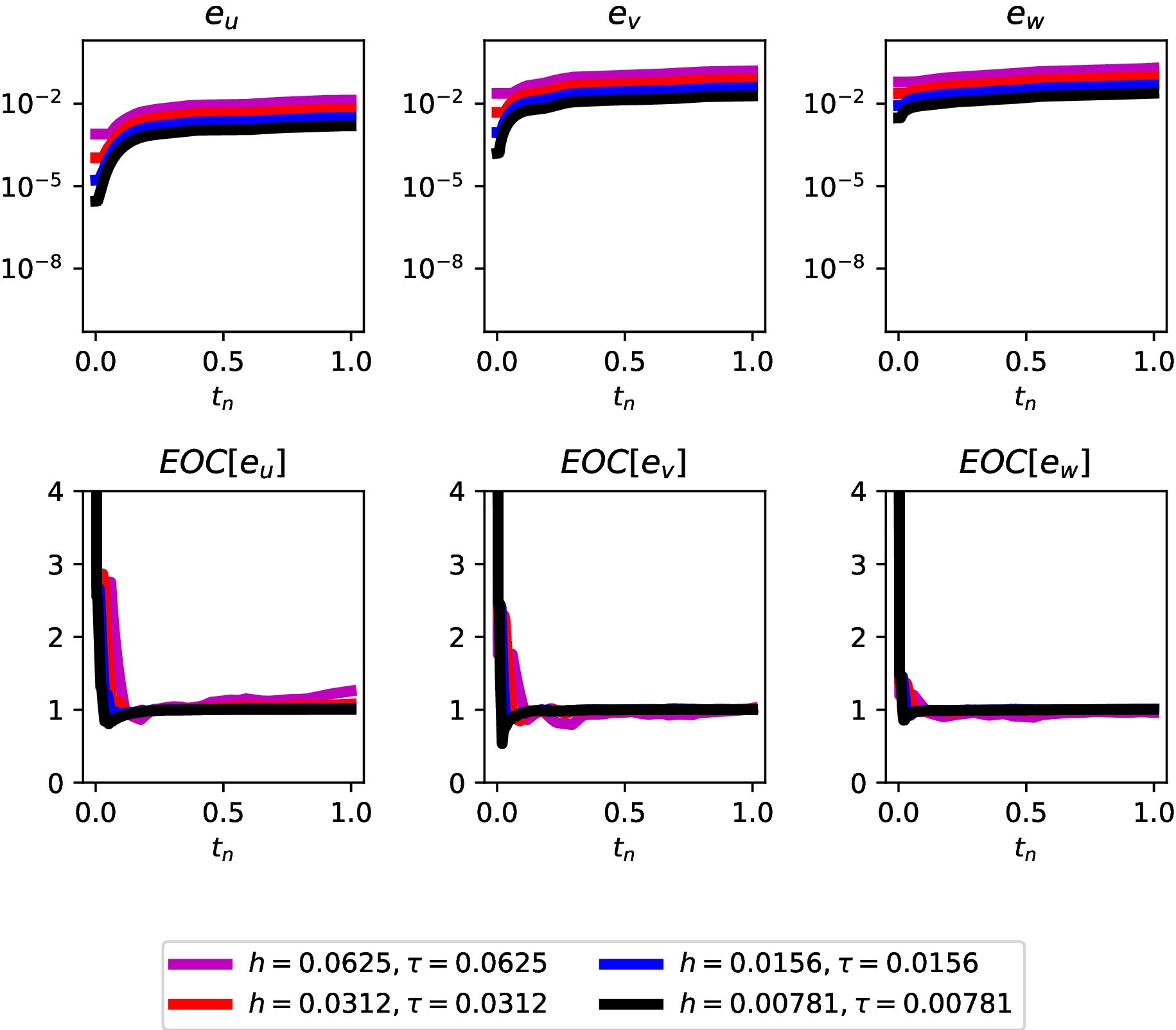}
  } \subfigure[][$q=1$ and $p=2$]{ \includegraphics[
    width=0.30\textwidth]{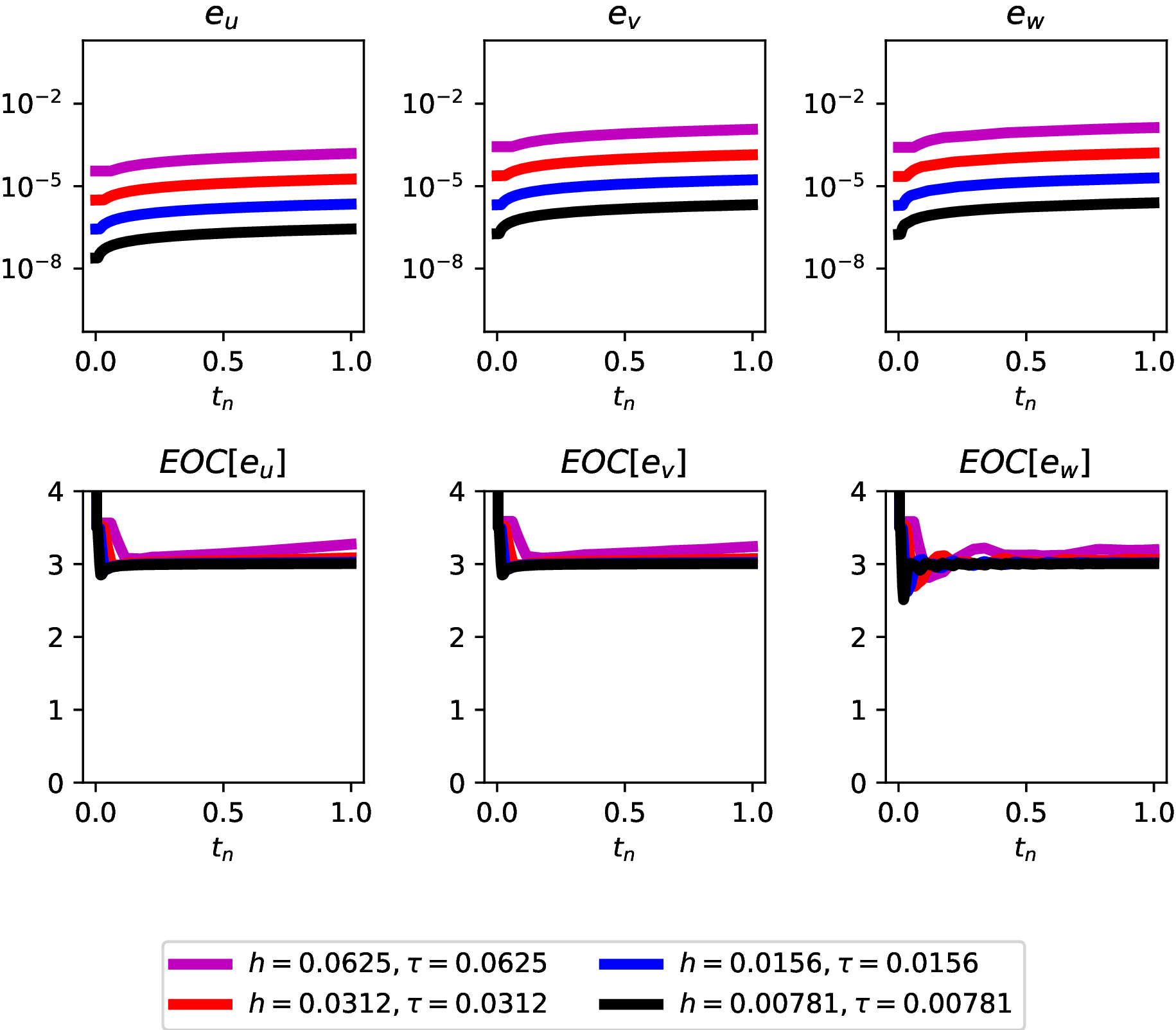}
  } \subfigure[][$q=1$ and $p=3$]{ \includegraphics[
    width=0.30\textwidth]{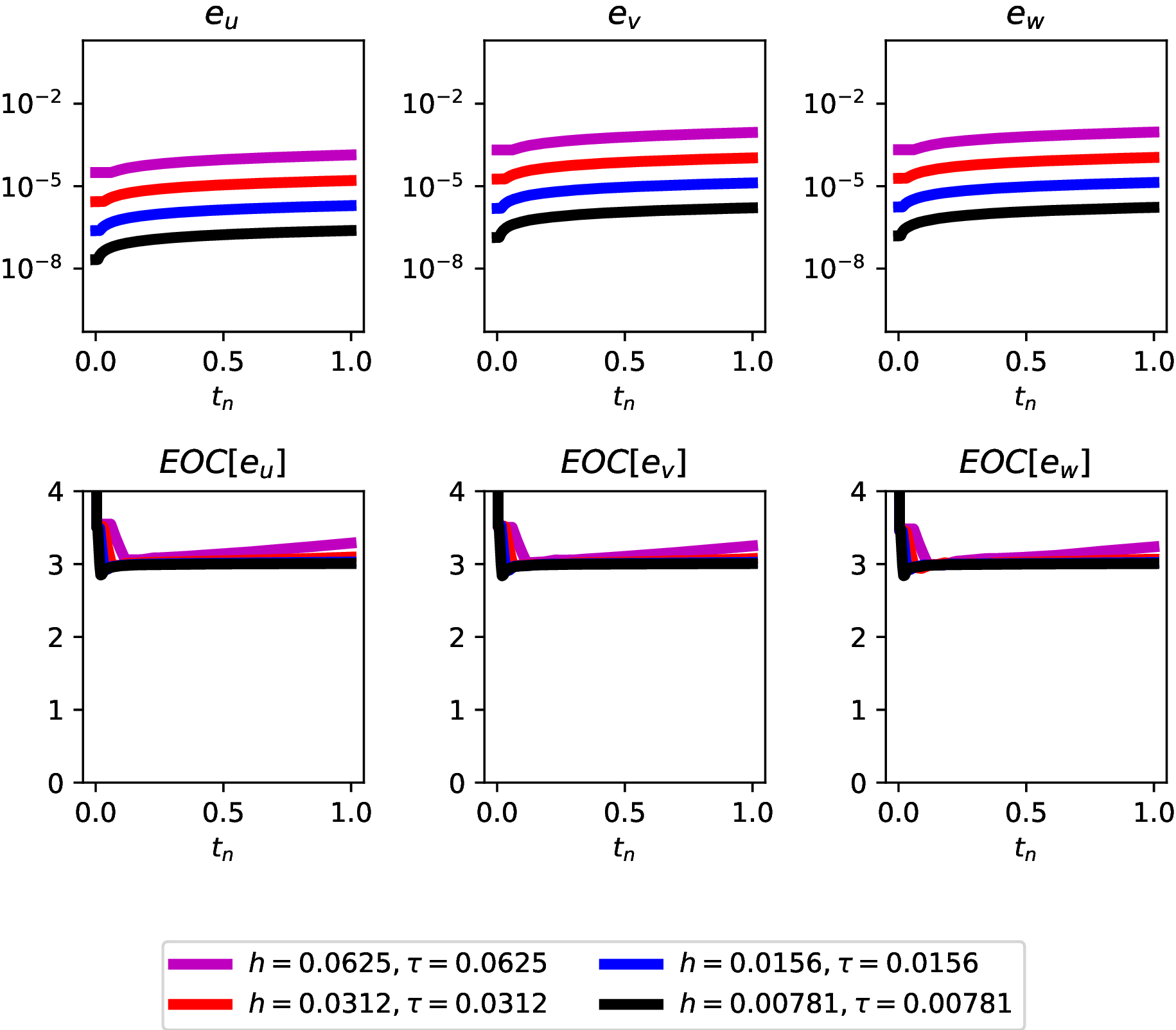}
  }
  \\
    \subfigure[][$q=2$ and $p=1$]{
    \includegraphics[
    width=0.30\textwidth]{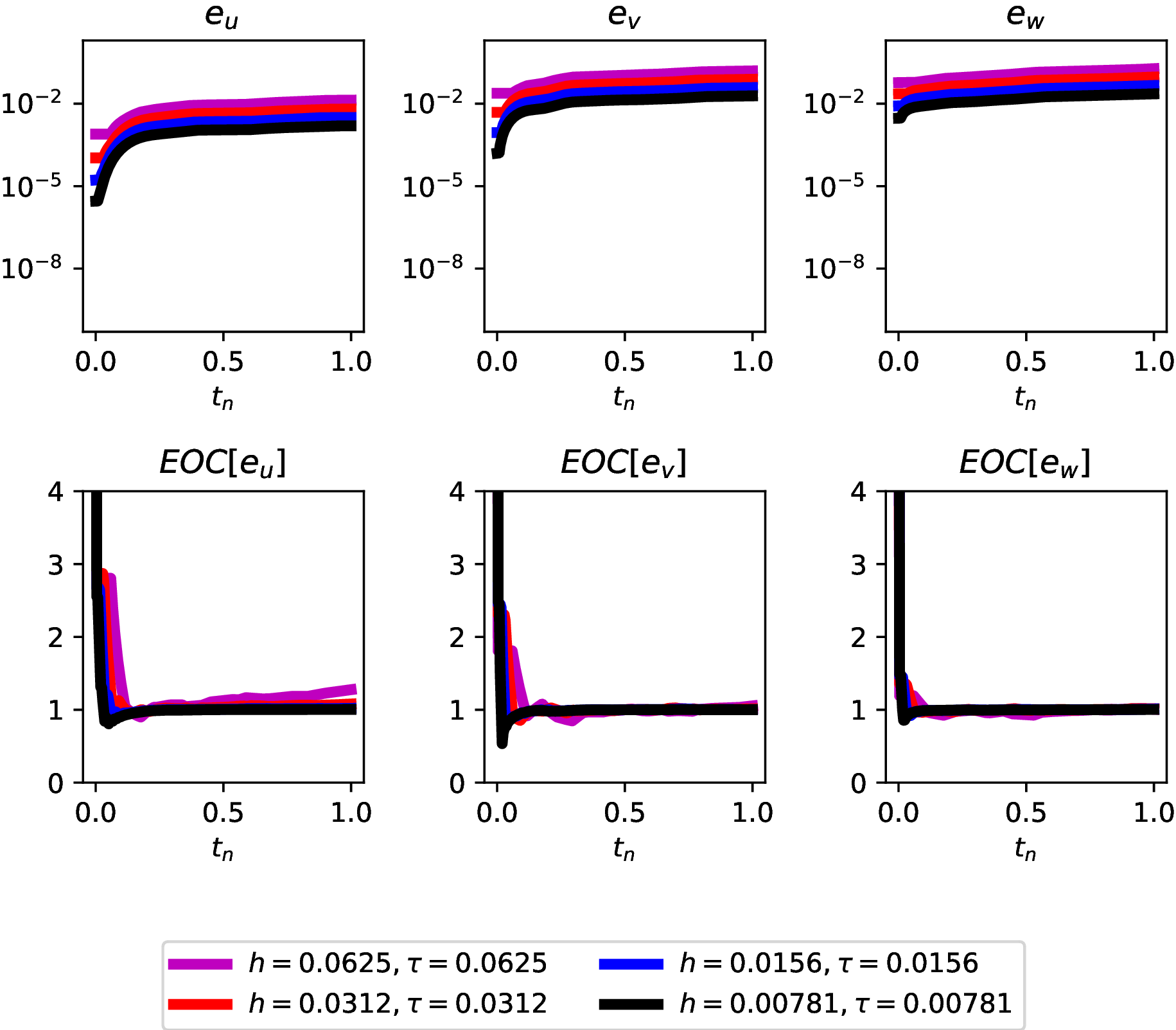}
  } \subfigure[][$q=2$ and $p=2$]{ \includegraphics[
    width=0.30\textwidth]{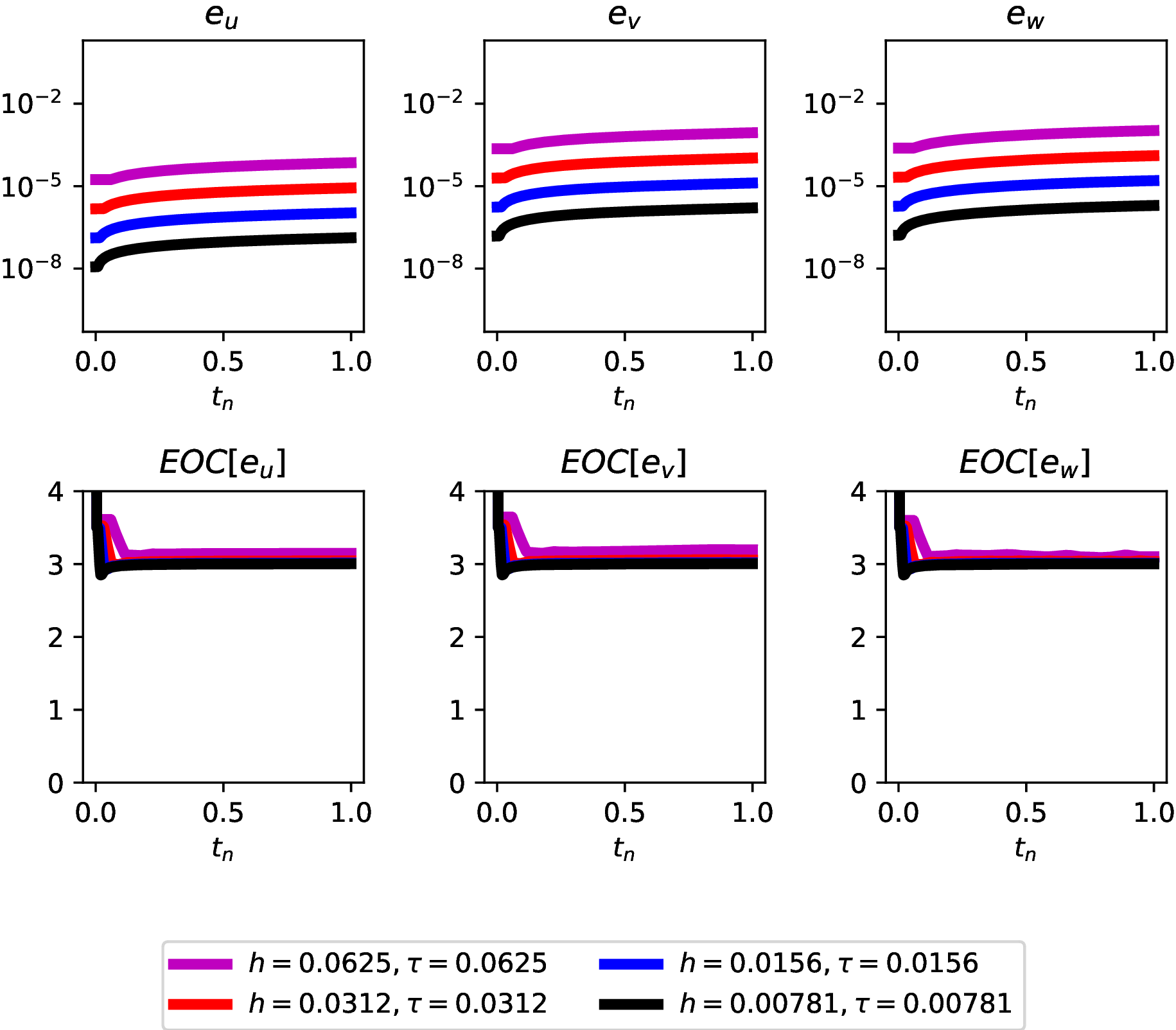}
  } \subfigure[][$q=2$ and $p=3$]{ \includegraphics[
    width=0.30\textwidth]{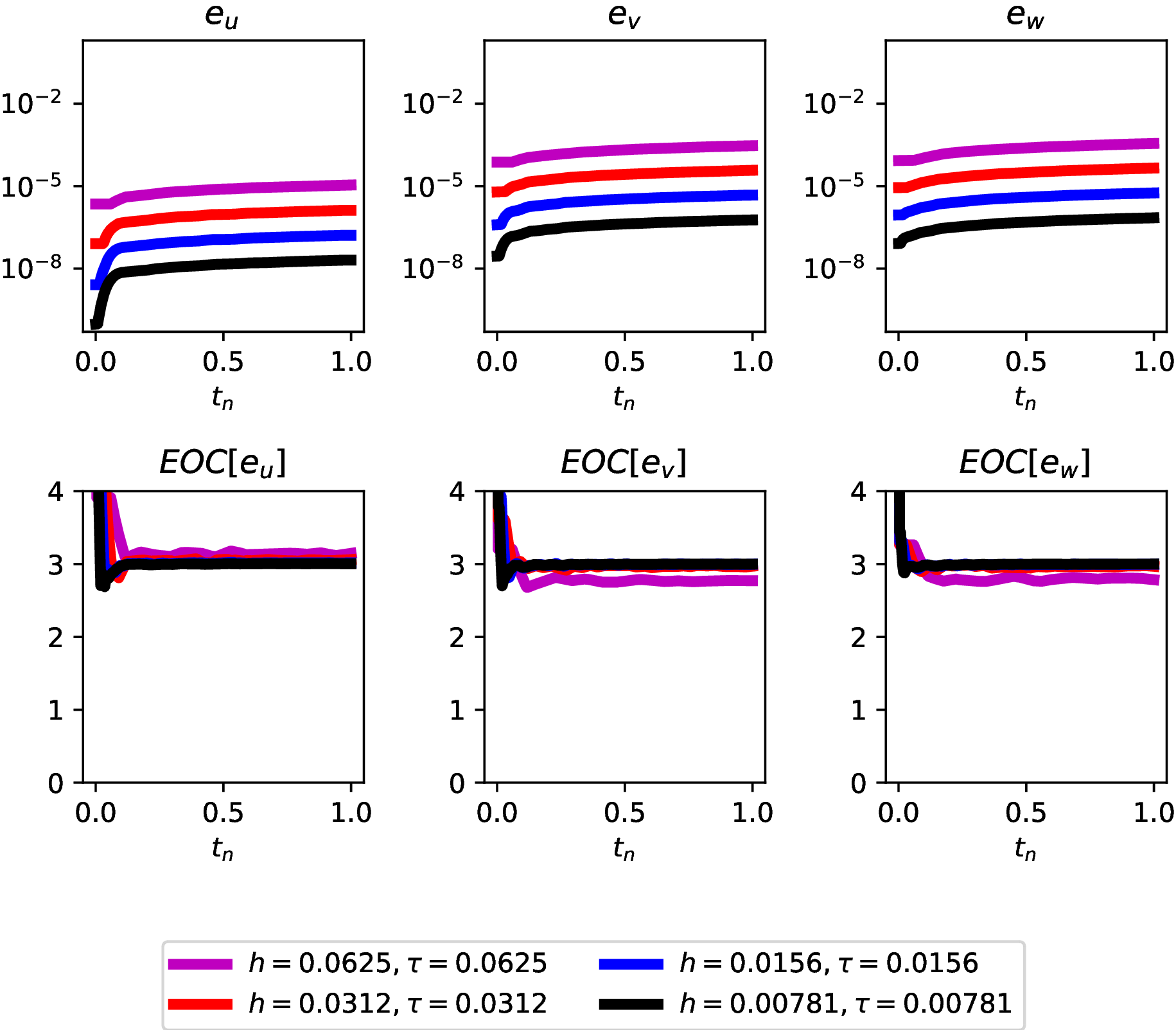}
  }

  \caption{The error of the spatially \emph{discontinuous} finite element scheme
    \eqref{eqn:stfemdg} for the linear wave equation \eqref{eqn:lw}
    initialised by the exact solution \eqref{eqn:lwexact}. Here
    $e_u,e_v,e_w$ denotes the errors in the components $U,V,W$
    measured in the Bochner norm \eqref{eqn:enorm}. Below we plot the
    EOC \eqref{eqn:eoc} corresponding to each of these errors. \label{fig:lw:eoc:dg}}
\end{figure}
These simulations suggest a temporal experimental order of accuracy of
$\mathcal{O}\bc{\dt{}^{q+2}}$, as expected as our temporal
discretisation is identical to the previous simulation. However,
spatially we observe that for $p$ odd we obtain sub-optimal convergence
at rate $\mathcal{O}\bc{\dx{}^q}$, but for even $p$ we converge
optimally at rate $\mathcal{O}\bc{\dx{}^{q+1}}$. This sub-optimality is
also observed in the spatially continuous case however the sub-optimal
and optimal cases are reversed. This phenomenon is supported by more
extensive numerical experimentation. In this case it is likely caused
by issues with the spatial first derivative operator $\gfunc{}$, which
is known to lack uniqueness for $p$ odd, see \cite{self:defocus}. We
may resolve this sub-optimal convergence for the linear wave equation
by modifying $\gfunc{}$ appropriately such that the spatial
derivatives involve an appropriate mix of upwind and downwind fluxes
as in \cite{GiesselmannPryer:2016}, however the resulting
discretisation not be valid for an arbitrarily multisymplectic PDE.

Further to these convergence simulations, all numerical simulations
for the linear wave equation preserve the momentum
\eqref{eqn:momentum} and energy \eqref{eqn:energy} up to machine
precision locally.
  
\subsection{Test 2: The nonlinear wave equation}

Throughout this test we shall restrict ourselves to the spatially
continuous scheme, as the numerical results do not differ in the
spatially discontinuous case. Instead choosing the potential in
Example \ref{ex:nlw} to be $\V{u} = \frac14 u^4$ we obtain a nonlinear
wave equation. We may explicitly rewrite the multisymplectic
formulation as the system
\begin{equation} \label{eqn:nlw}
  \begin{split}
    v_t - w_x + \Vp{u} & = 0 \\
    v - u_t & = 0 \\
    w - u_x & = 0
    .
  \end{split}
\end{equation}
In this setting, our proposed finite element scheme \eqref{eqn:stfem}
and the momentum conserving alternative \eqref{eqn:mstfem} are
distinct. In addition to the conservation laws obtained via the
multisymplectic structure we also expect simulations to be mass
conserving, i.e.,
\begin{equation} \label{eqn:nlwmass}
  \int_{S^1} U\bc{t_n,x} \di{x}
  =
  \int_{S^1} U\bc{0,x} \di{x}
  .
\end{equation}
Through fixing $\dt{}=0.1$ and $\dx{}=0.01$ we run extensive
numerical simulations investigating the deviation in mass, momentum
\eqref{eqn:momentum} and energy \eqref{eqn:energy} in Figure
\ref{fig:nlw:dev}.
\begin{figure}[h]
  \centering
  \subfigure[][$q=0$ and $p=1$]{
    \includegraphics[
    width=0.30\textwidth]{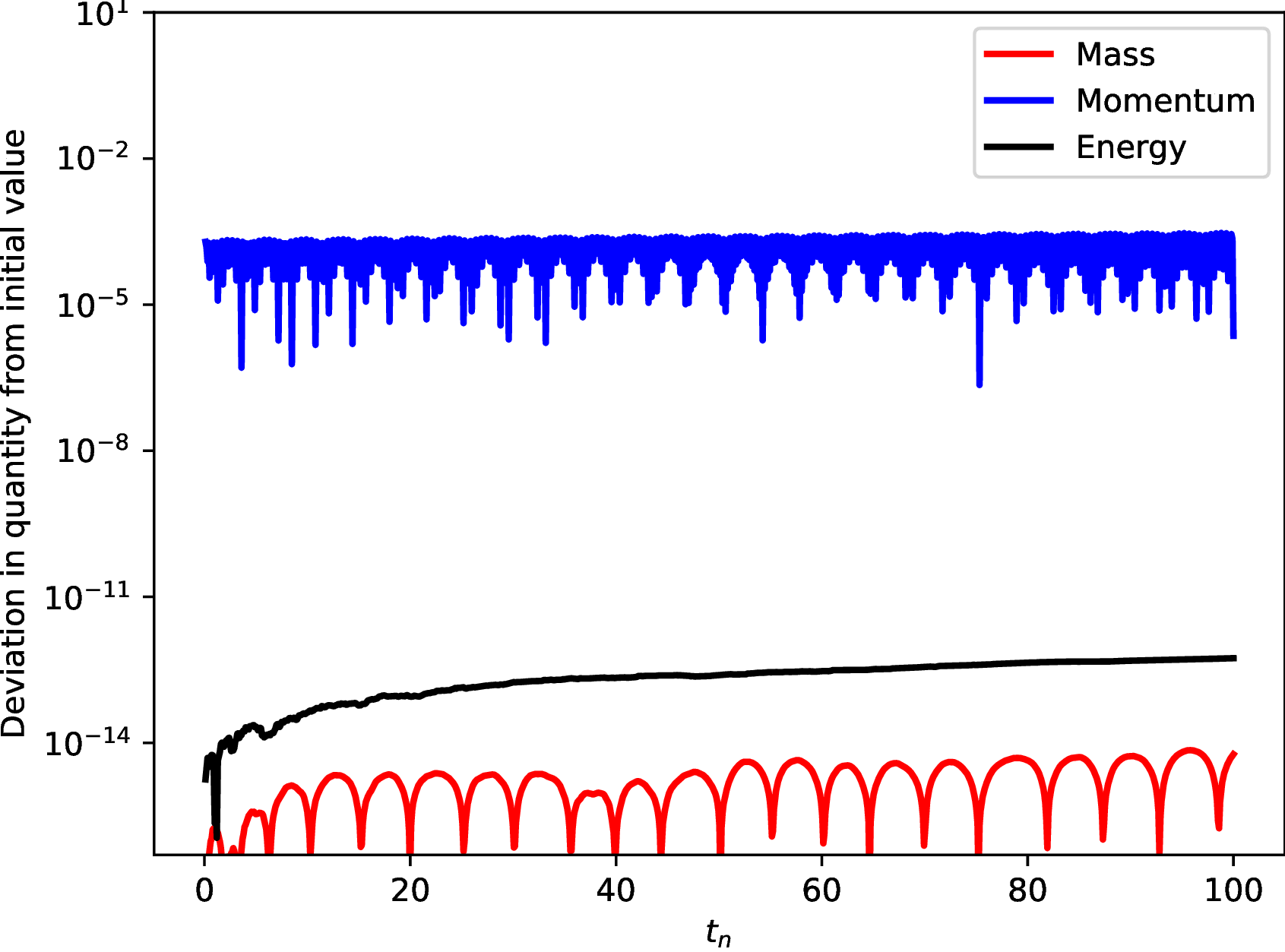}
  } \subfigure[][$q=0$ and $p=2$]{ \includegraphics[
    width=0.30\textwidth]{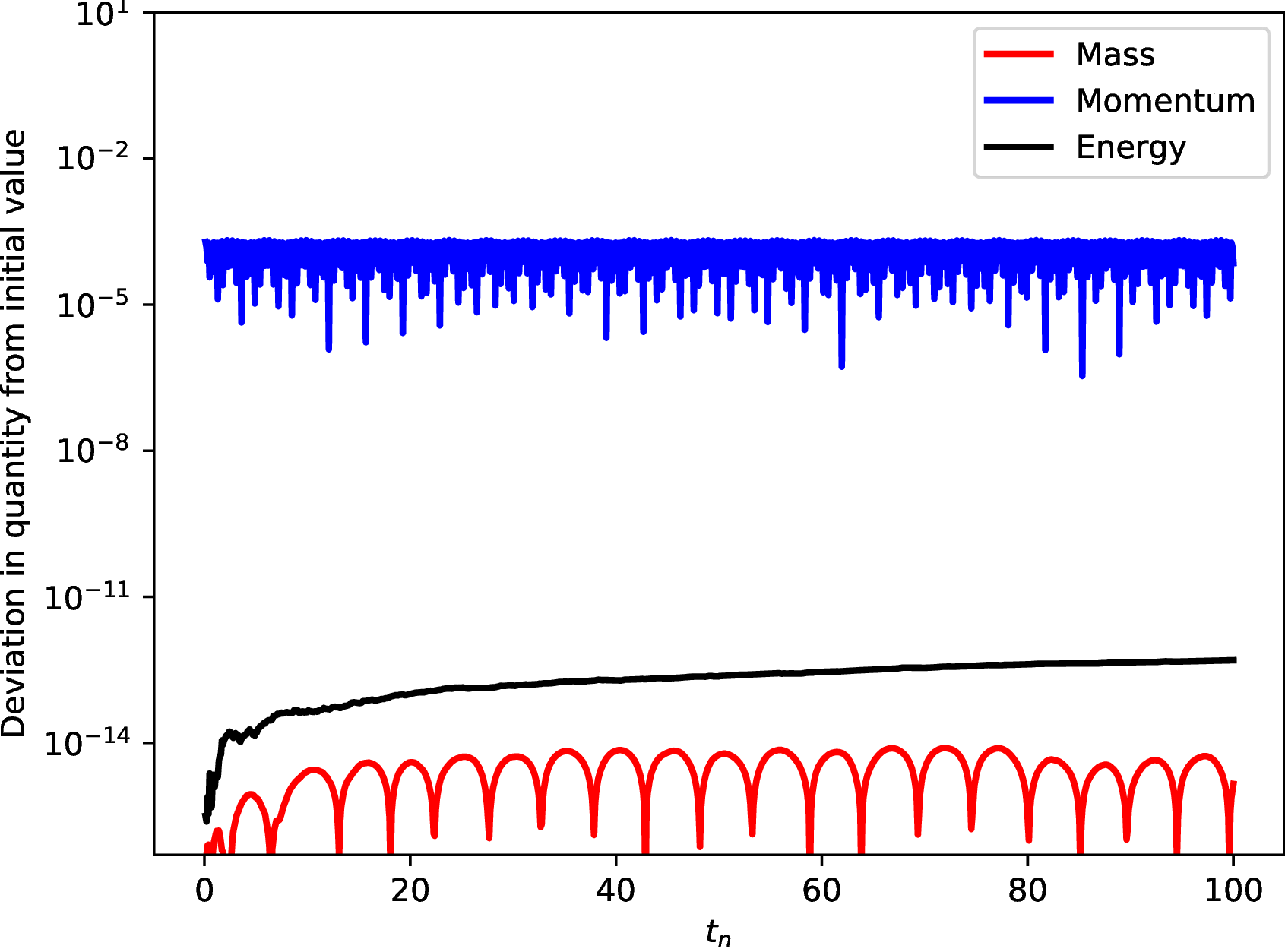}
  } \subfigure[][$q=0$ and $p=3$]{ \includegraphics[
    width=0.30\textwidth]{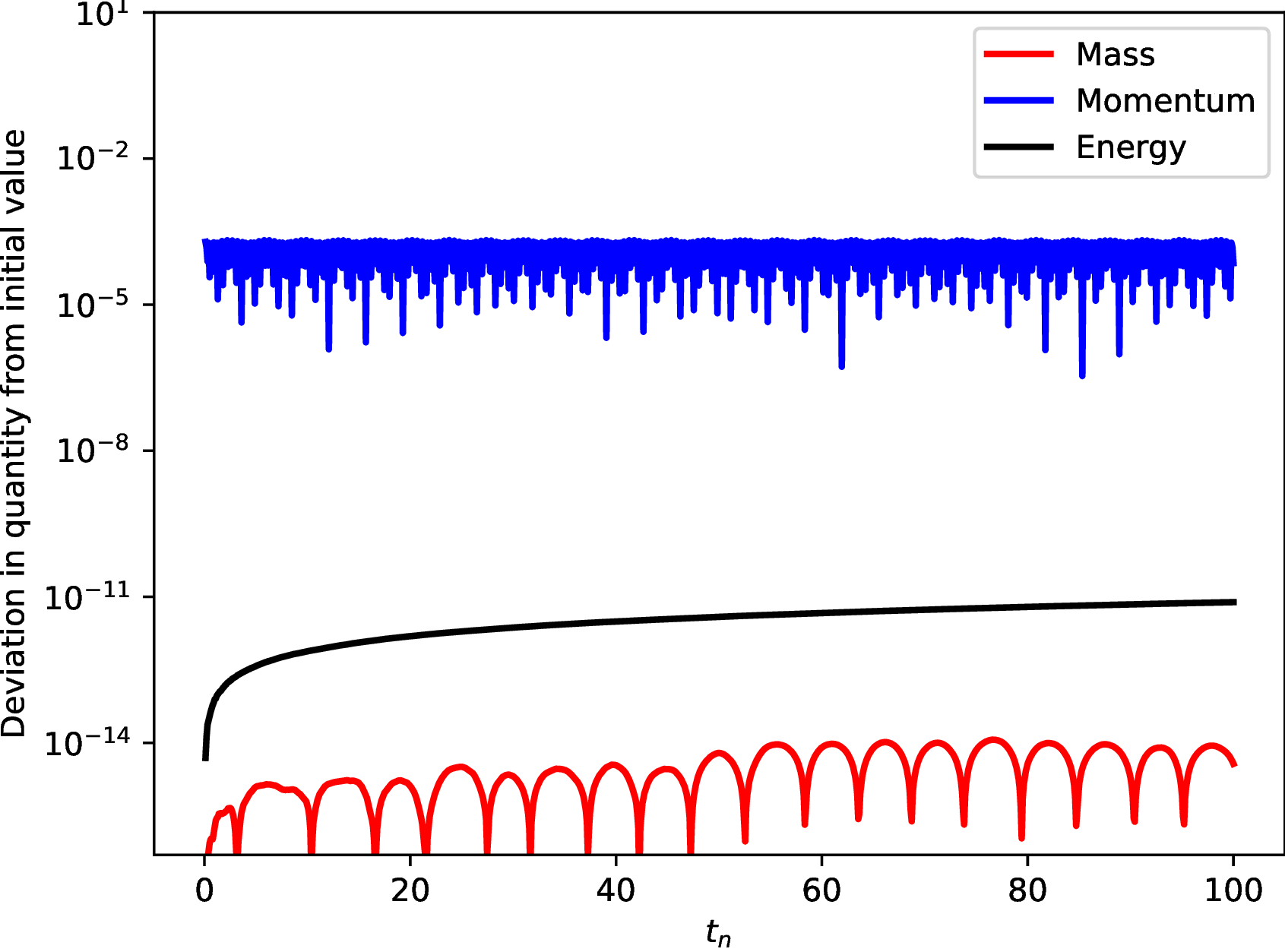}
  }
  \\
    \subfigure[][$q=1$ and $p=1$]{
    \includegraphics[
    width=0.30\textwidth]{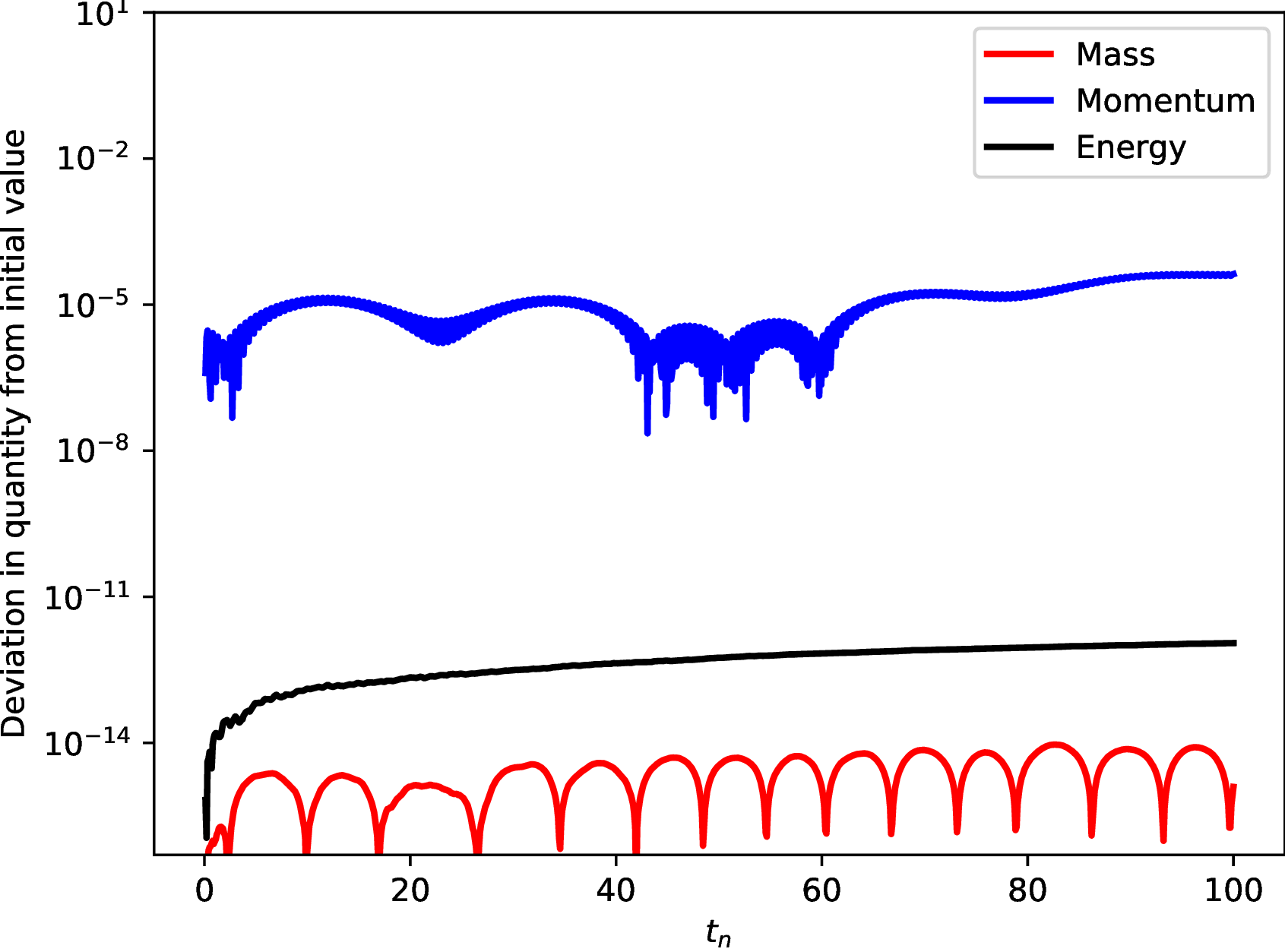}
  } \subfigure[][$q=1$ and $p=2$]{ \includegraphics[
    width=0.30\textwidth]{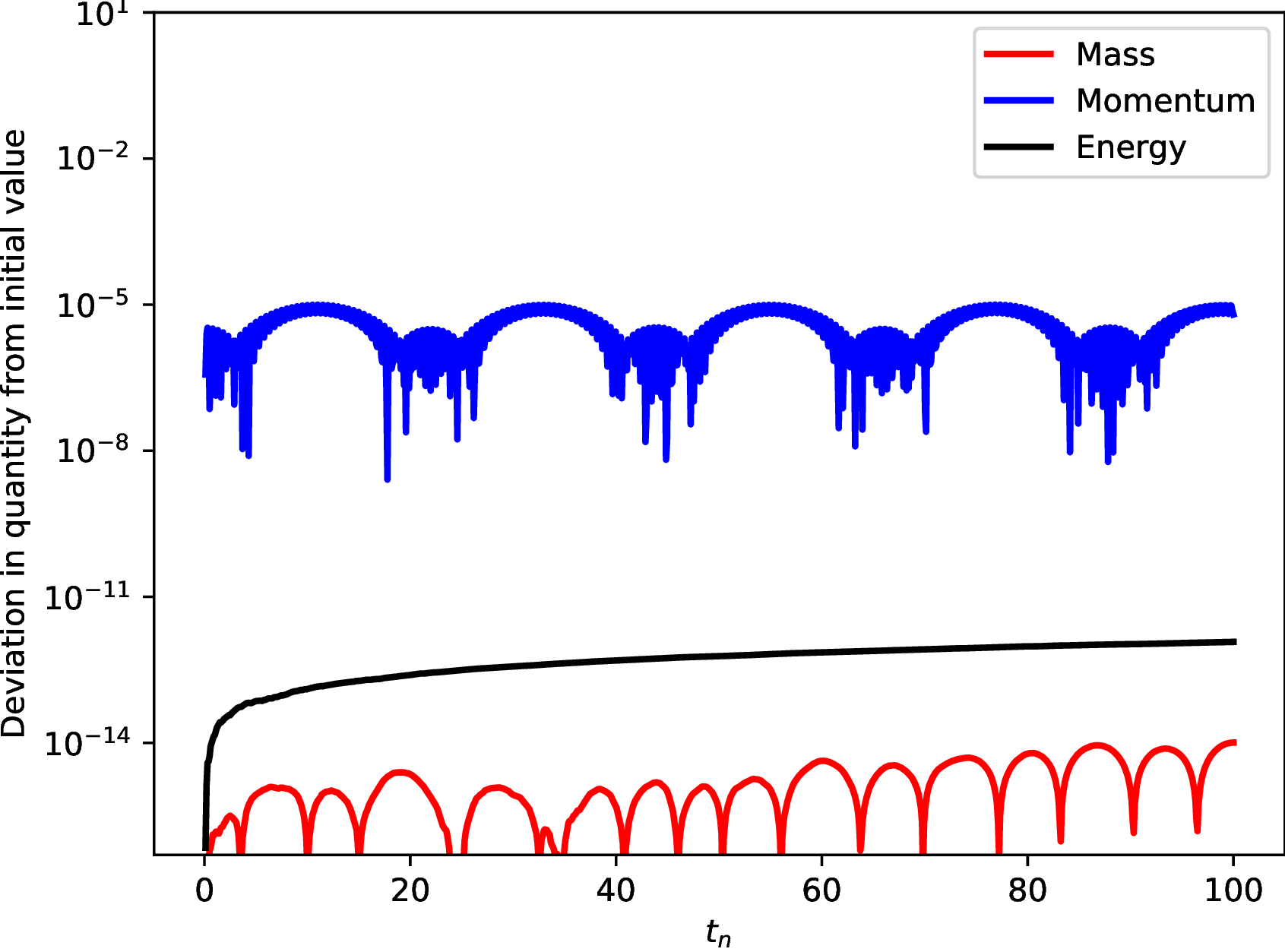}
  } \subfigure[][$q=1$ and $p=3$]{ \includegraphics[
    width=0.30\textwidth]{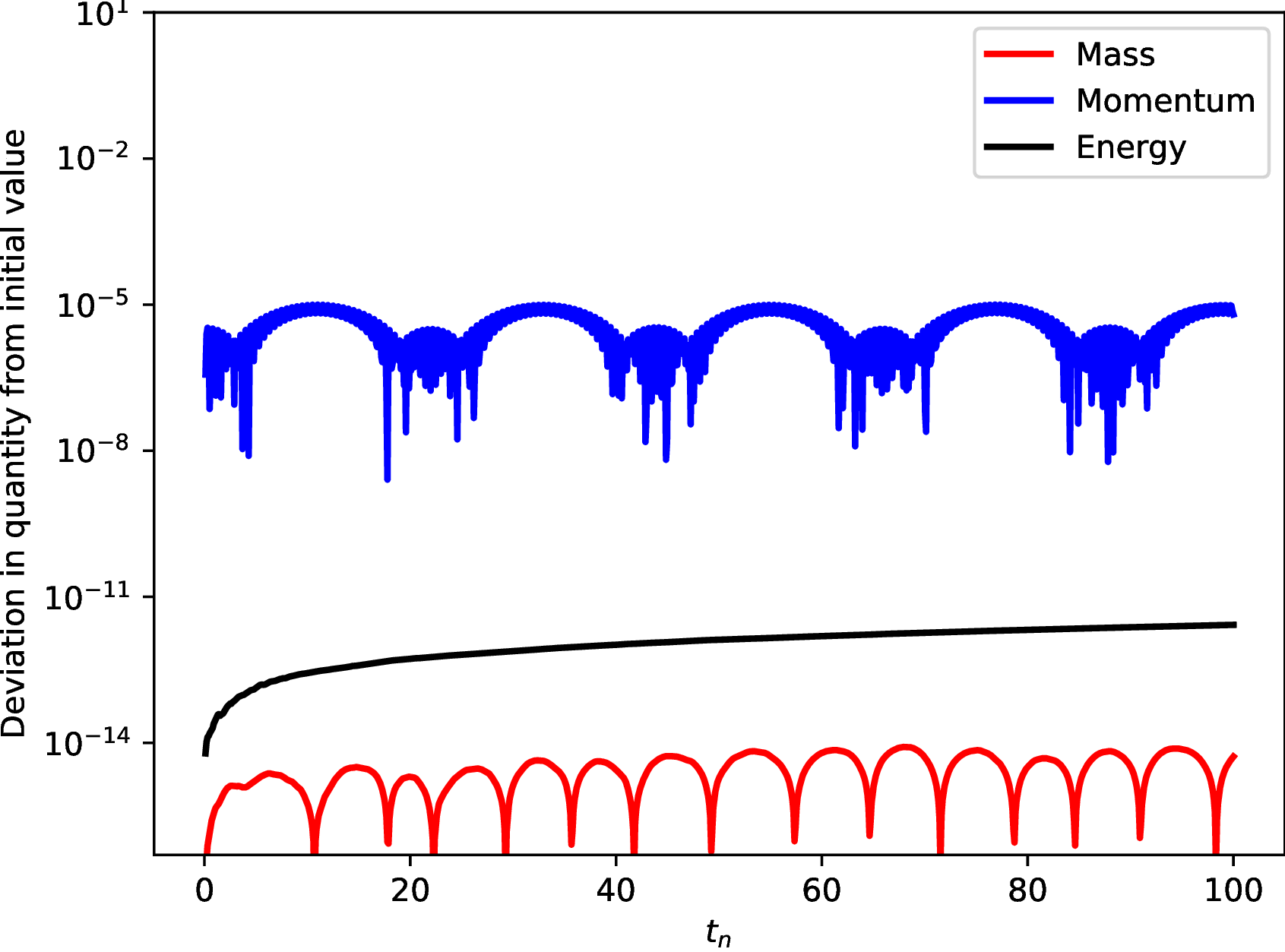}
  }
  \\
    \subfigure[][$q=2$ and $p=1$]{
    \includegraphics[
    width=0.30\textwidth]{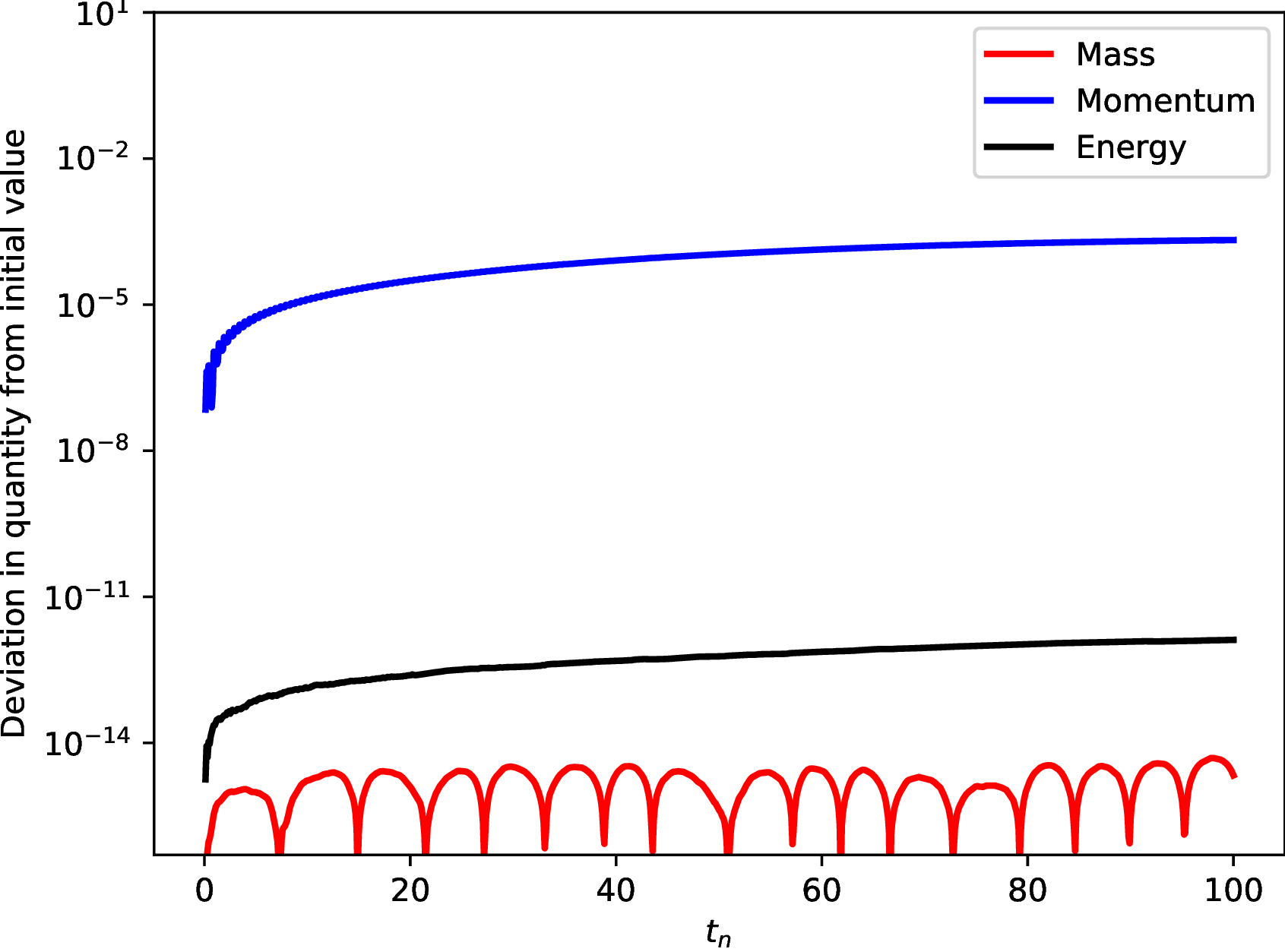}
  } \subfigure[][$q=2$ and $p=2$]{ \includegraphics[
    width=0.30\textwidth]{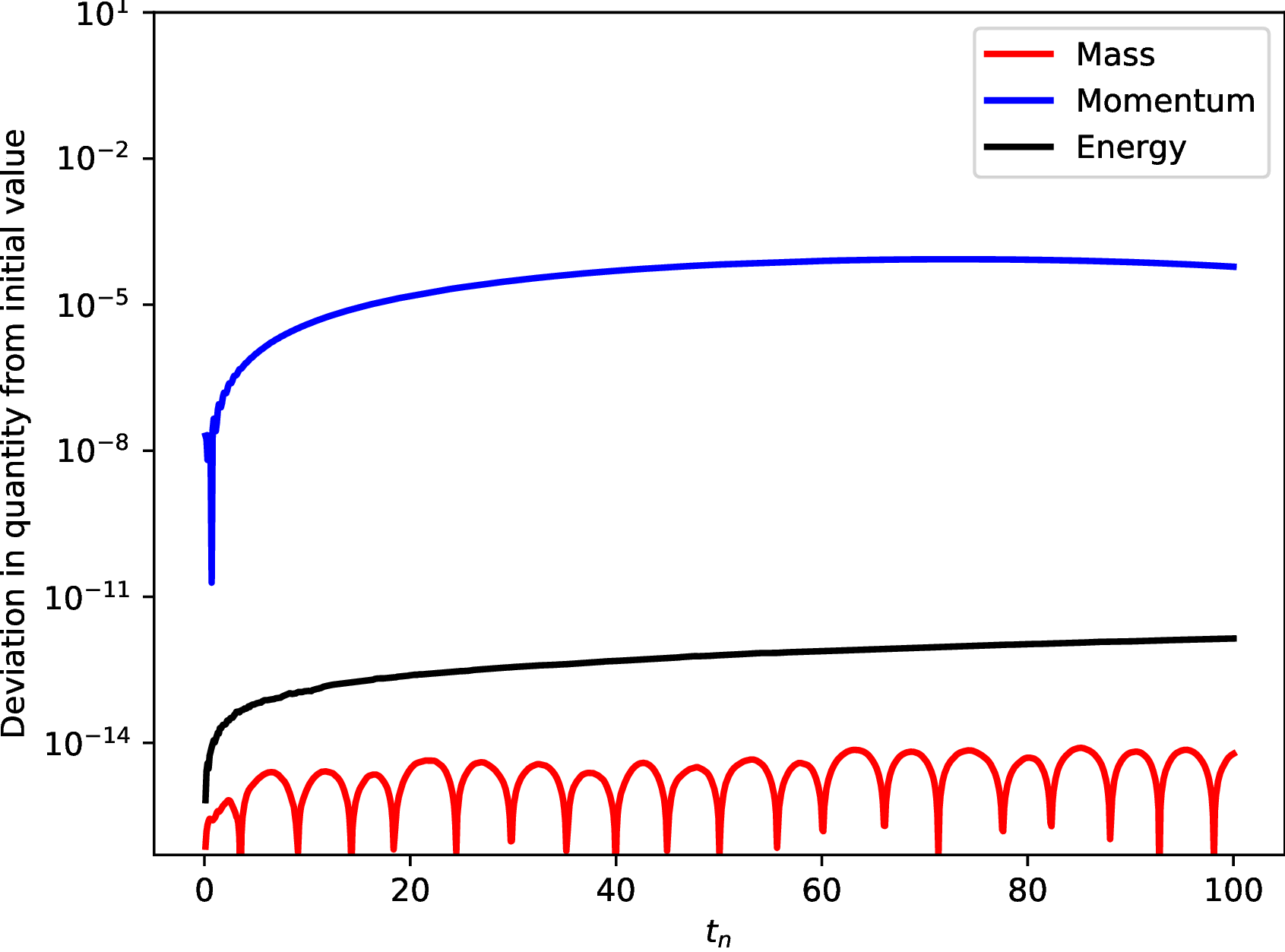}
  } \subfigure[][$q=2$ and $p=3$]{ \includegraphics[
    width=0.30\textwidth]{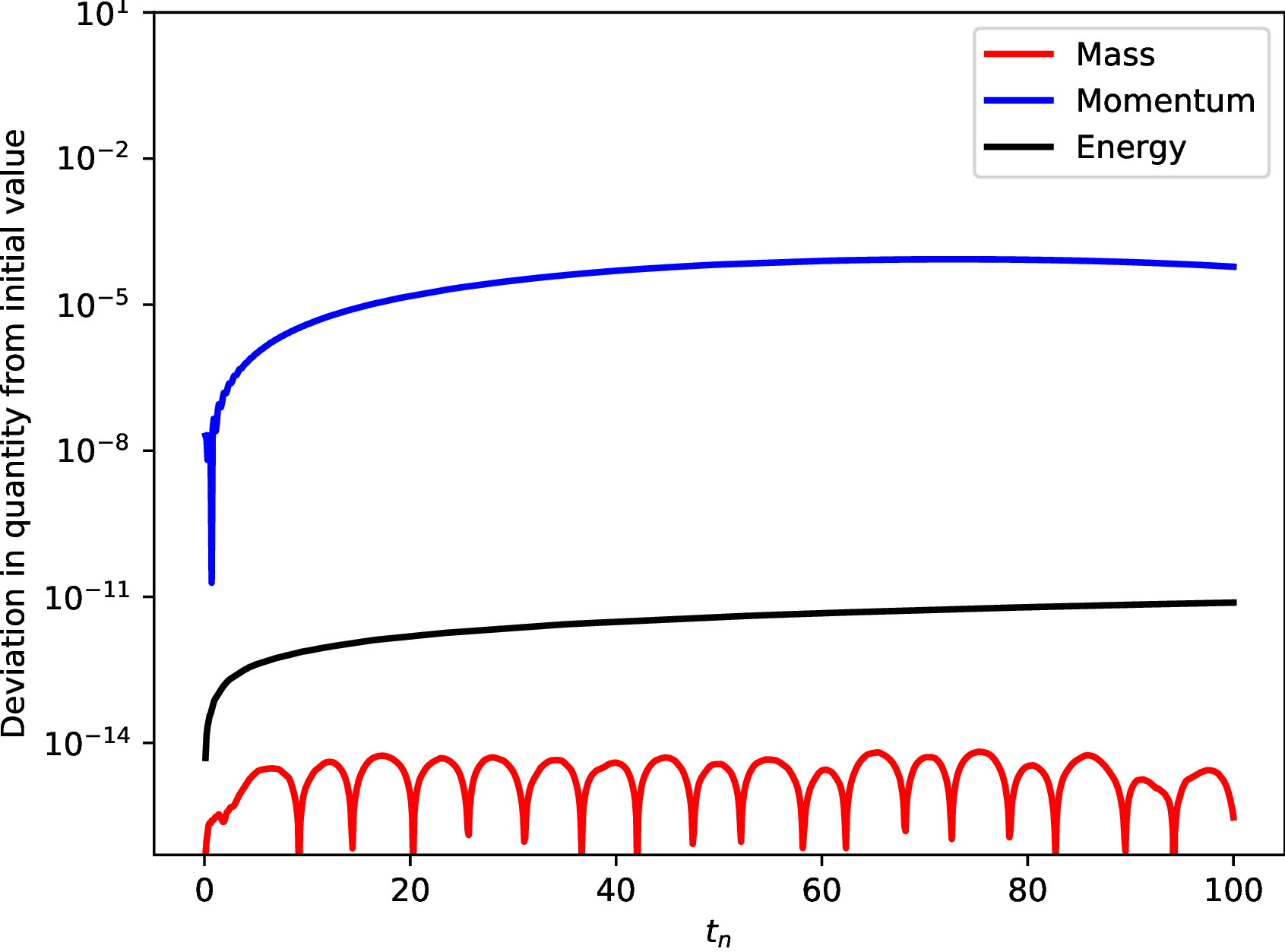}
  }

  \caption{The deviation in the conservation laws \eqref{eqn:nlwmass},
    \eqref{eqn:momentum} and \eqref{eqn:energy} for the
    spatially continuous finite element scheme \eqref{eqn:stlfem} for the
    nonlinear wave equation \eqref{eqn:nlw} initialised by the exact
    solution \eqref{eqn:lwexact} with varying temporal degree $q$ and
    spatial degree $p$. All simulations have uniform temporal and
    spatial elements with $\dt{}=0.1$ and $\dx{}=0.01$. We notice that
    the energy density is exactly preserved locally over time and
    globally over space. The deviation in momentum, while not exactly
    preserved, remains bounded over long time. \label{fig:nlw:dev}}
\end{figure}
We observe that the energy density is indeed conserved locally in
time and globally in space. Over long time errors below solver
tolerance propagate in the energy, with this propagation becoming more
significant for higher polynomial degree. For all simulations the
deviation in mass does not propagate over long time, and the momentum
remains bounded deviating by approximately $10^{-5}$.

Through a small perturbation to the finite element approximation,
namely that made within Remark \ref{rem:mclaw} with \eqref{eqn:mclaw},
we also investigate the numerical deviation in conservation laws in
Figure \ref{fig:nlw:devm}.
\begin{figure}[h]
  \centering
  \subfigure[][$q=0$ and $p=1$]{
    \includegraphics[
    width=0.30\textwidth]{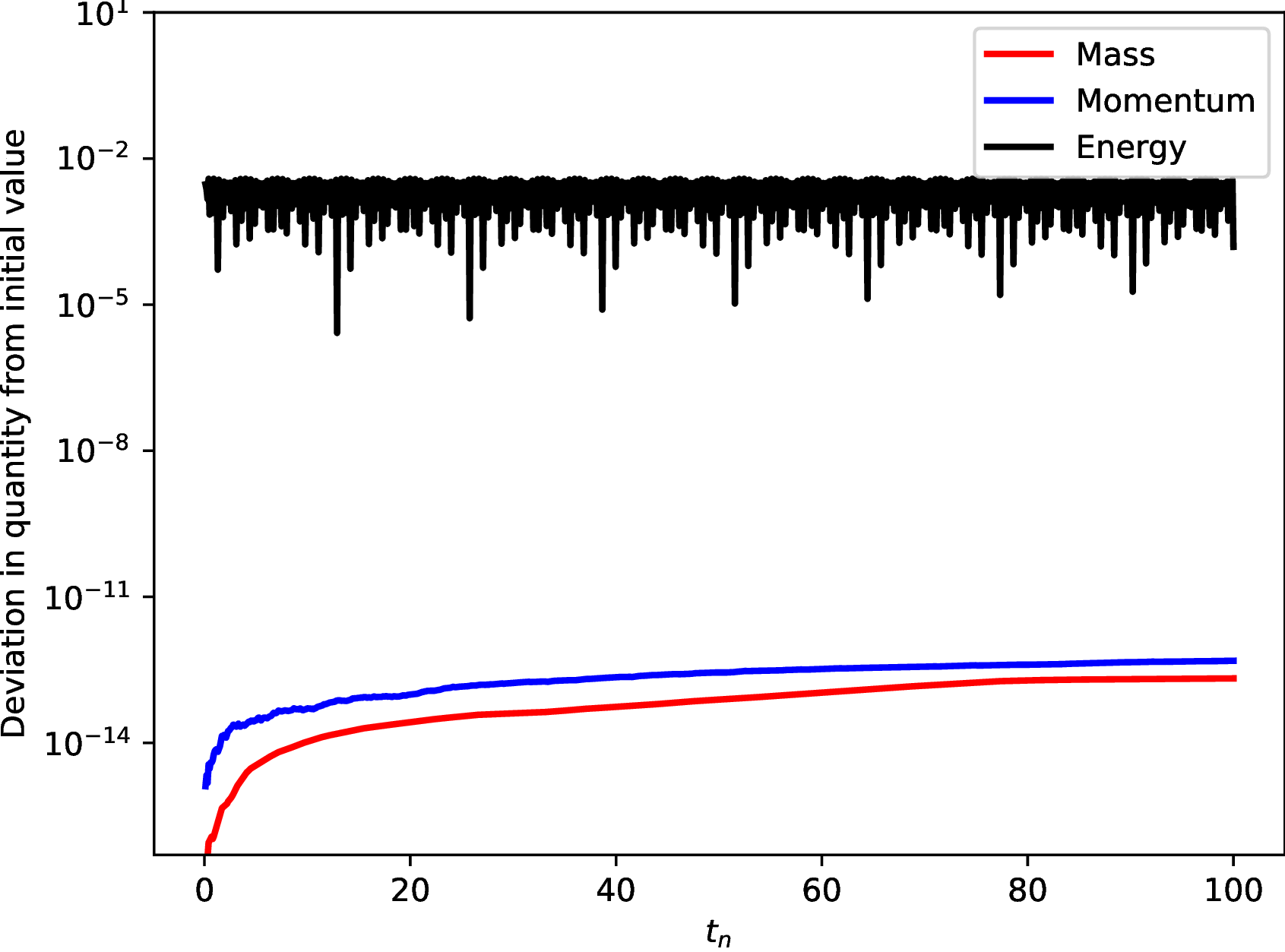}
  } \subfigure[][$q=0$ and $p=2$]{ \includegraphics[
    width=0.30\textwidth]{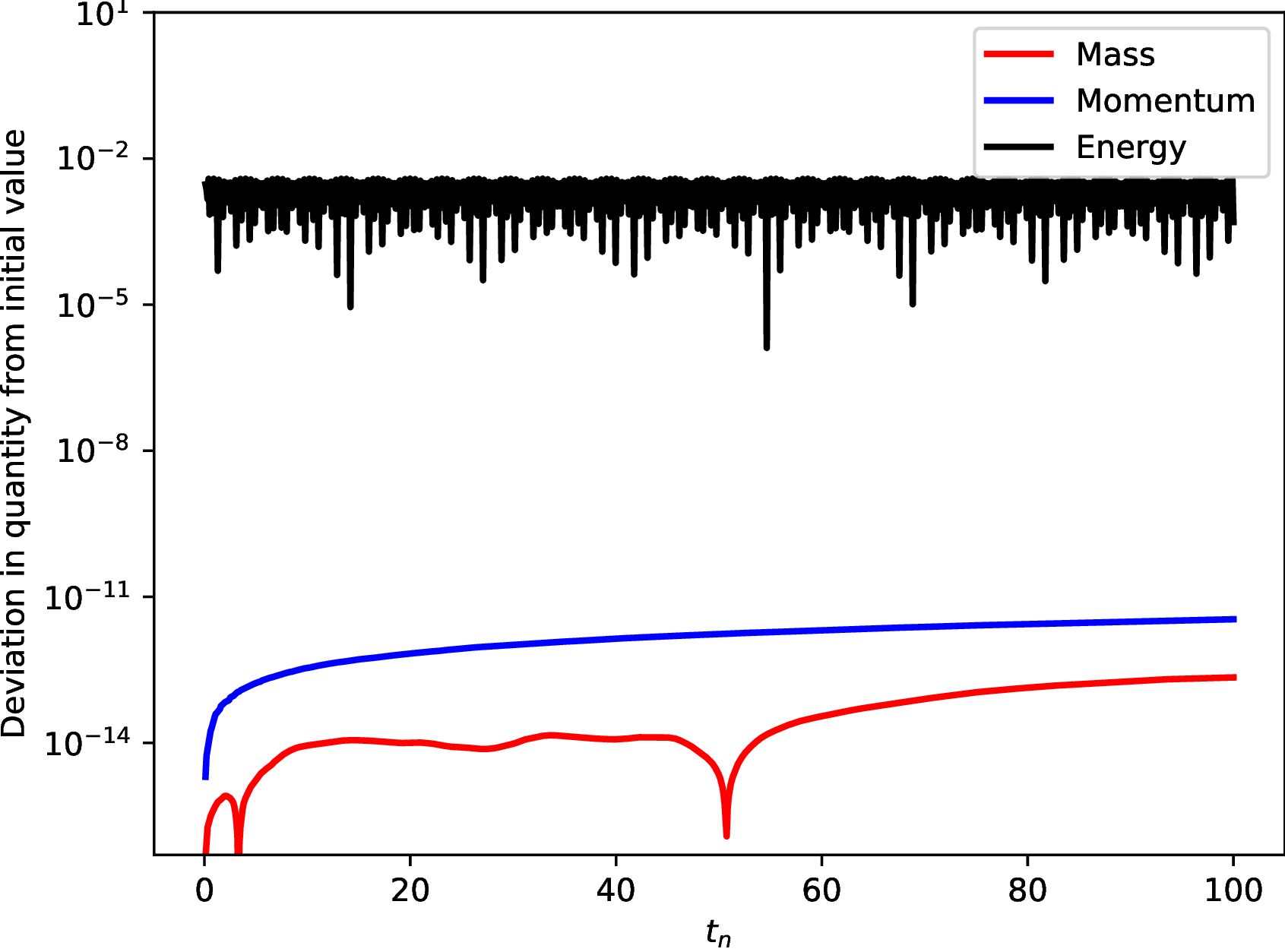}
  } \subfigure[][$q=0$ and $p=3$]{ \includegraphics[
    width=0.30\textwidth]{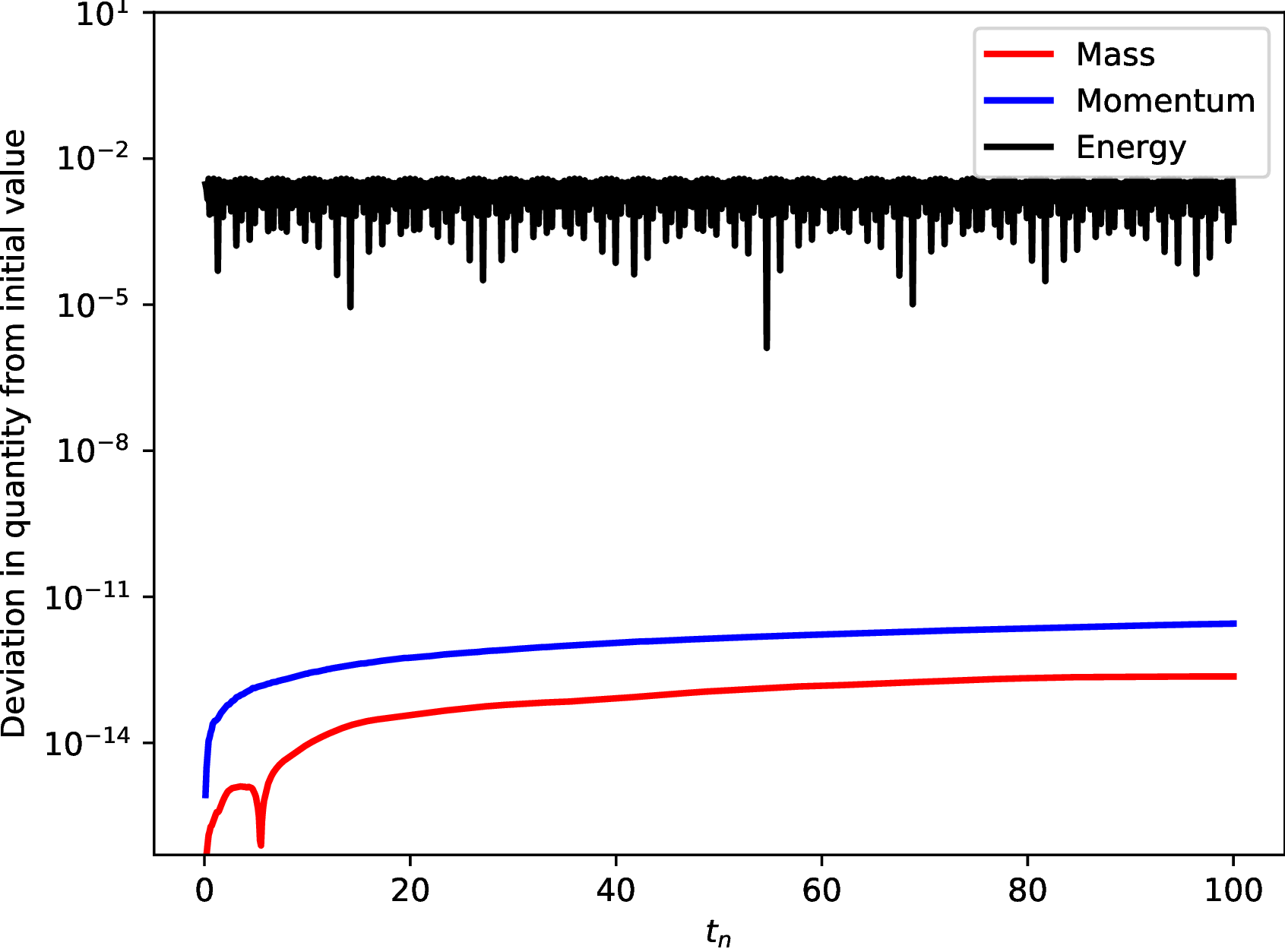}
  }
  \\
    \subfigure[][$q=1$ and $p=1$]{
    \includegraphics[
    width=0.30\textwidth]{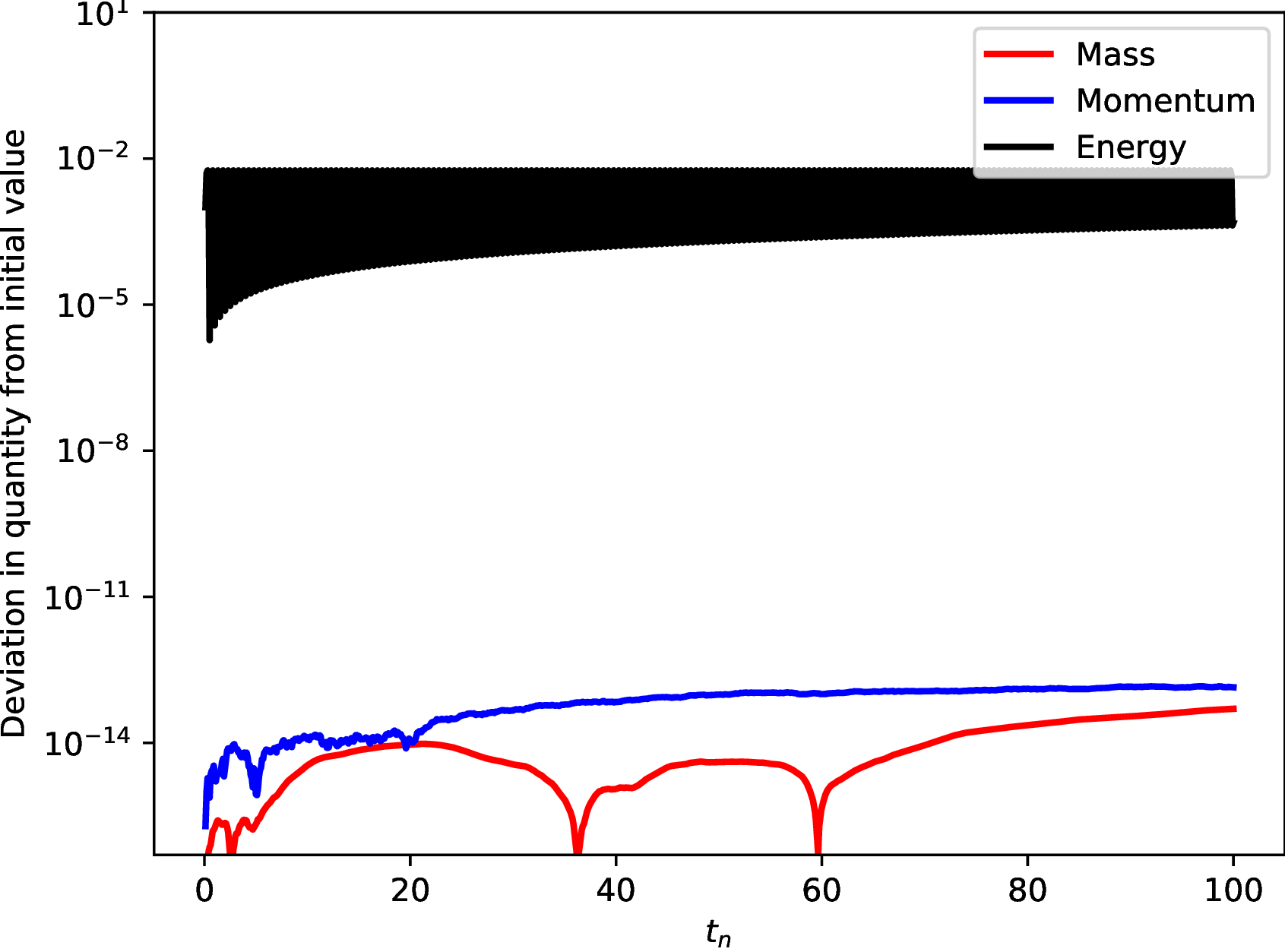}
  } \subfigure[][$q=1$ and $p=2$]{ \includegraphics[
    width=0.30\textwidth]{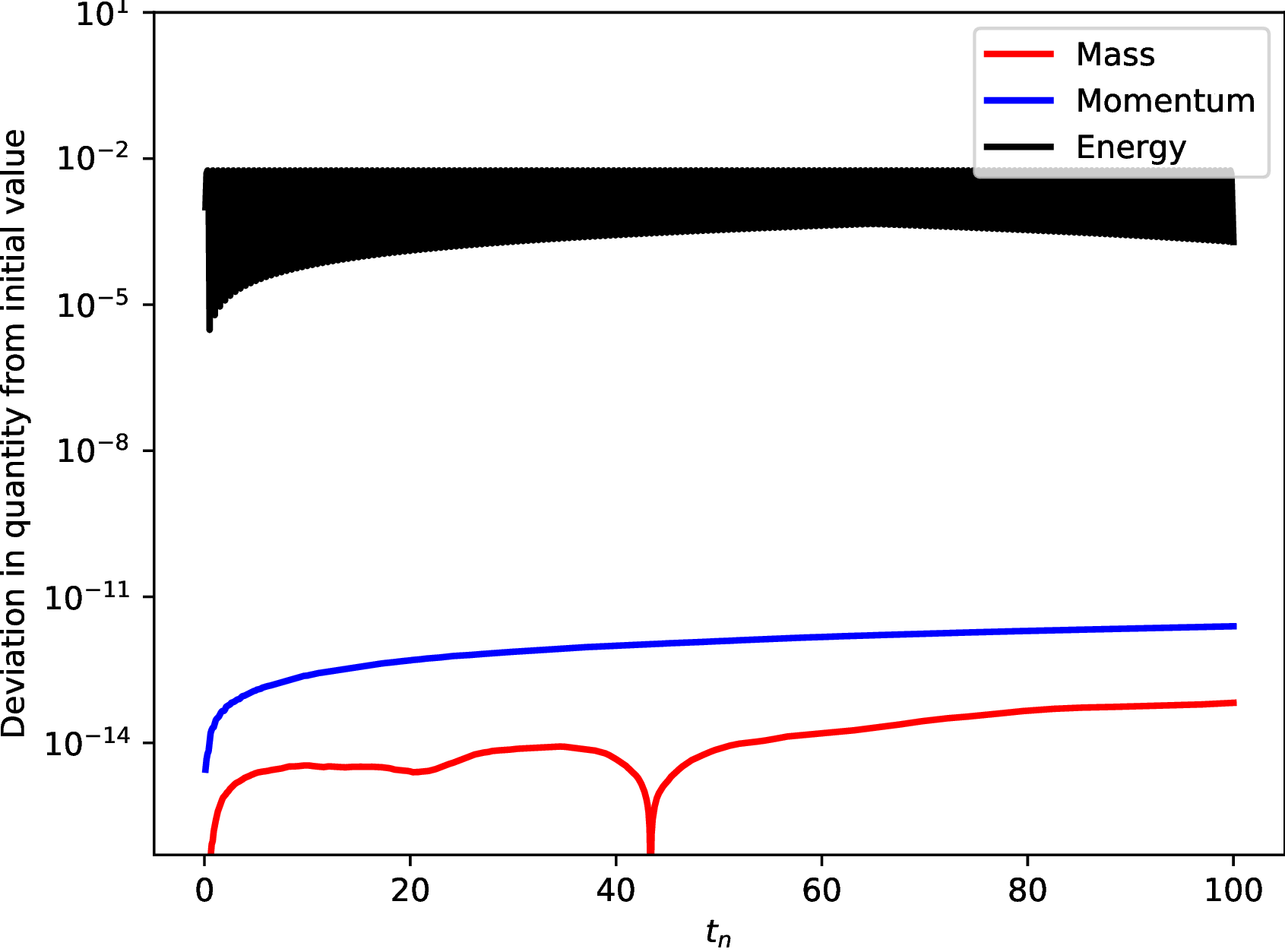}
  } \subfigure[][$q=1$ and $p=3$]{ \includegraphics[
    width=0.30\textwidth]{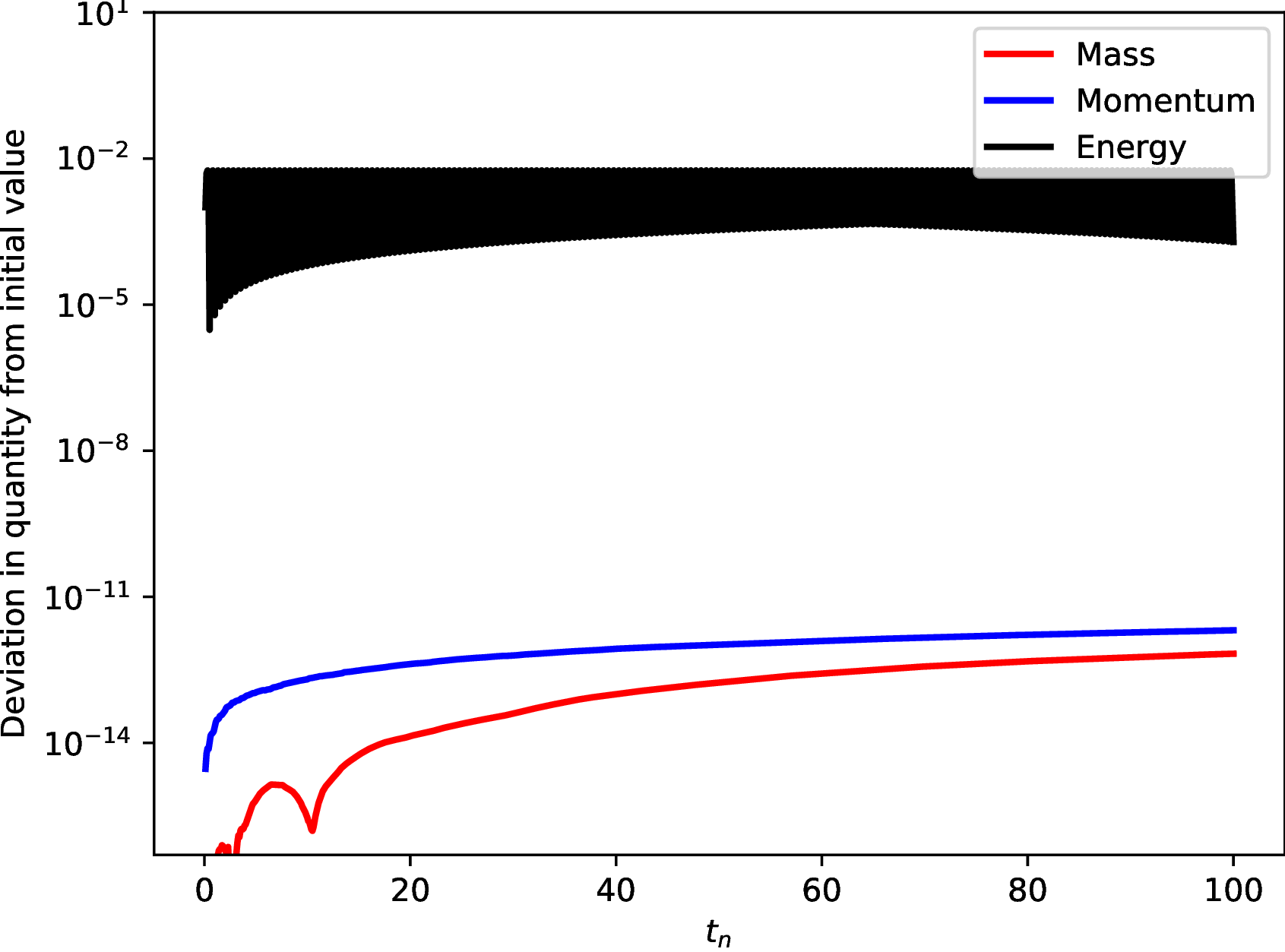}
  }
  \\
    \subfigure[][$q=2$ and $p=1$]{
    \includegraphics[
    width=0.30\textwidth]{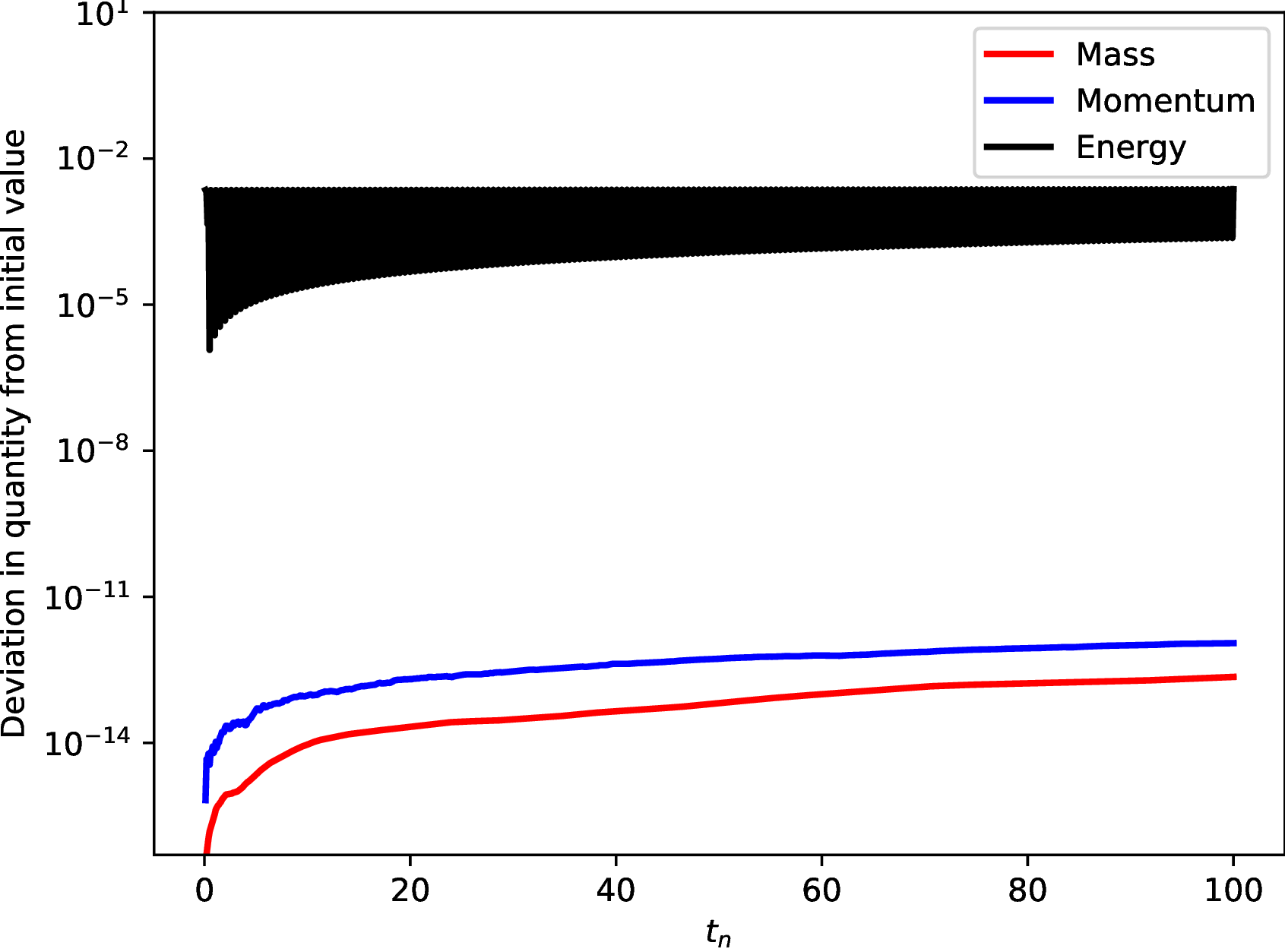}
  } \subfigure[][$q=2$ and $p=2$]{ \includegraphics[
    width=0.30\textwidth]{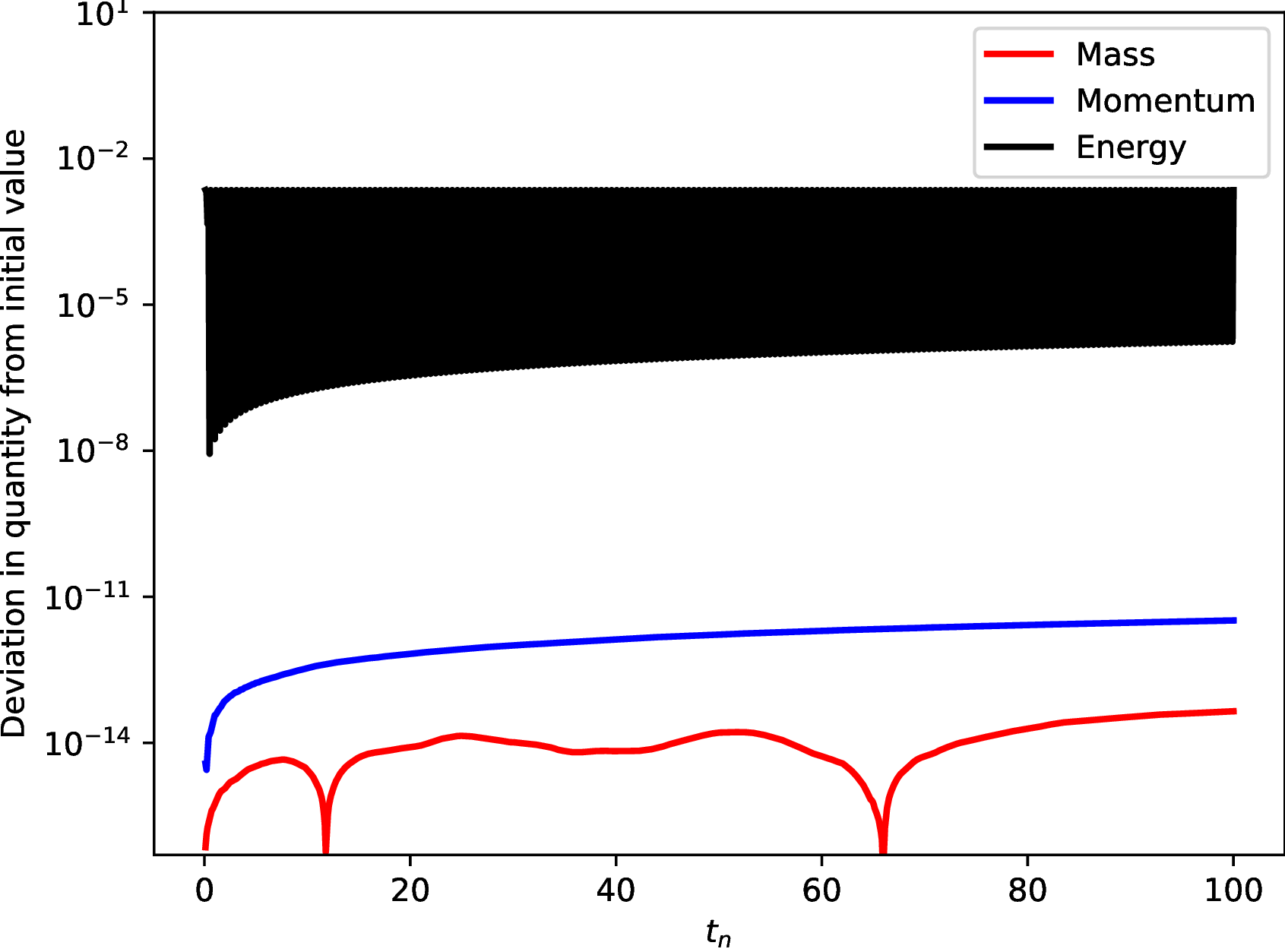}
  } \subfigure[][$q=2$ and $p=3$]{ \includegraphics[
    width=0.30\textwidth]{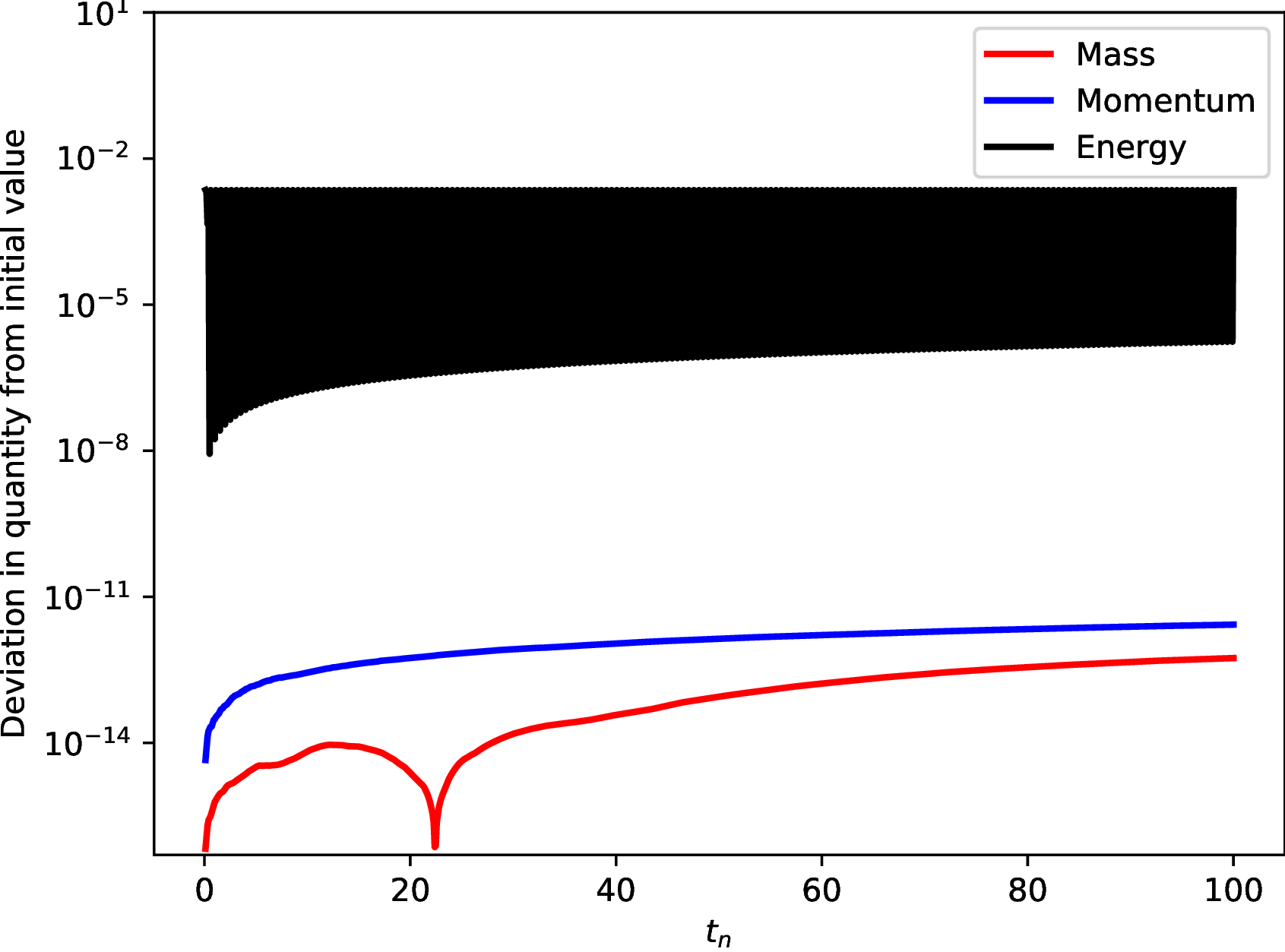}
  }

  \caption{The deviation in the conservation laws \eqref{eqn:nlwmass},
    \eqref{eqn:momentum} and \eqref{eqn:energy} for the alternative
    finite element scheme \eqref{eqn:stlfem} for the
    nonlinear wave equation \eqref{eqn:nlw} initialised by the exact
    solution \eqref{eqn:lwexact} with varying temporal degree $q$ and
    spatial degree $p$. All simulations have uniform temporal and
    spatial elements with $\dt{}=0.1$ and $\dx{}=0.01$. We notice that
    the momentum density is exactly preserved locally over time and
    globally over space. The deviation in energy, while not exactly
    preserved, remains bounded over long time. Comparing against
    Figure \ref{fig:nlw:dev} we notice the deviation in energy here is
    larger than the deviation in momentum for our primary
    discretisation \eqref{eqn:stfem}, and does not globally decrease
    as we increase polynomial degrees. \label{fig:nlw:devm}}
\end{figure}
We observe that while the modified scheme \eqref{eqn:mstfem} does
preserve the mass and momentum density locally over time and globally
over space, the deviation of both propagates over long time leading to
significant errors. In addition, the deviation in energy remains
globally by $10^{-2}$ regardless of the polynomial degree of the
approximation.

\subsection{Test 3: The nonlinear Schr\"odinger equation}

We may explicitly express the multisymplectic formulation of the NLS
equation, given in Example \ref{ex:nls}, as
\begin{equation} \label{eqn:nls}
  \begin{split}
    - v_t + p_x + \frac12 u\bc{u^2+v^2} & = 0 \\
    u_t + q_x  + \frac12 v\bc{u^2+v^2}  & = 0 \\
    - u_x + p & = 0 \\
    - v_x + q & = 0
    ,
  \end{split}
\end{equation}
see for example \cite{islas01gif}. Throughout our numerical study, we
initialise simulations of \eqref{eqn:nls} with the soliton solution
\begin{equation} \label{eqn:nlsexact}
  \begin{split}
    u(t,x) & = 2\cos{t}\sech{x} \\
    v(t,x) & = 2\sin{t}\sech{x} \\
    p(t,x) & = -2\cos{t}\sinh{x}\sech{x}^2 \\
    q(t,x) & = -2\sin{t}\sinh{x}\sech{x}^2
    ,
  \end{split}
\end{equation}
over the stretched periodic spatial domain of $x\in[0,40)$. We note
that \eqref{eqn:nlsexact} is defined over $\mathbb{R}$, however, we
may accurately approximate it on a periodic domain as it is a
travelling wave.


Similarly to the linear wave equation, we may benchmark spatially
continuous scheme for the nonlinear Schr{\"o}dinger equation for
various polynomial degree, as can be seen in Figure \ref{fig:nls:eoc}.
\begin{figure}[h]
  \centering
  \subfigure[][$q=0$ and $p=1$]{
    \includegraphics[
    width=0.30\textwidth]{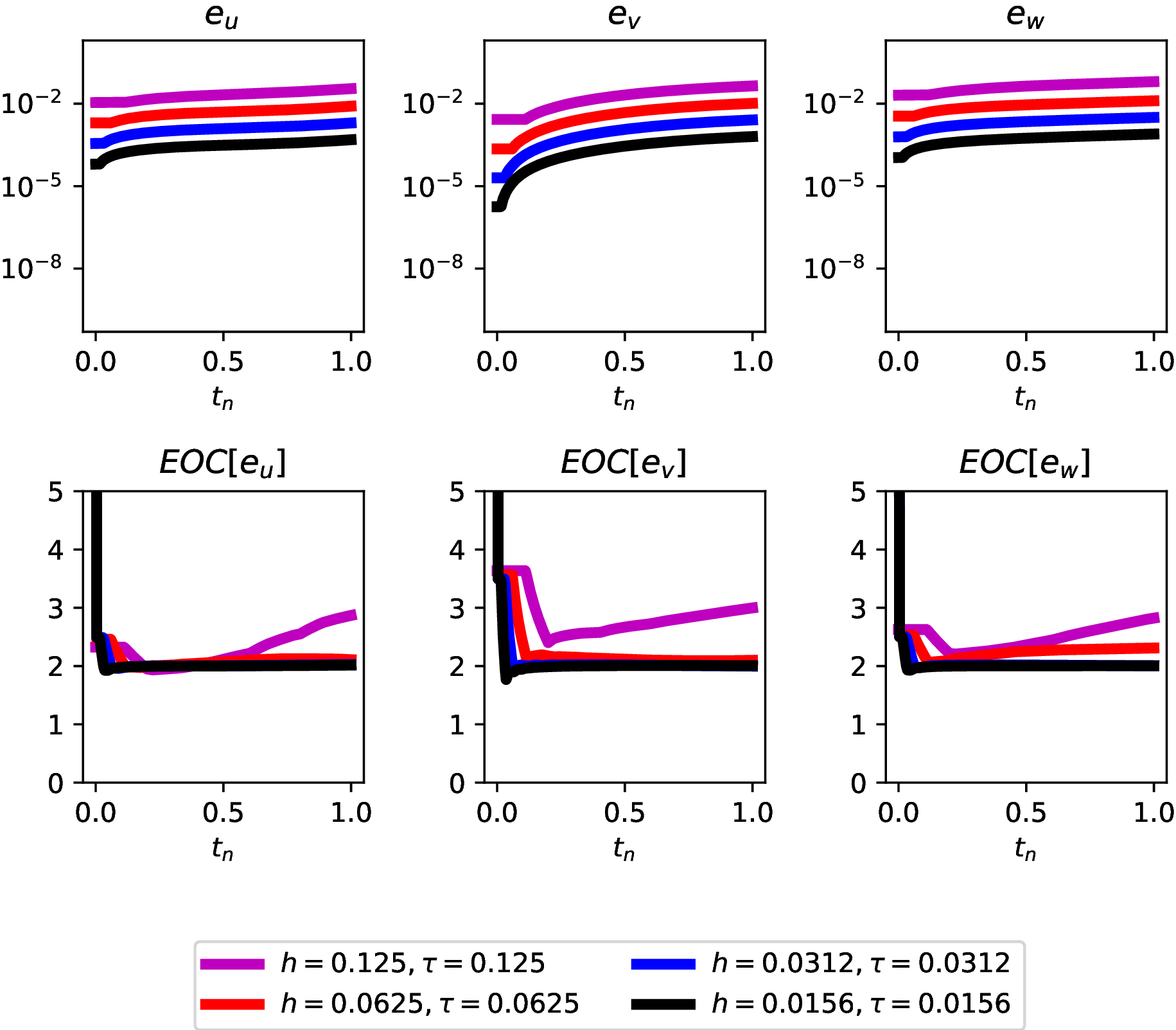}
  } \subfigure[][$q=0$ and $p=2$]{ \includegraphics[
    width=0.30\textwidth]{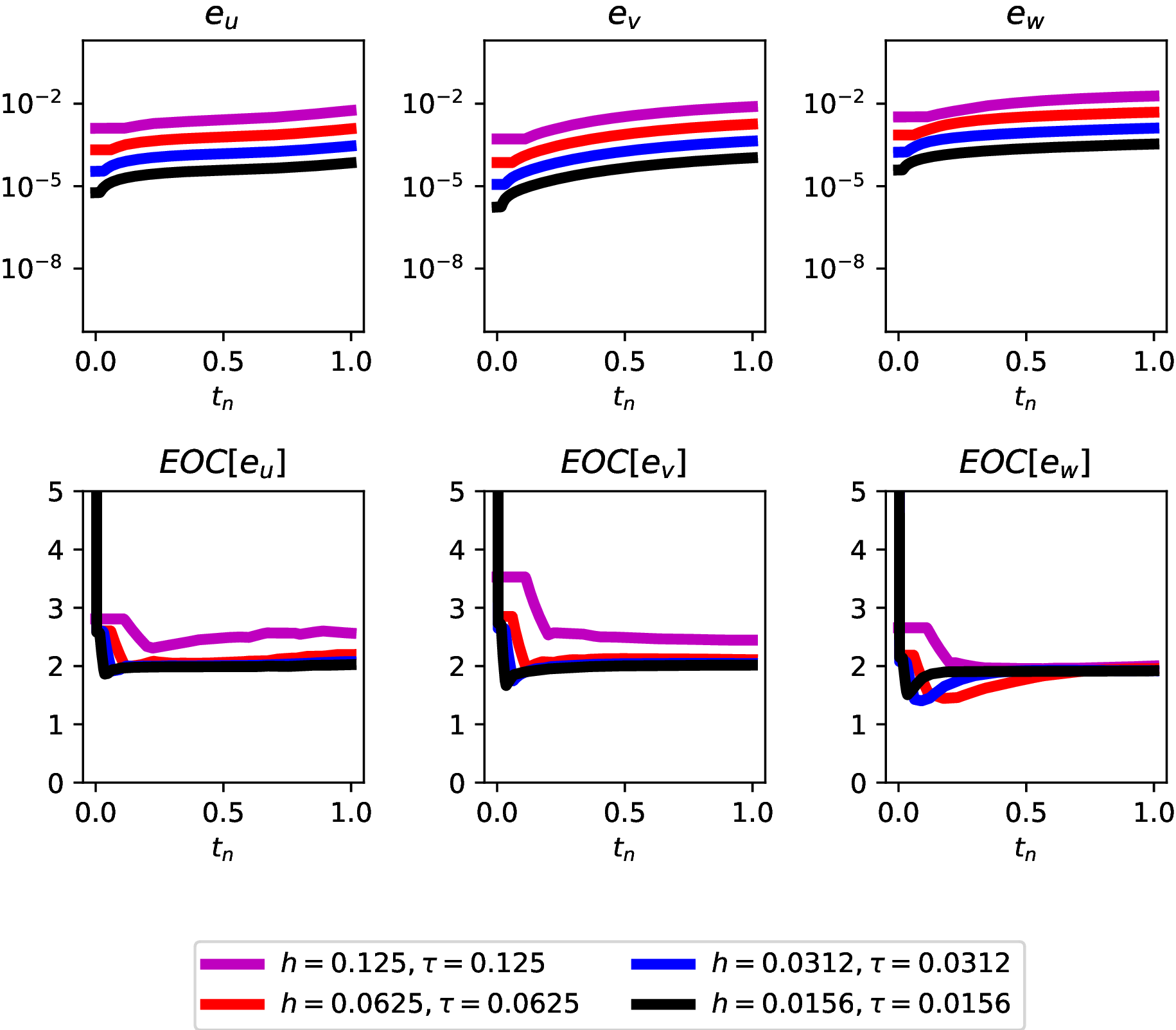}
  } \subfigure[][$q=0$ and $p=3$]{ \includegraphics[
    width=0.30\textwidth]{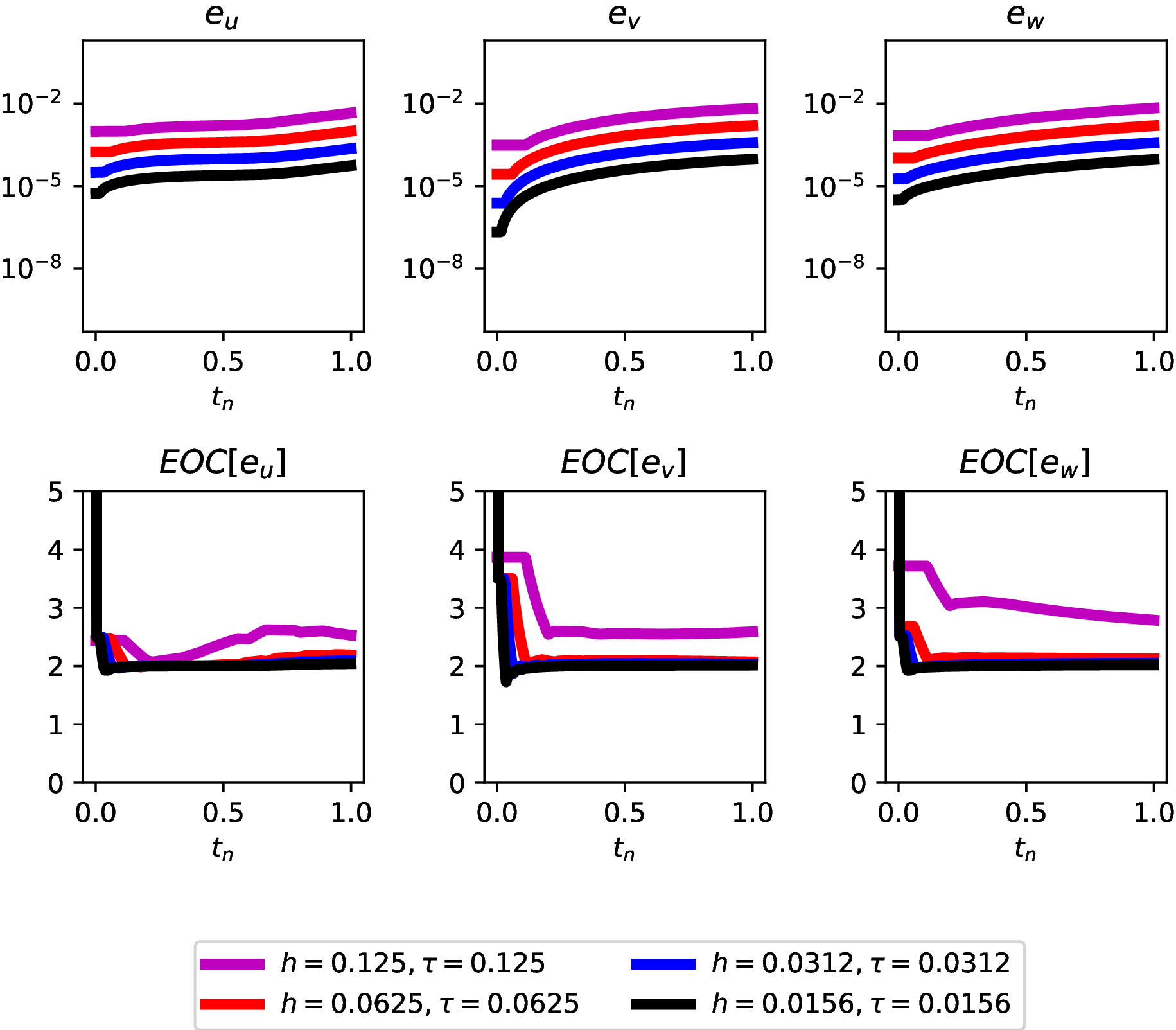}
  }
  \\
    \subfigure[][$q=1$ and $p=1$]{
    \includegraphics[
    width=0.30\textwidth]{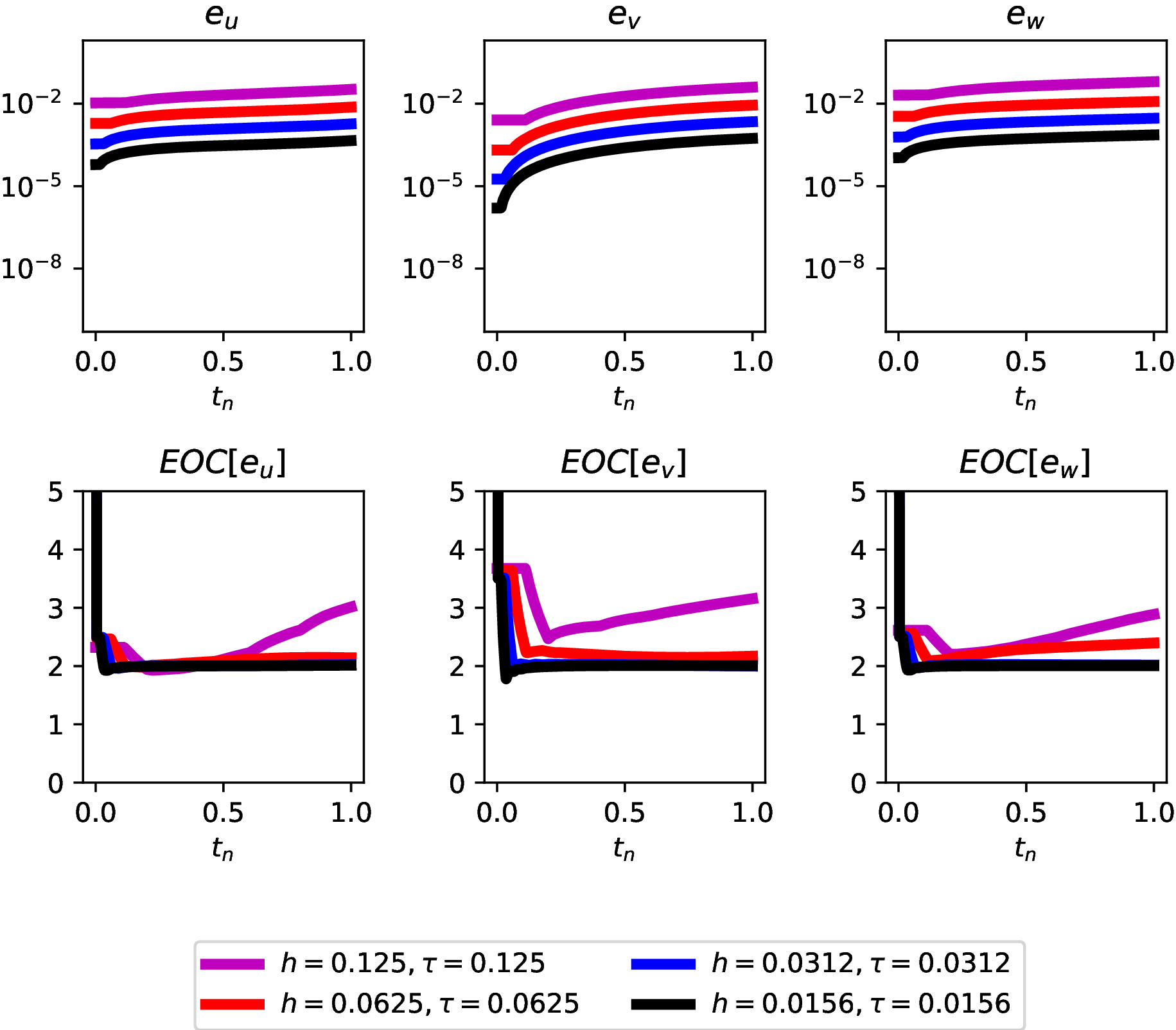}
  } \subfigure[][$q=1$ and $p=2$]{ \includegraphics[
    width=0.30\textwidth]{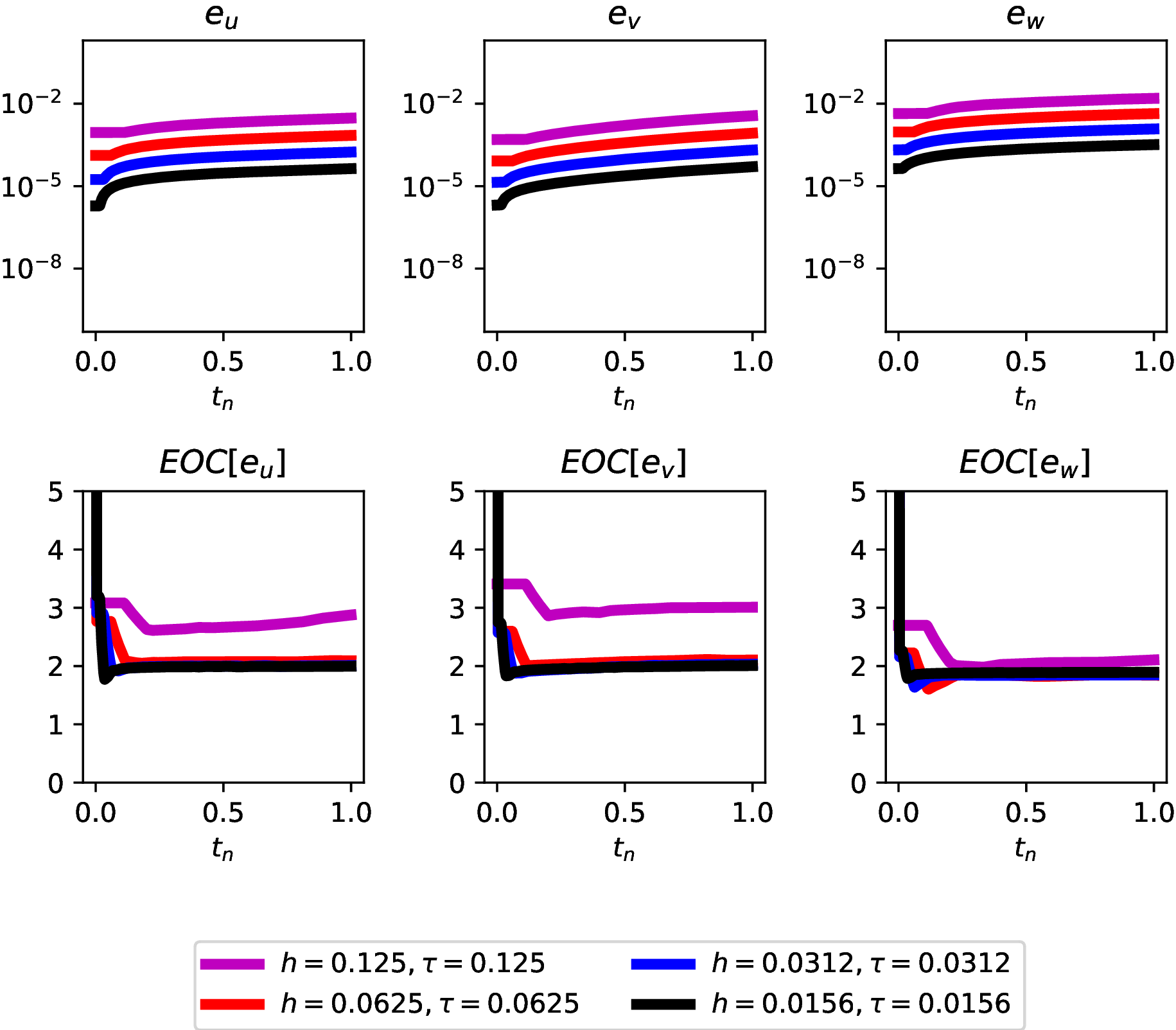}
  } \subfigure[][$q=1$ and $p=3$]{ \includegraphics[
    width=0.30\textwidth]{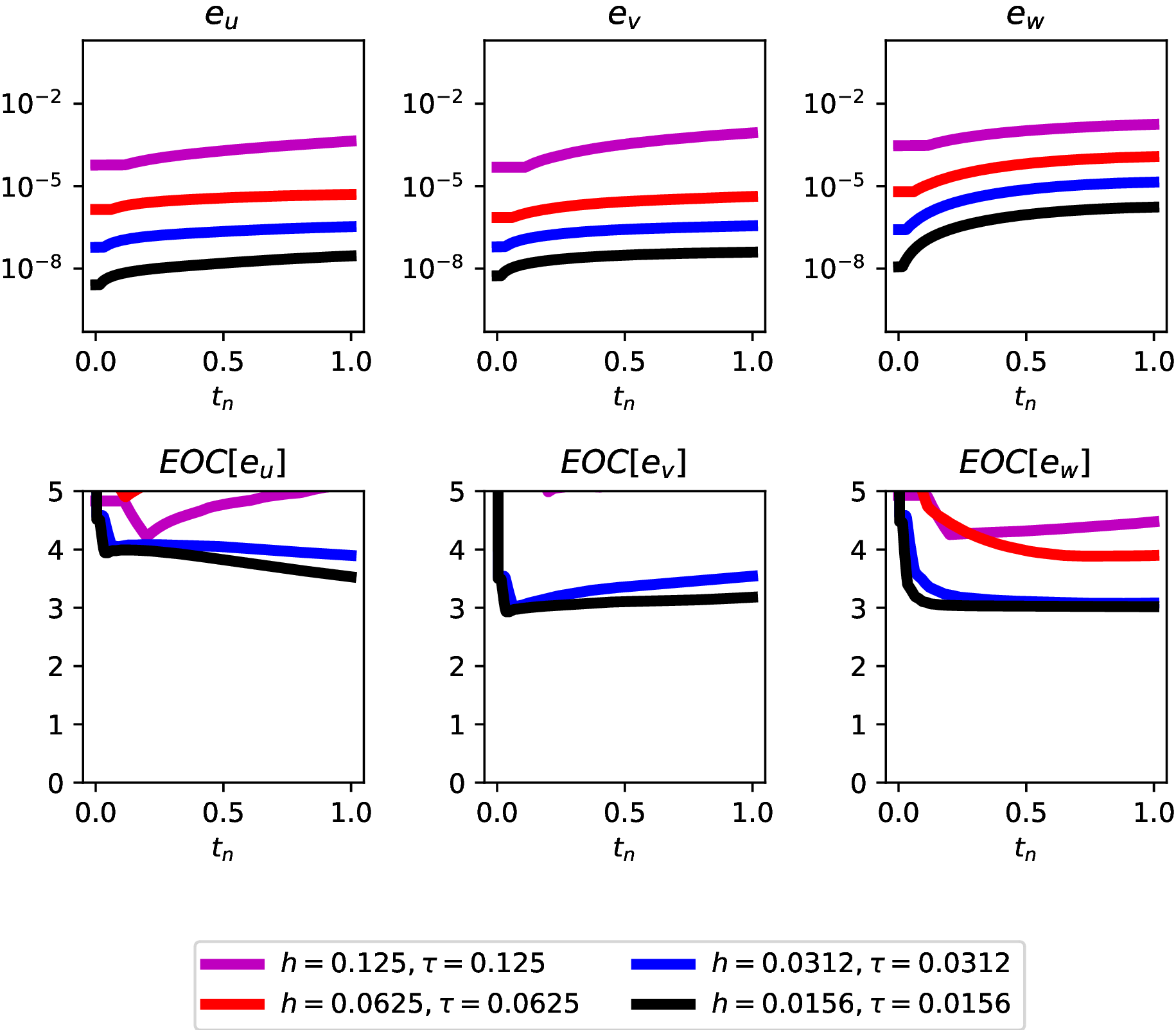}
  }
  \\
    \subfigure[][$q=2$ and $p=1$]{
    \includegraphics[
    width=0.30\textwidth]{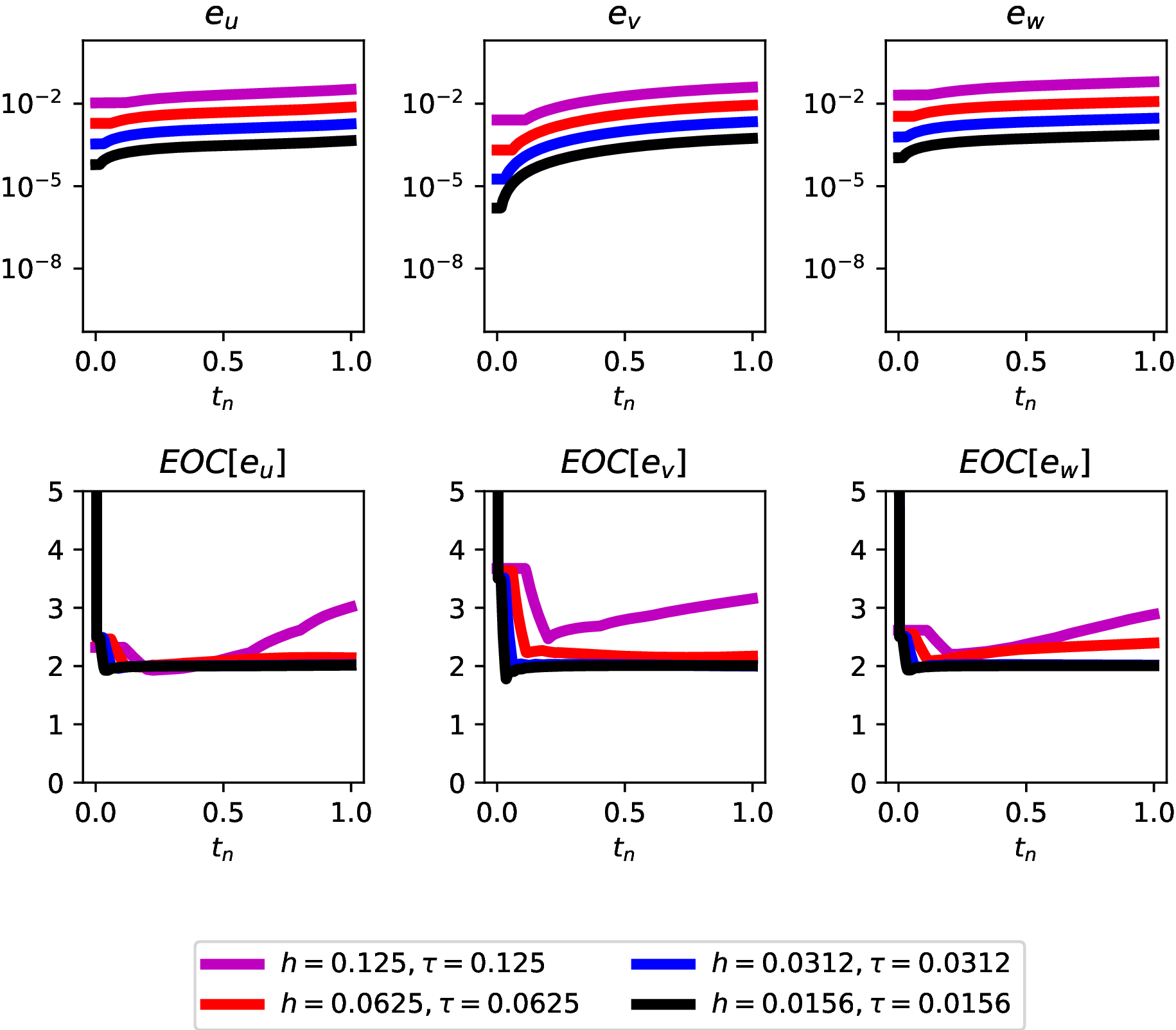}
  } \subfigure[][$q=2$ and $p=2$]{ \includegraphics[
    width=0.30\textwidth]{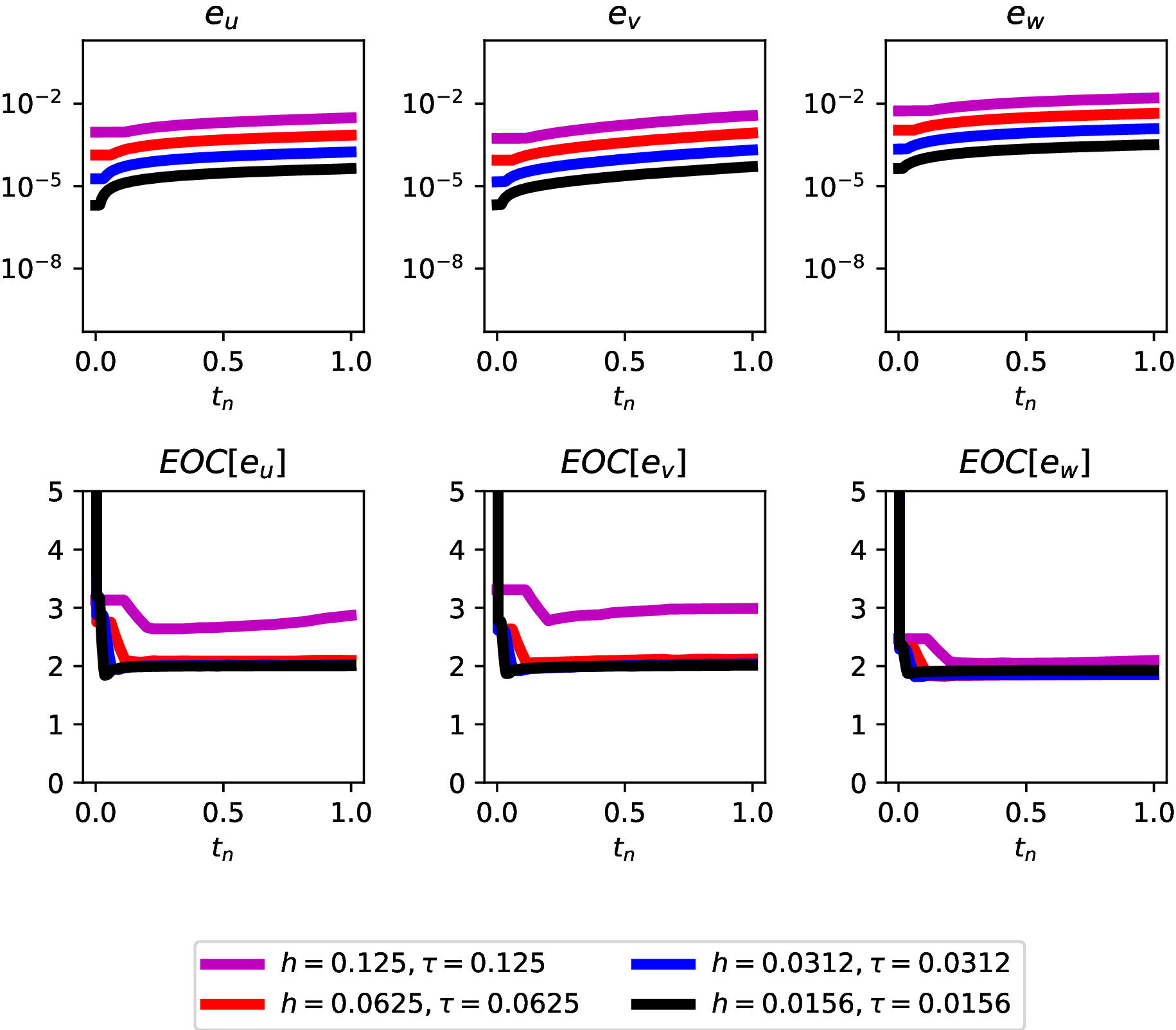}
  } \subfigure[][$q=2$ and $p=3$]{ \includegraphics[
    width=0.30\textwidth]{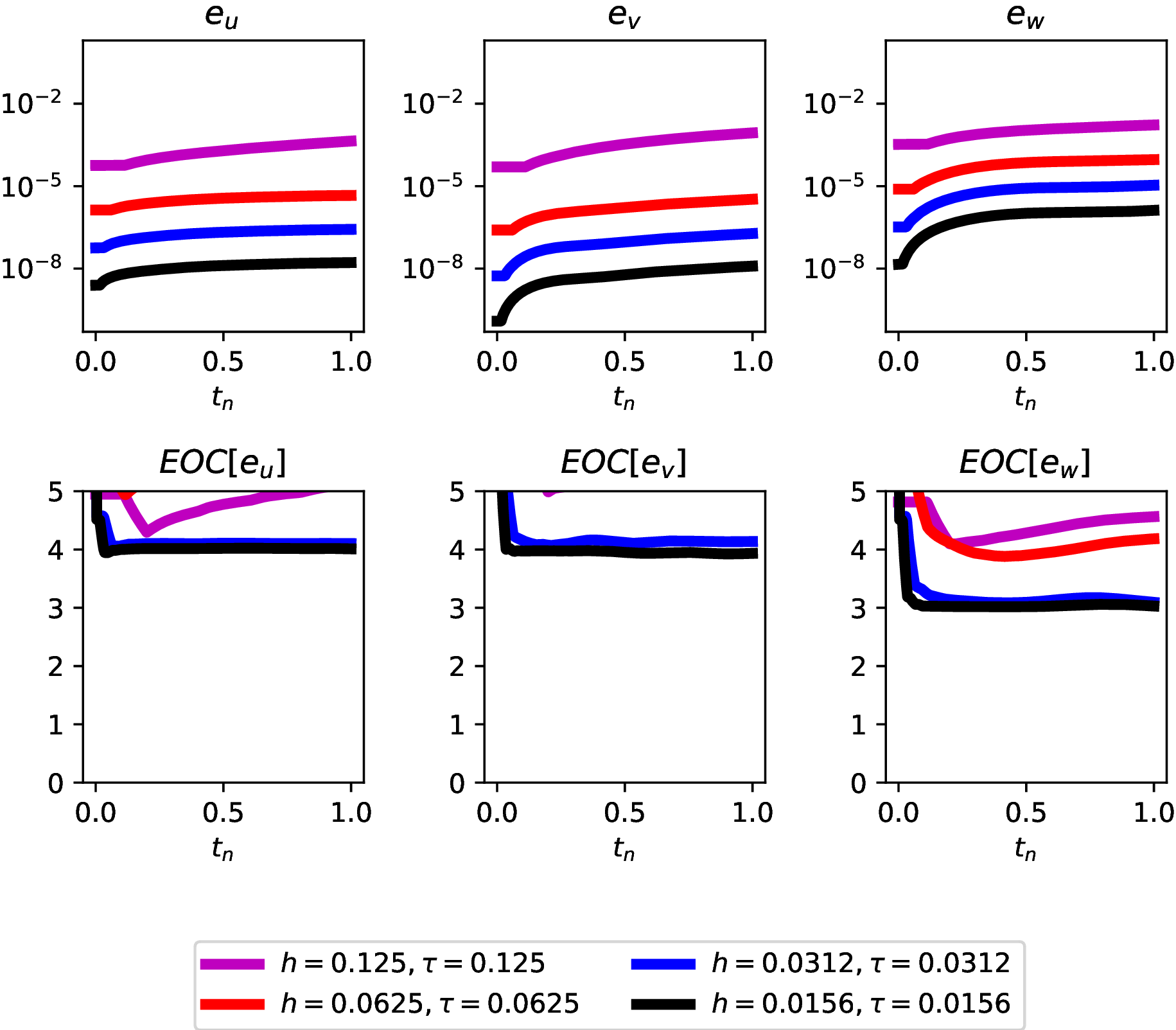}
  }

  \caption{The error of the spatially continuous finite element scheme
    \eqref{eqn:stlfem} for the nonlinear Schr{\"o}dinger equation \eqref{eqn:nls}
    initialised by the exact solution \eqref{eqn:nlsexact}. Here
    $e_u,e_v,e_w$ denotes the errors in the components $U,V,W$
    measured in the Bochner norm \eqref{eqn:enorm}. Below we plot the
    EOC \eqref{eqn:eoc} corresponding to each of these errors. \label{fig:nls:eoc}}
\end{figure}
Instead simulating the spatially discontinuous finite element
approximation \eqref{eqn:stfemdg} we obtain Figure \ref{fig:nls:eoc:dg}.
\begin{figure}[h]
  \centering
  \subfigure[][$q=0$ and $p=1$]{
    \includegraphics[
    width=0.30\textwidth]{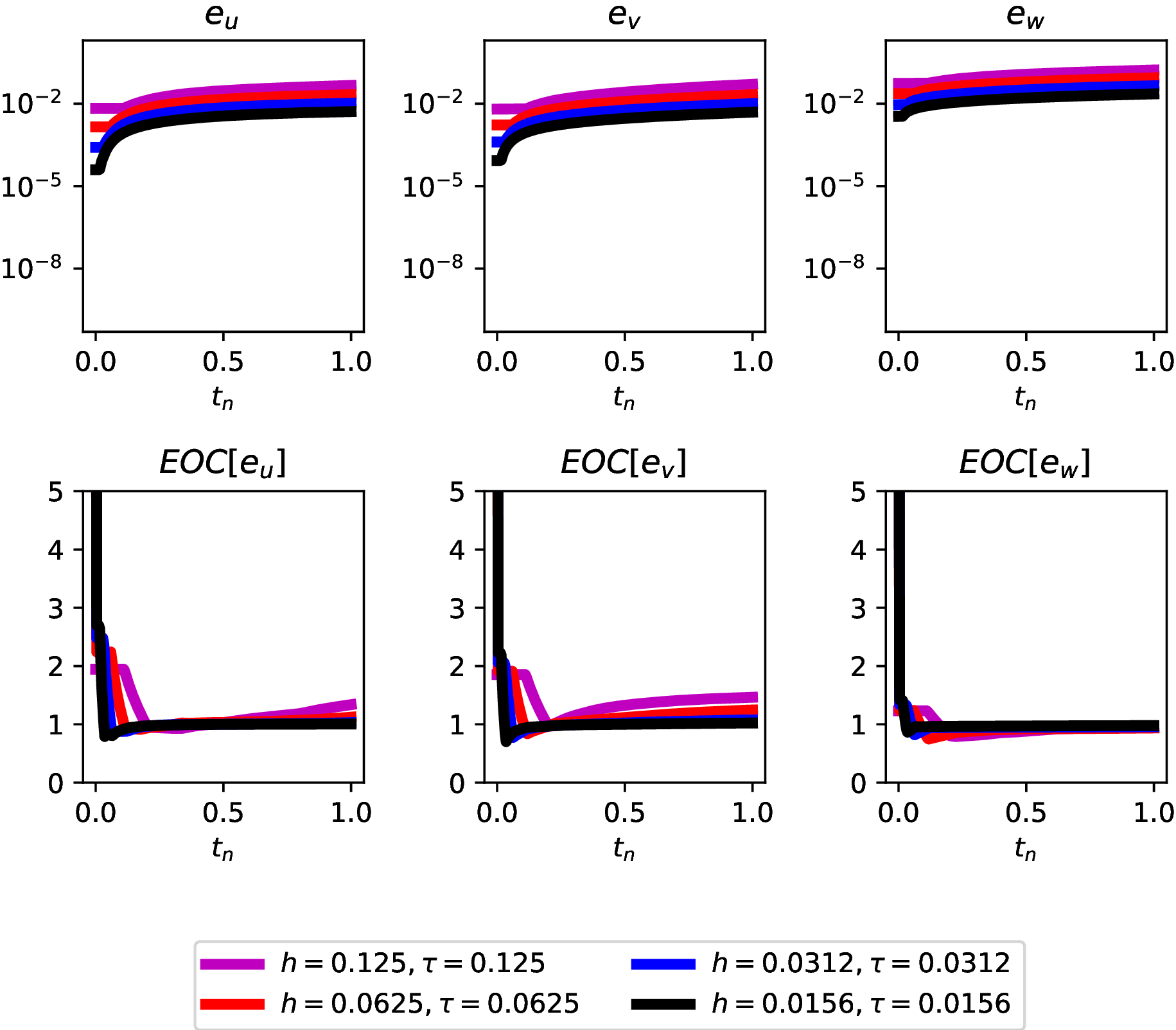}
  } \subfigure[][$q=0$ and $p=2$]{ \includegraphics[
    width=0.30\textwidth]{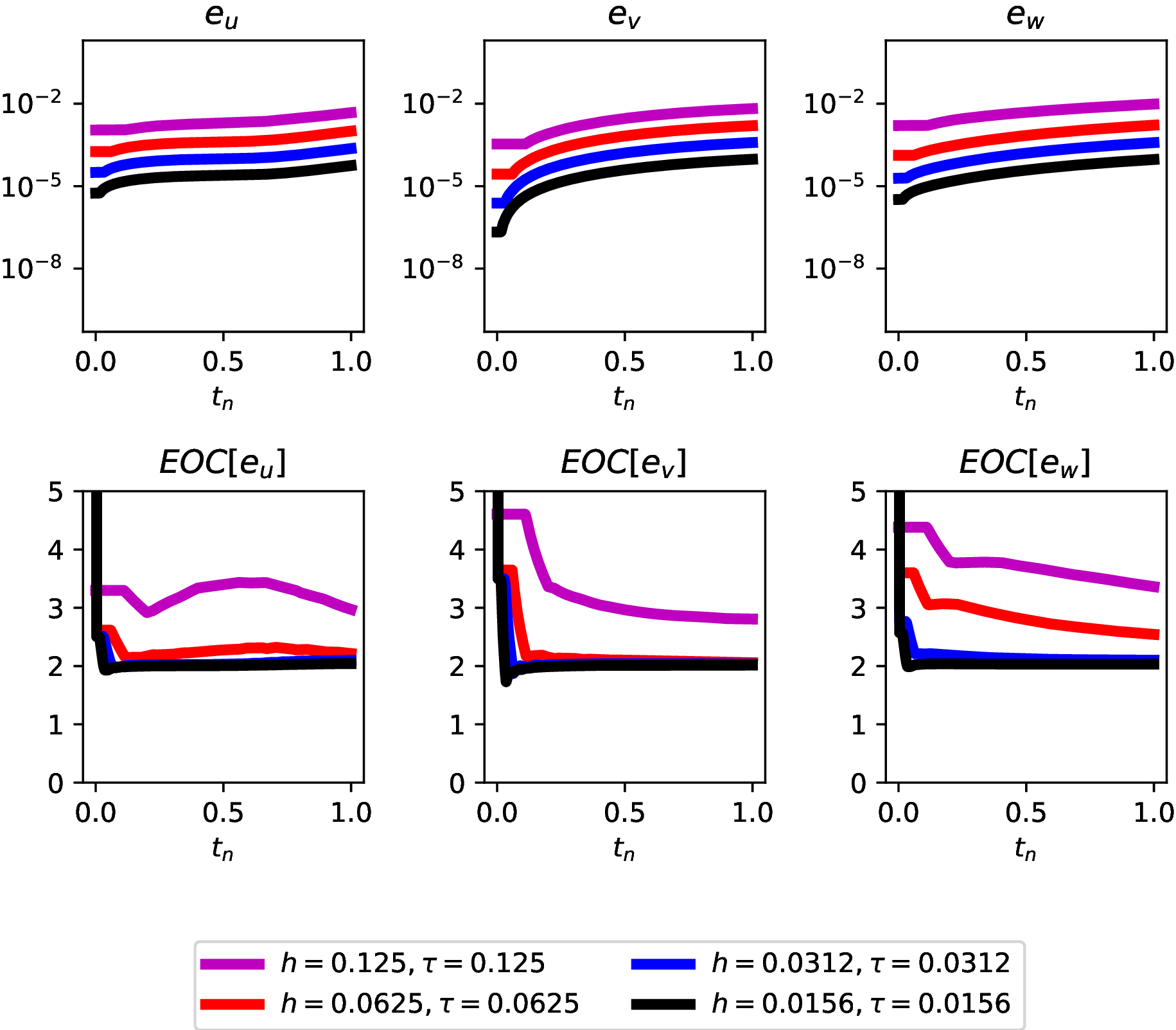}
  } \subfigure[][$q=0$ and $p=3$]{ \includegraphics[
    width=0.30\textwidth]{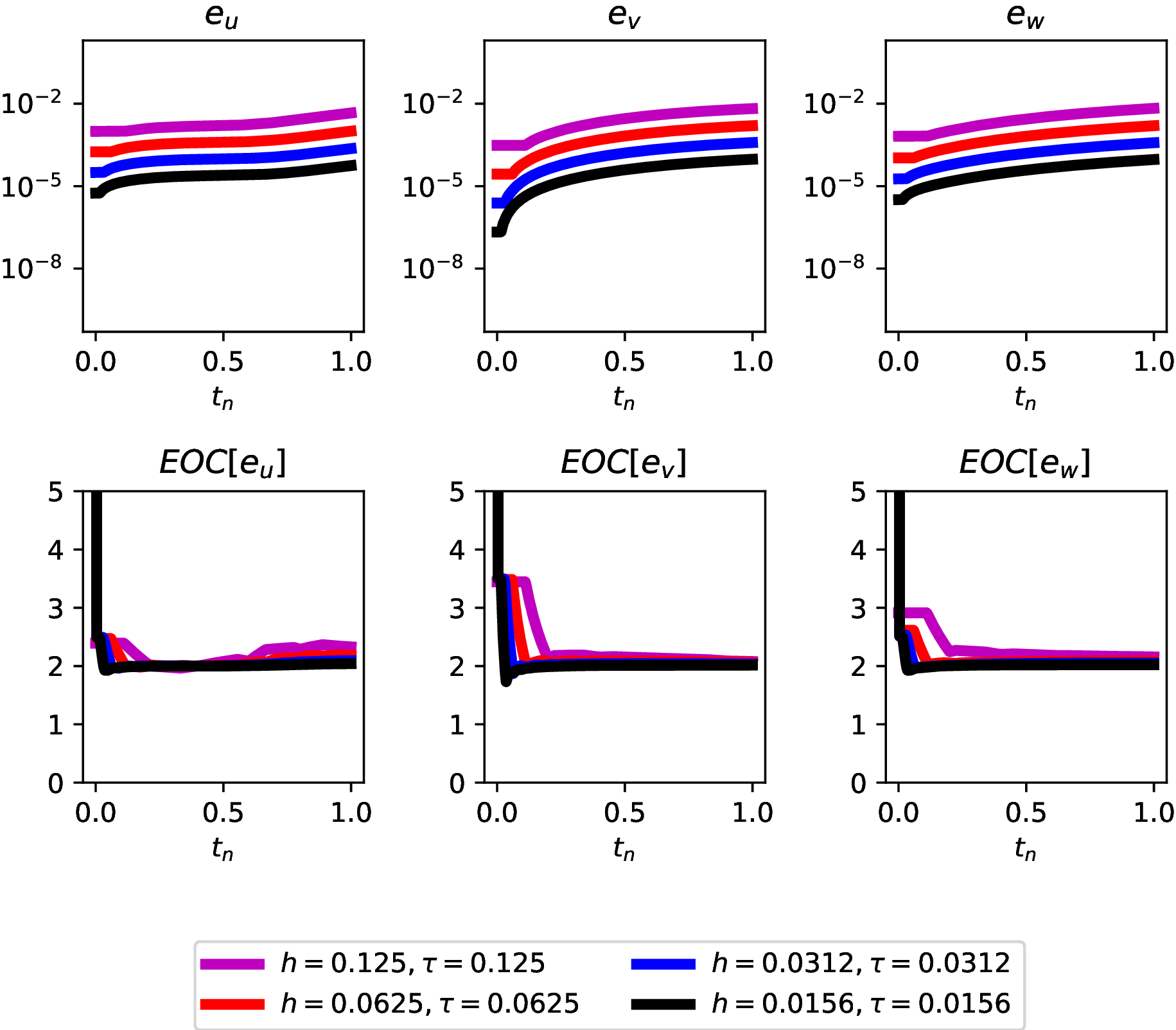}
  }
  \\
    \subfigure[][$q=1$ and $p=1$]{
    \includegraphics[
    width=0.30\textwidth]{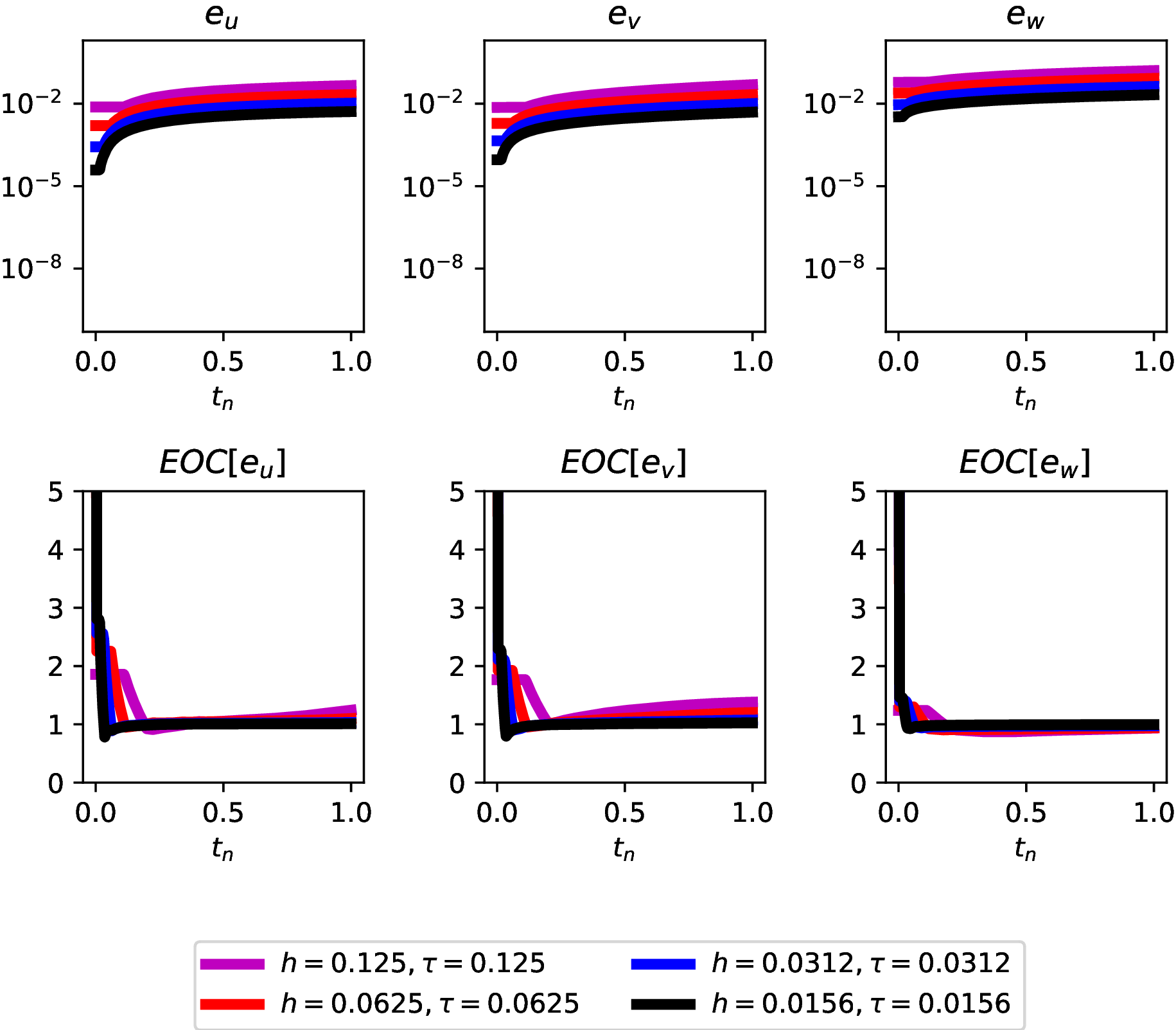}
  } \subfigure[][$q=1$ and $p=2$]{ \includegraphics[
    width=0.30\textwidth]{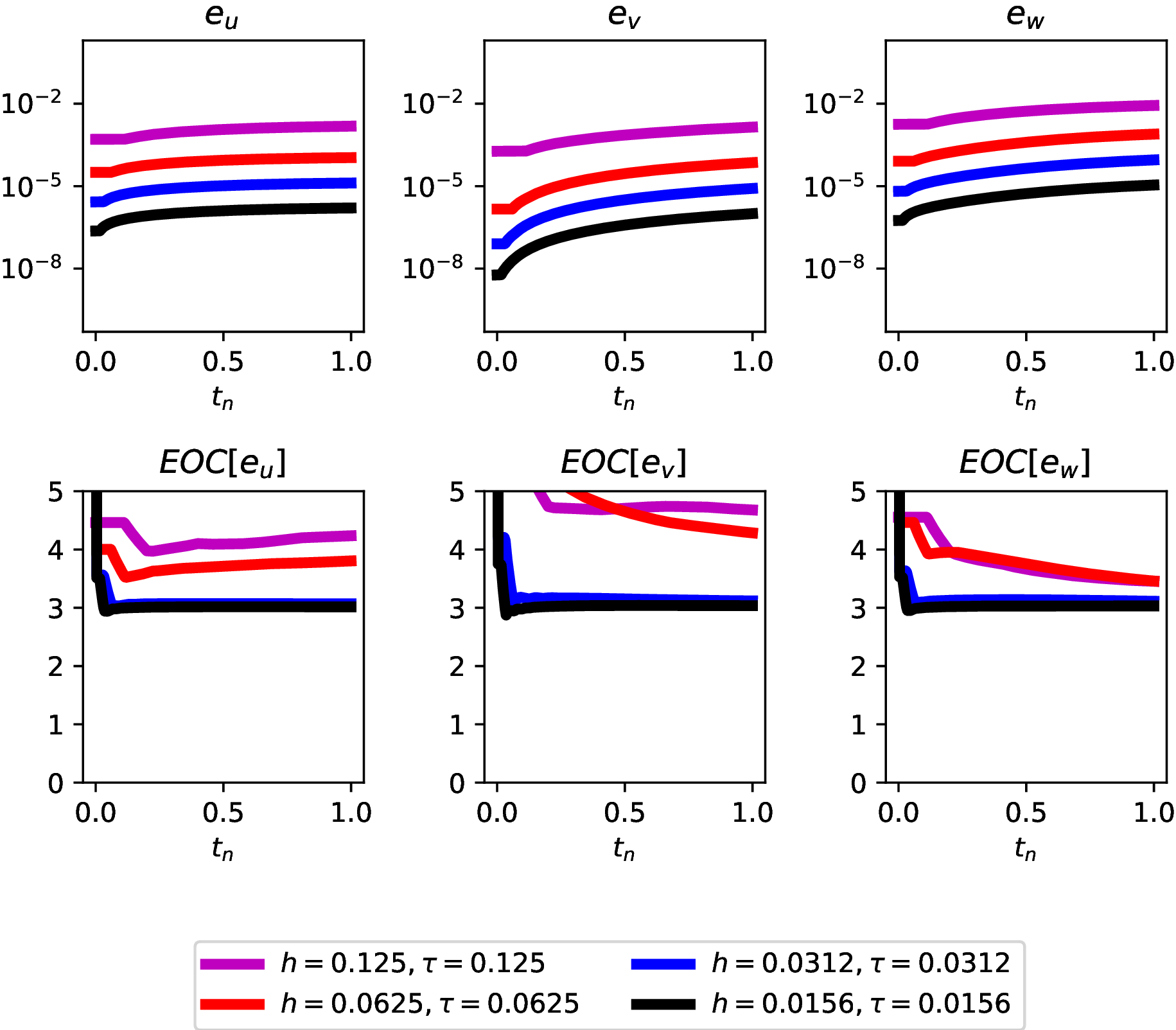}
  } \subfigure[][$q=1$ and $p=3$]{ \includegraphics[
    width=0.30\textwidth]{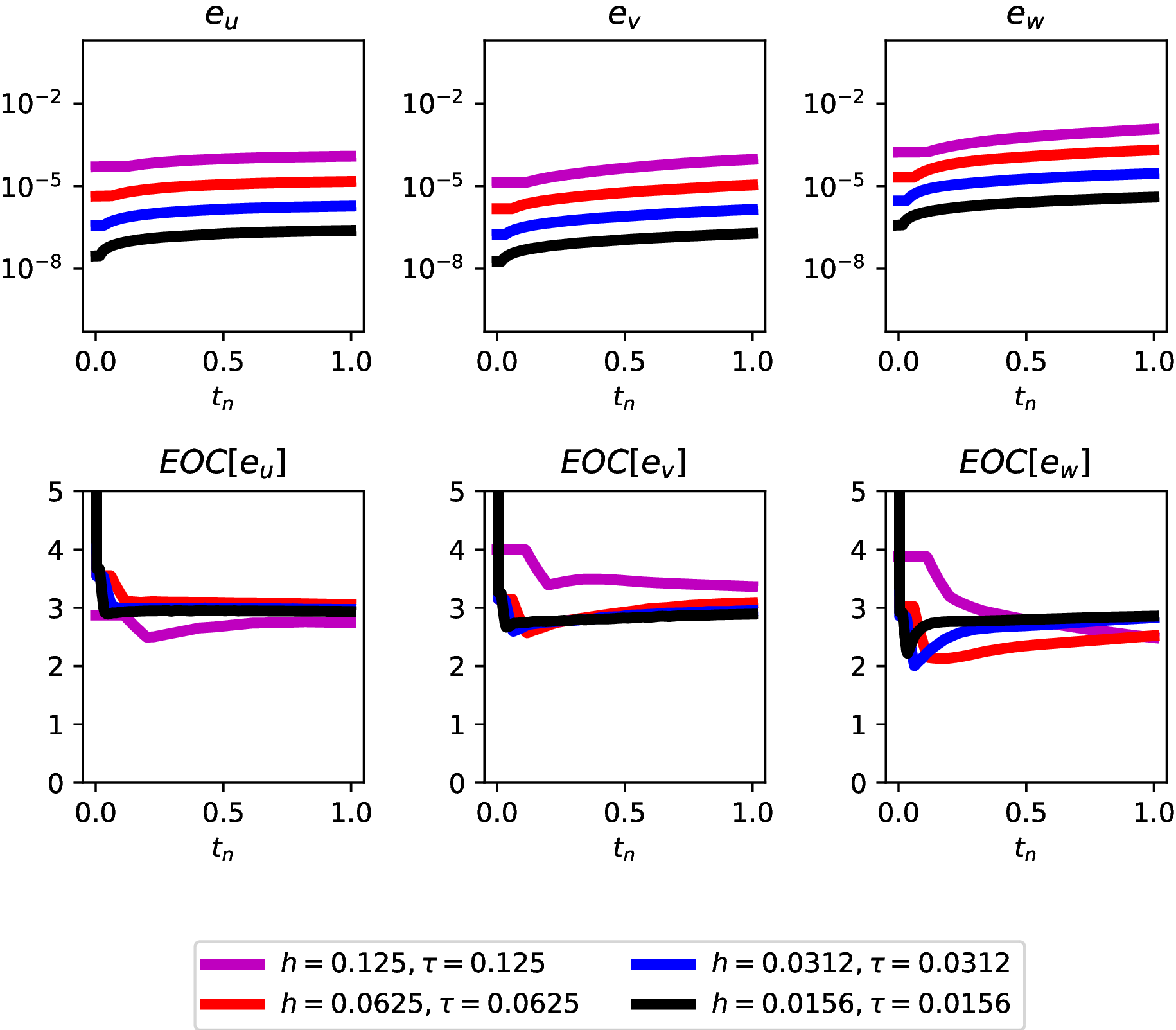}
  }
  \\
    \subfigure[][$q=2$ and $p=1$]{
    \includegraphics[
    width=0.30\textwidth]{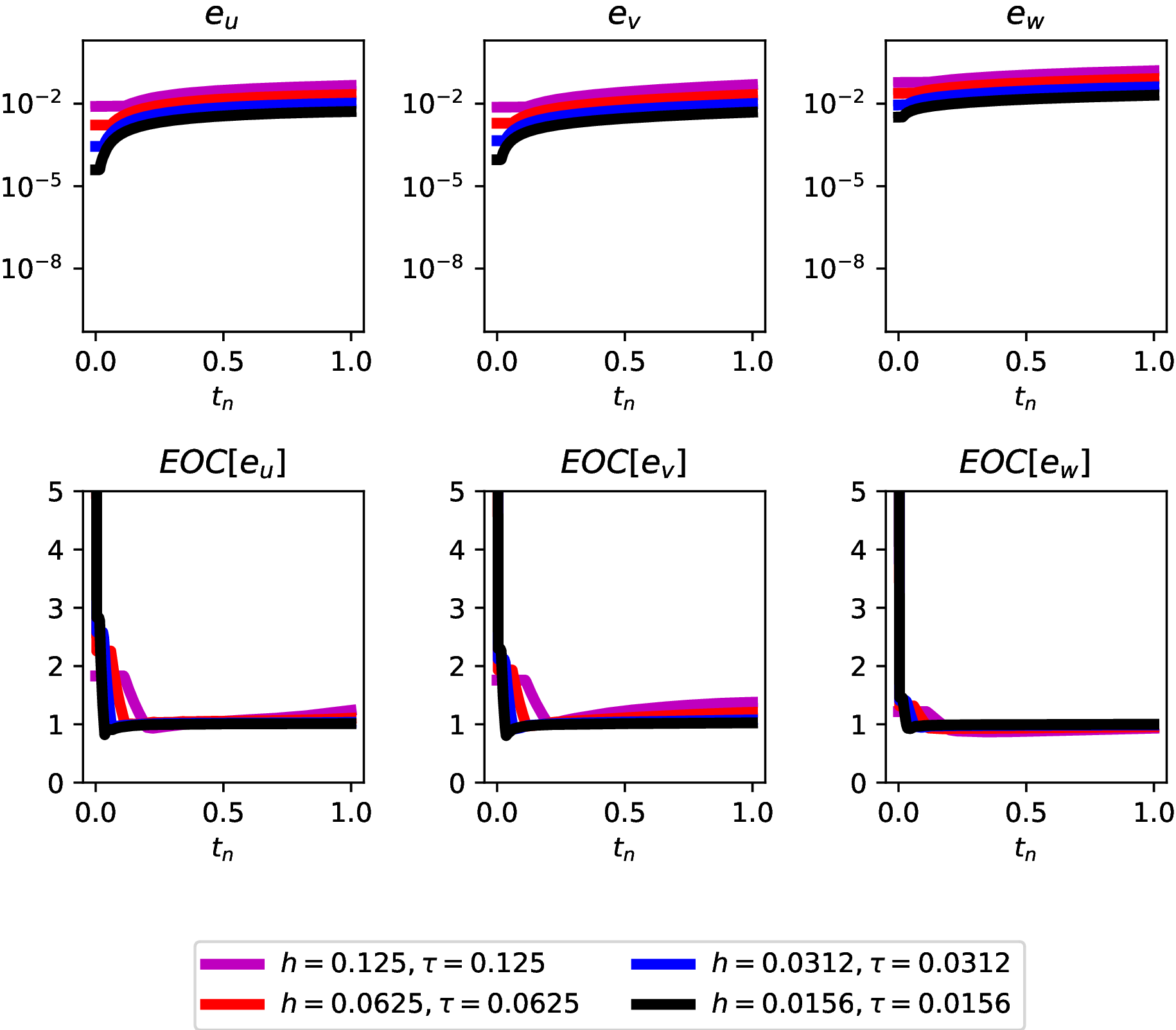}
  } \subfigure[][$q=2$ and $p=2$]{ \includegraphics[
    width=0.30\textwidth]{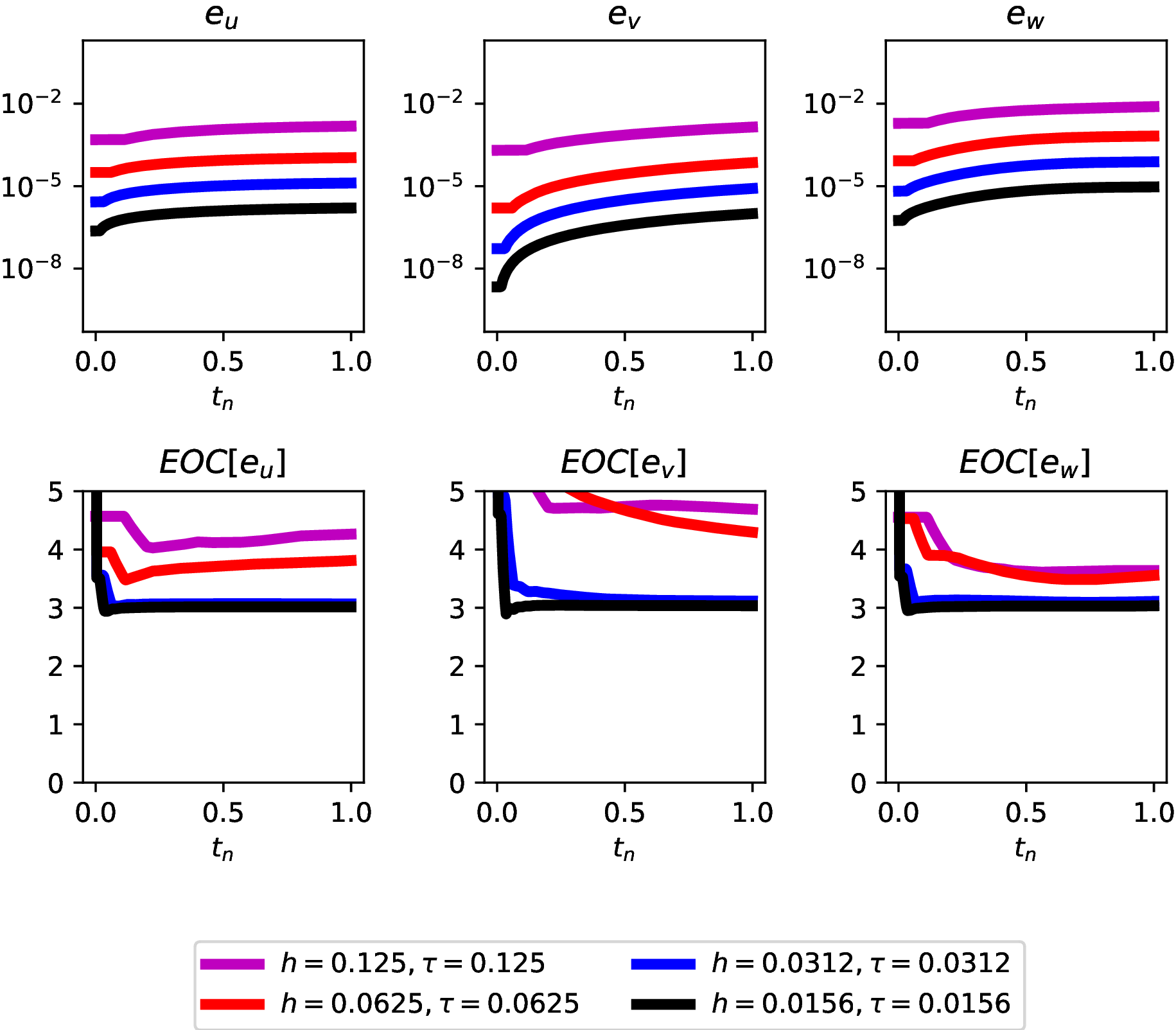}
  } \subfigure[][$q=2$ and $p=3$]{ \includegraphics[
    width=0.30\textwidth]{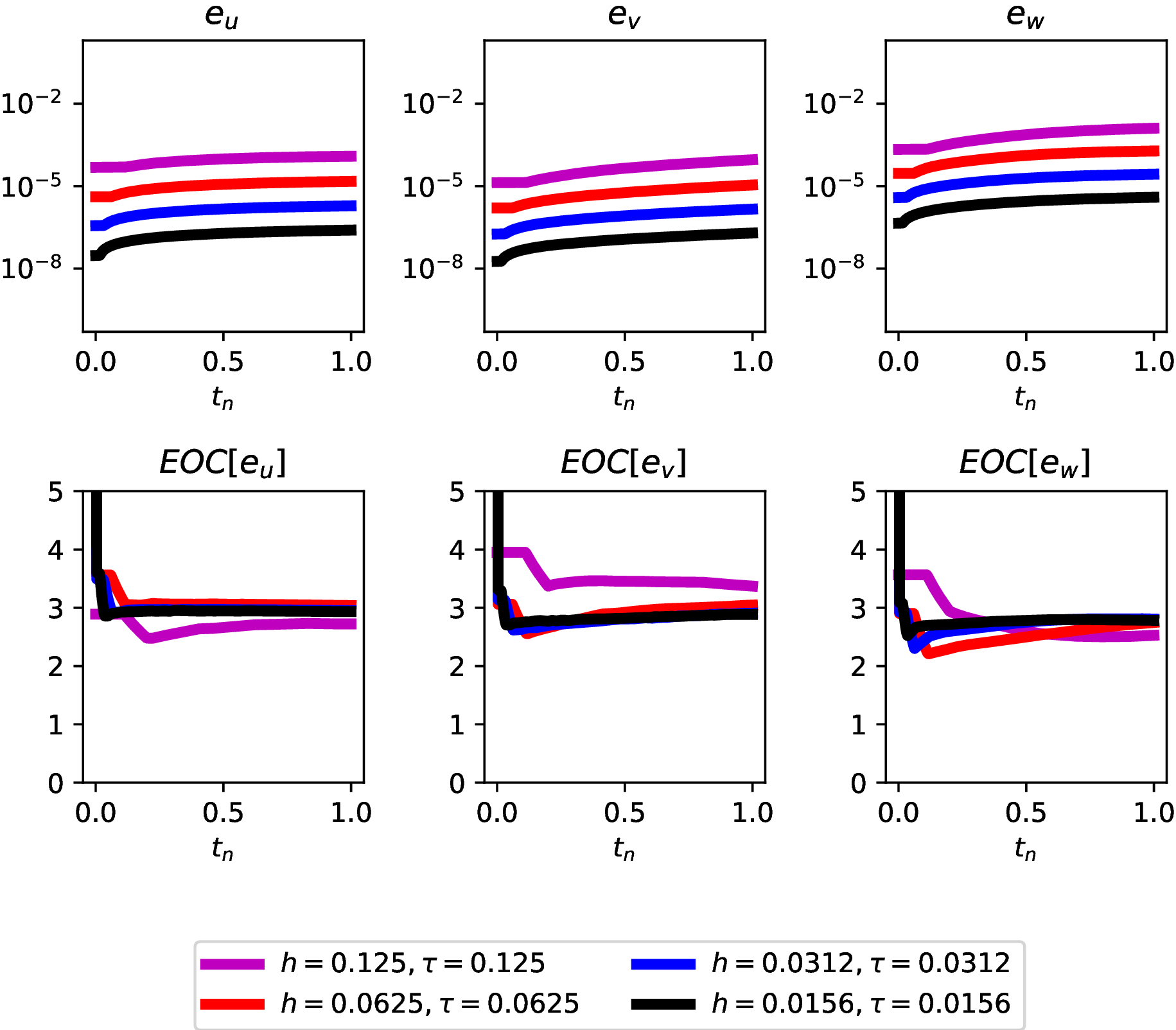}
  }

  \caption{The error of the spatially \emph{discontinuous} finite element scheme
    \eqref{eqn:stfemdg} for the nonlinear Schr{\"o}dinger equation \eqref{eqn:nls}
    initialised by the exact solution \eqref{eqn:nlsexact}. Here
    $e_u,e_v,e_w$ denotes the errors in the components $U,V,W$
    measured in the Bochner norm \eqref{eqn:enorm}. Below we plot the
    EOC \eqref{eqn:eoc} corresponding to each of these errors. \label{fig:nls:eoc:dg}}
\end{figure}
These simulations further support numerical results we found for the linear
wave equation.


Through fixing $\dt{}=0.1$ and $\dx{}=0.04$ we may conduct extensive
numerical simulations investigating the deviation in momentum and
energy conservation laws. Aforementioned simulations for the spatially
continuous approximation \eqref{eqn:stlfem} yield Figure
\ref{fig:nls:dev}.
\begin{figure}[h]
  \centering
  \subfigure[][$q=0$ and $p=1$]{
    \includegraphics[
    width=0.30\textwidth]{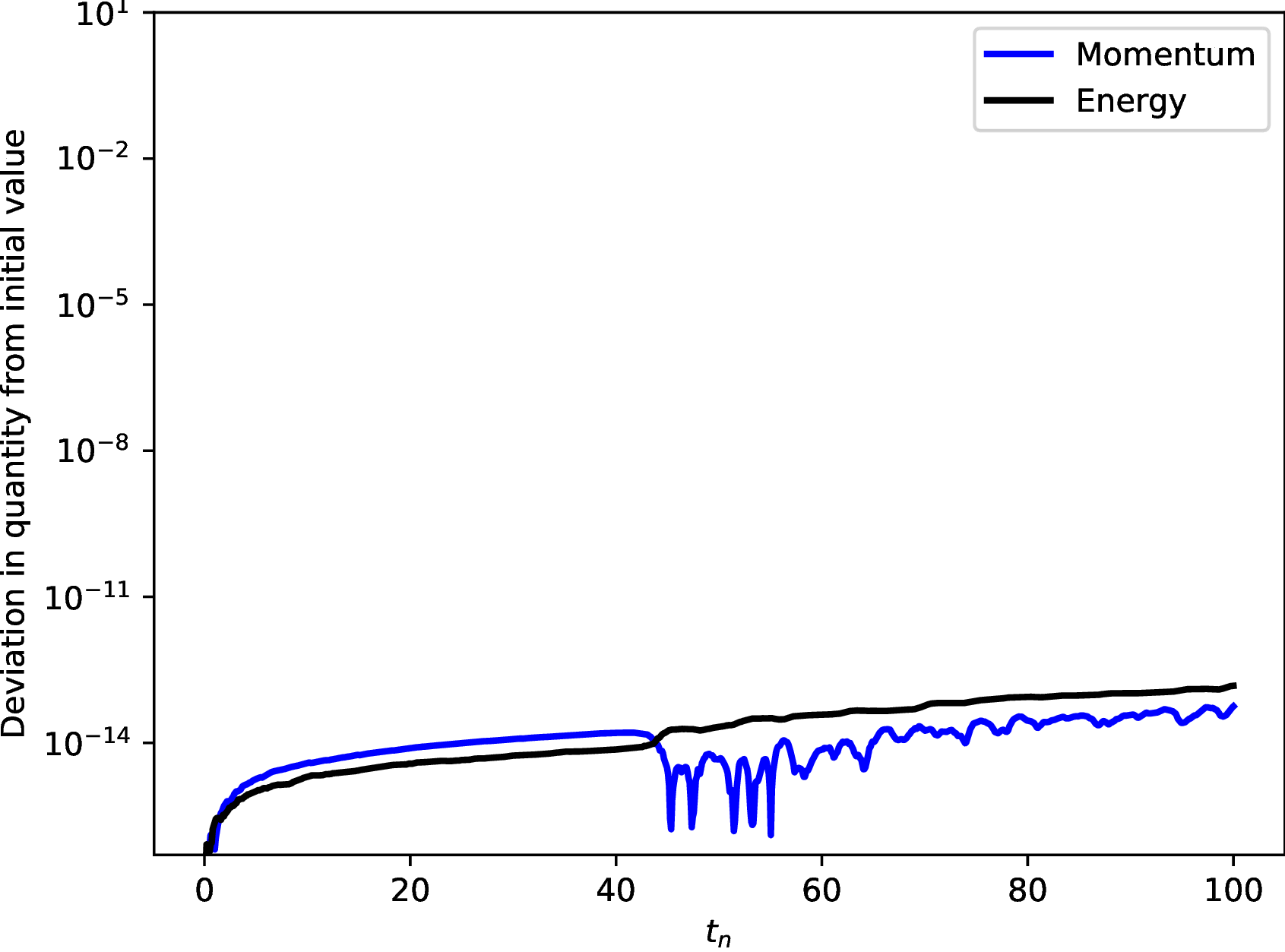}
  } \subfigure[][$q=0$ and $p=2$]{ \includegraphics[
    width=0.30\textwidth]{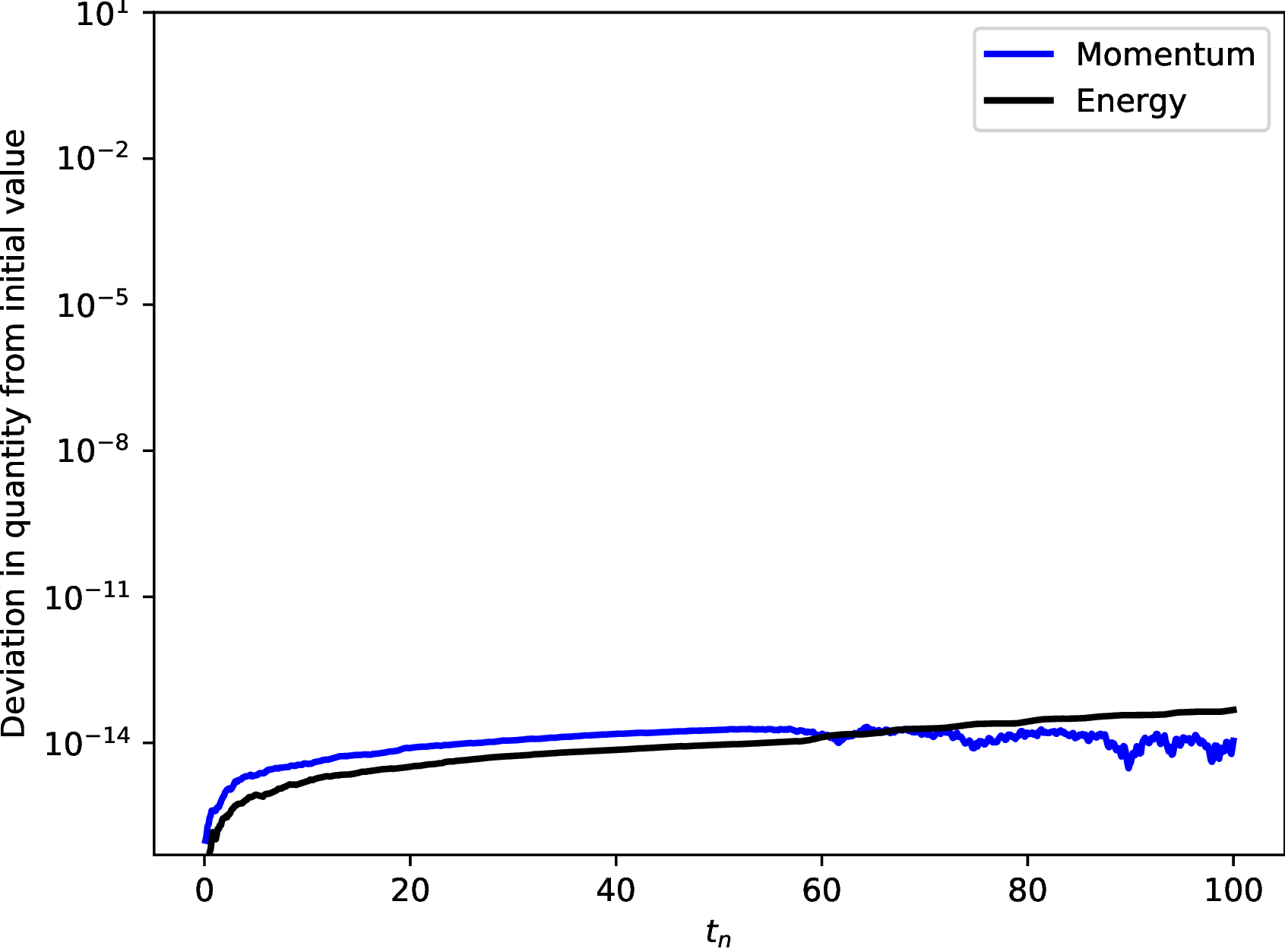}
  } \subfigure[][$q=0$ and $p=3$]{ \includegraphics[
    width=0.30\textwidth]{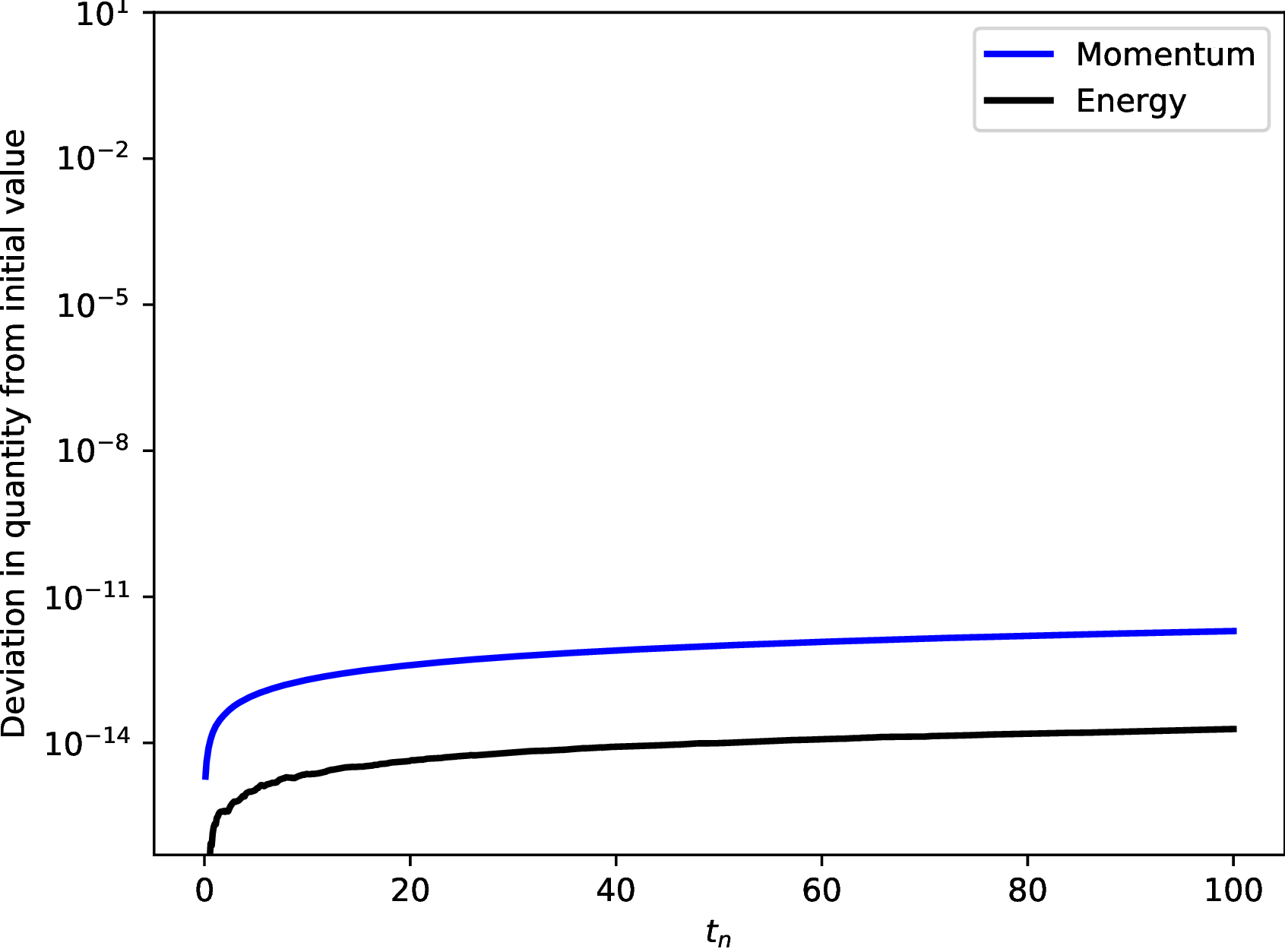}
  }
  \\
    \subfigure[][$q=1$ and $p=1$]{
    \includegraphics[
    width=0.30\textwidth]{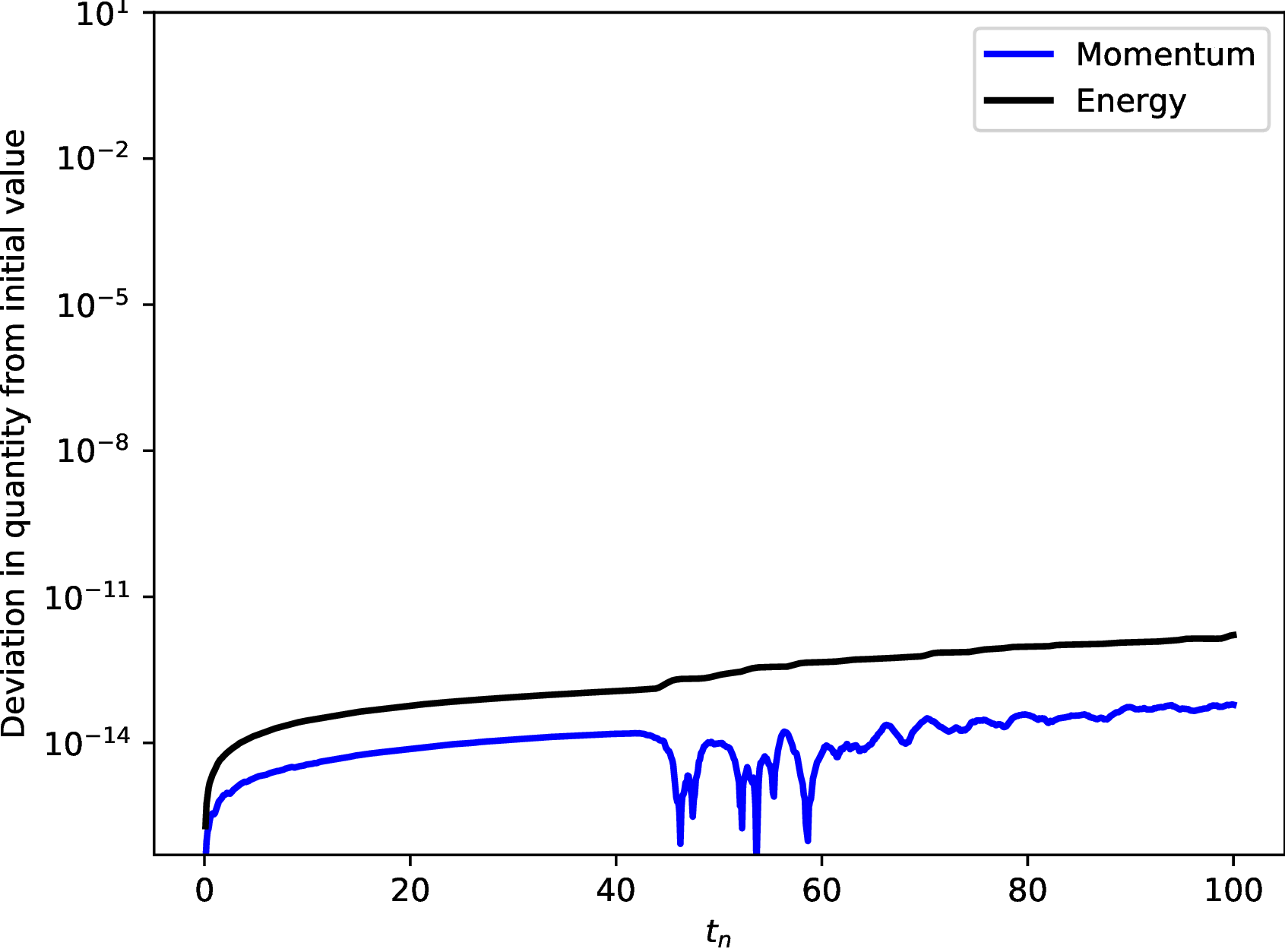}
  } \subfigure[][$q=1$ and $p=2$]{ \includegraphics[
    width=0.30\textwidth]{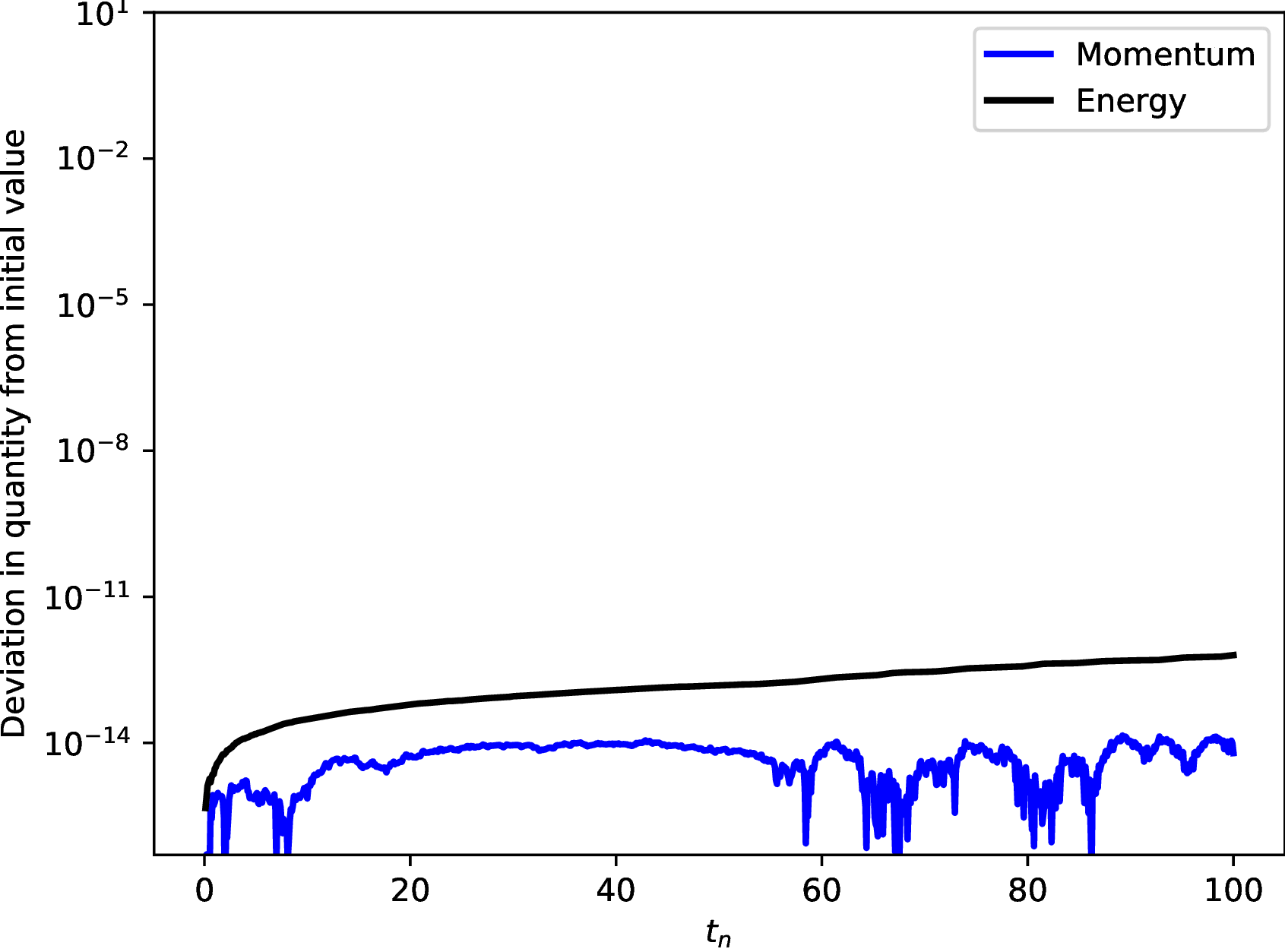}
  } \subfigure[][$q=1$ and $p=3$]{ \includegraphics[
    width=0.30\textwidth]{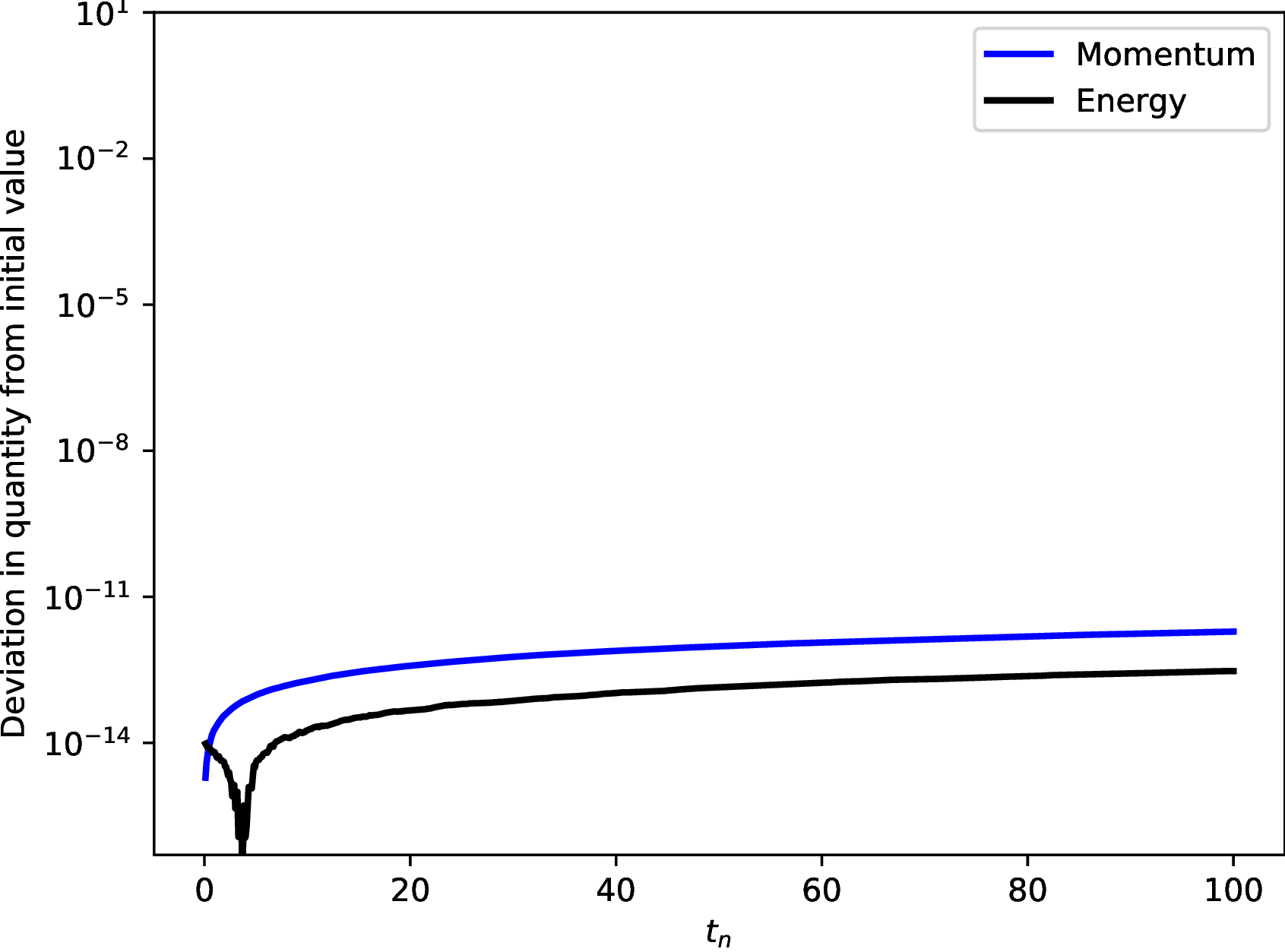}
  }
  \\
    \subfigure[][$q=2$ and $p=1$]{
    \includegraphics[
    width=0.30\textwidth]{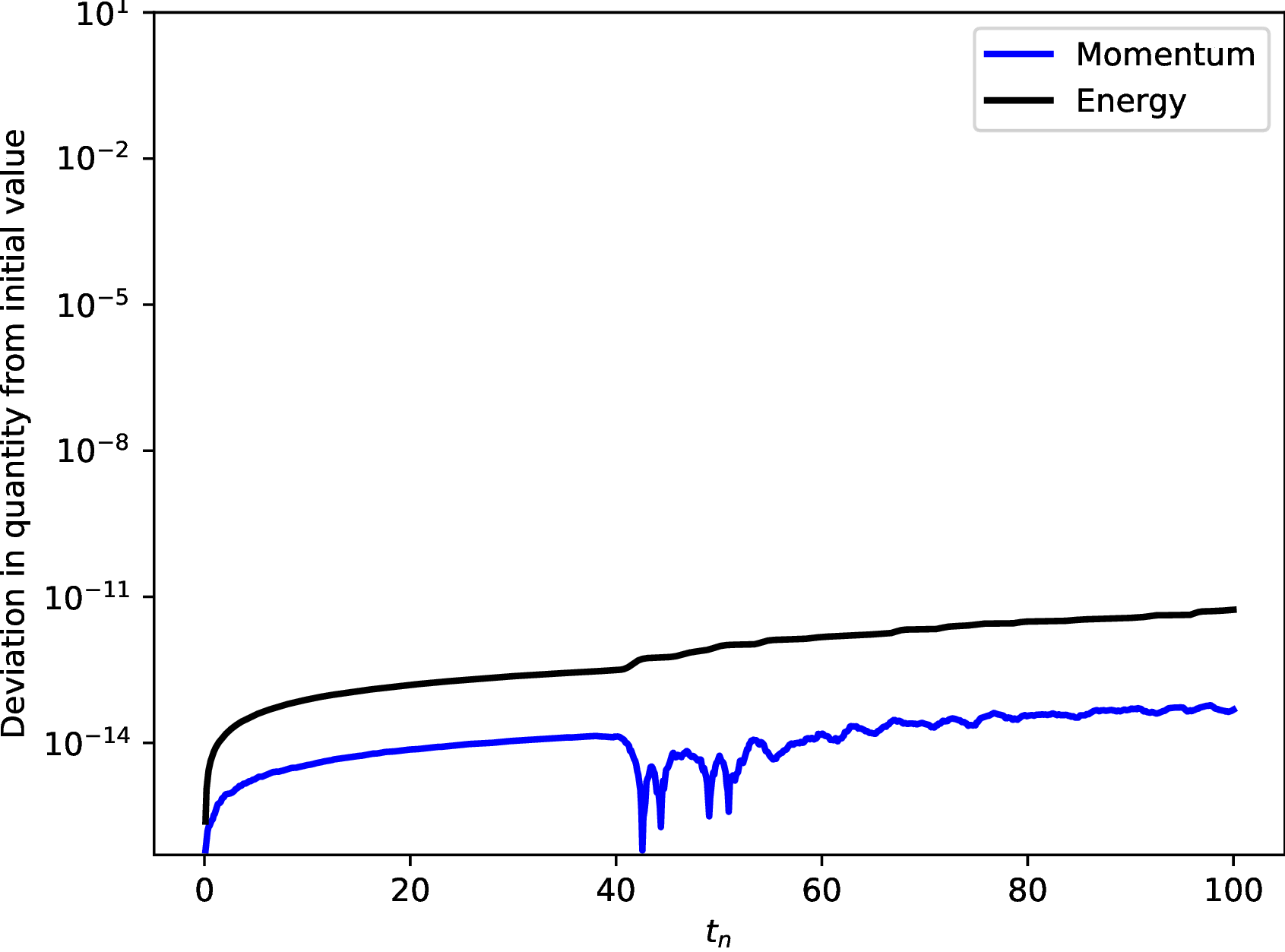}
  } \subfigure[][$q=2$ and $p=2$]{ \includegraphics[
    width=0.30\textwidth]{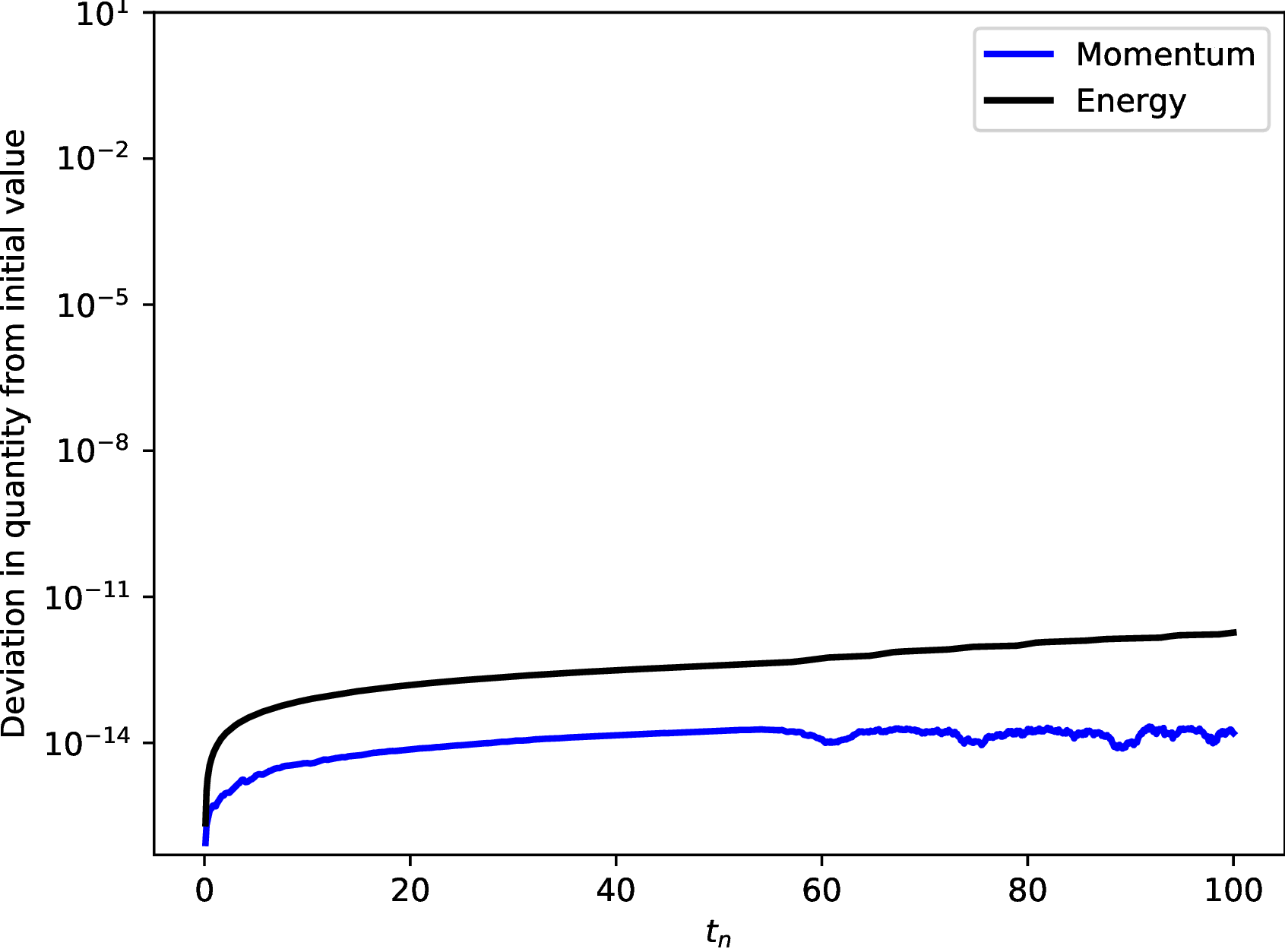}
  } \subfigure[][$q=2$ and $p=3$]{ \includegraphics[
    width=0.30\textwidth]{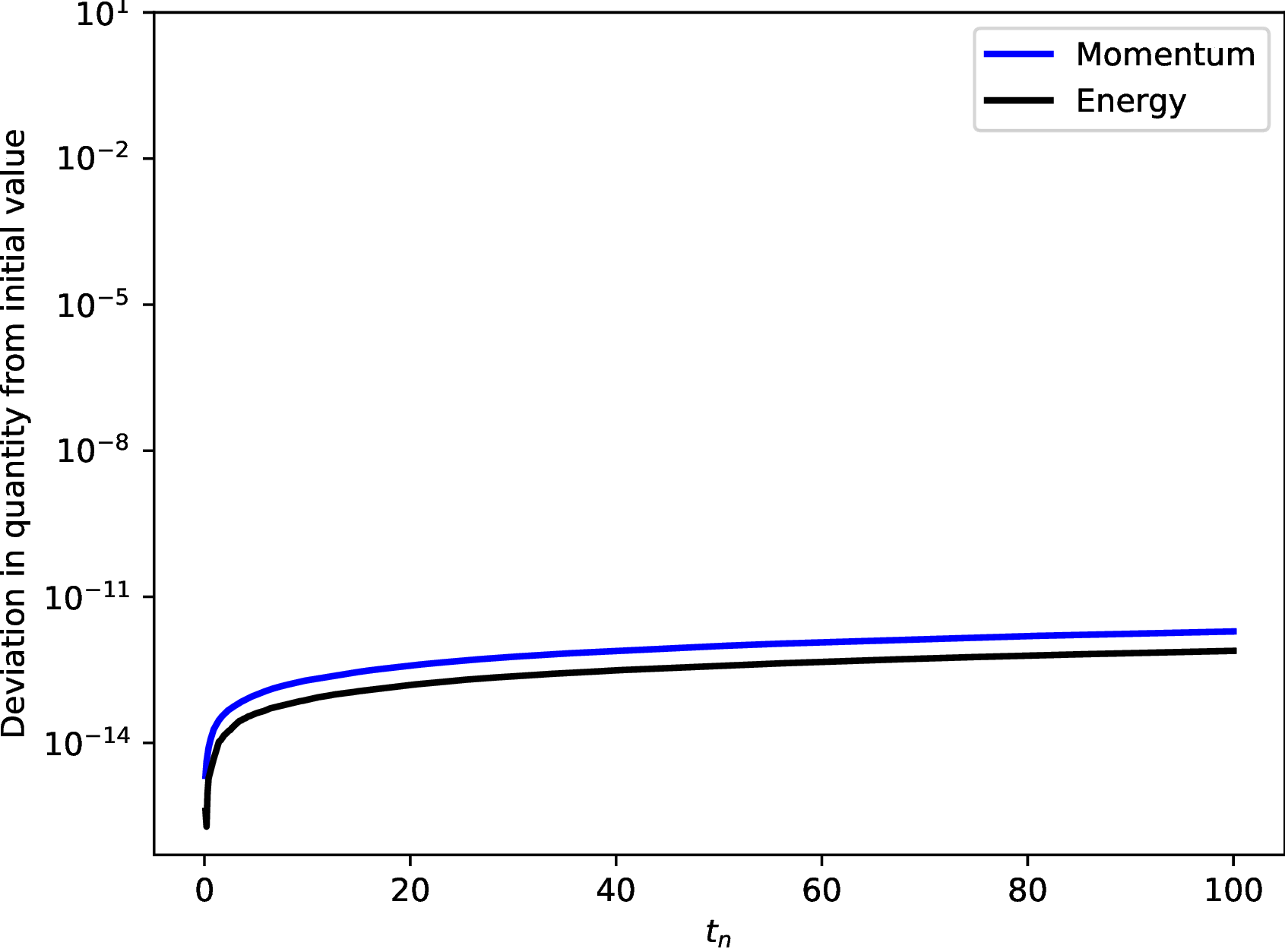}
  }

  \caption{The deviation in the conservation laws \eqref{eqn:momentum}
    and \eqref{eqn:energy} for the spatially continuous
    finite element scheme \eqref{eqn:stlfem} for the
    NLS equation \eqref{eqn:nls} initialised by \eqref{eqn:nlsexact}
    with varying temporal degree $q$ and spatial degree $p$. All
    simulations have uniform temporal and spatial elements with
    $\dt{}=0.1$ and $\dx{}=0.4$. We notice that both momentum and
    energy are preserved below our solver tolerance in all simulations. \label{fig:nls:dev}}
\end{figure}
We observe in Figure \ref{fig:nls:dev} that the energy is preserved
locally, with the deviation in the energy propagating over
time. Furthermore, we note that the momentum is also preserved, which
is not analytically guaranteed. We note that we may exactly preserve
the momentum here as the consistent momentum conservation law
\eqref{eqn:mclaw} is equal to the true momentum conservation law
\eqref{eqn:momentum2} up to solver precision. This is, in part, due to
the non-linearity in the NLS equation being of low order.

Running the same experiment for the spatially discontinuous scheme
\eqref{eqn:stfemdg} we obtain Figure \ref{fig:nls:dev:dg}.
\begin{figure}[h]
  \centering
  \subfigure[][$q=0$ and $p=1$]{
    \includegraphics[
    width=0.30\textwidth]{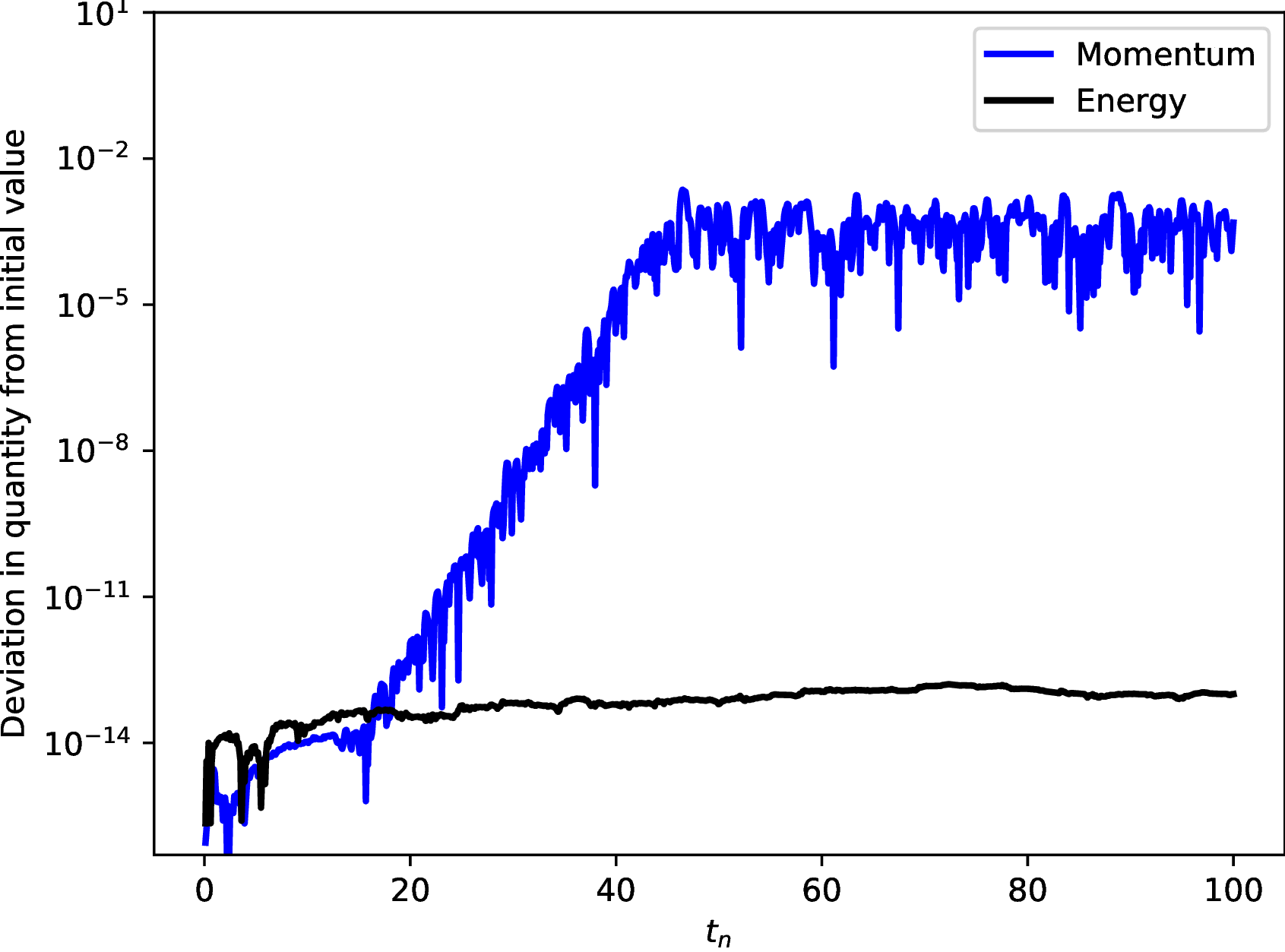}
  } \subfigure[][$q=0$ and $p=2$]{ \includegraphics[
    width=0.30\textwidth]{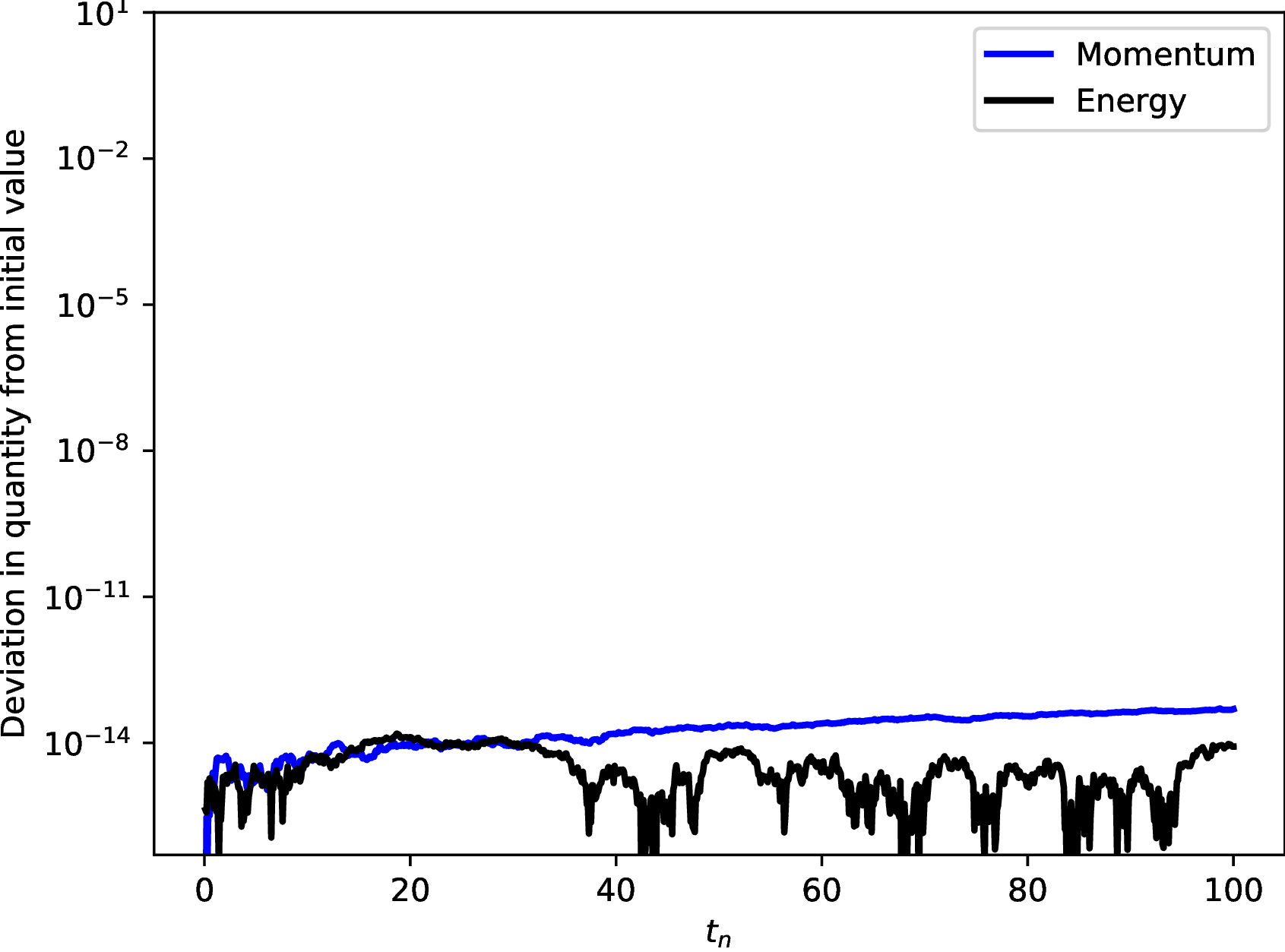}
  } \subfigure[][$q=0$ and $p=3$]{ \includegraphics[
    width=0.30\textwidth]{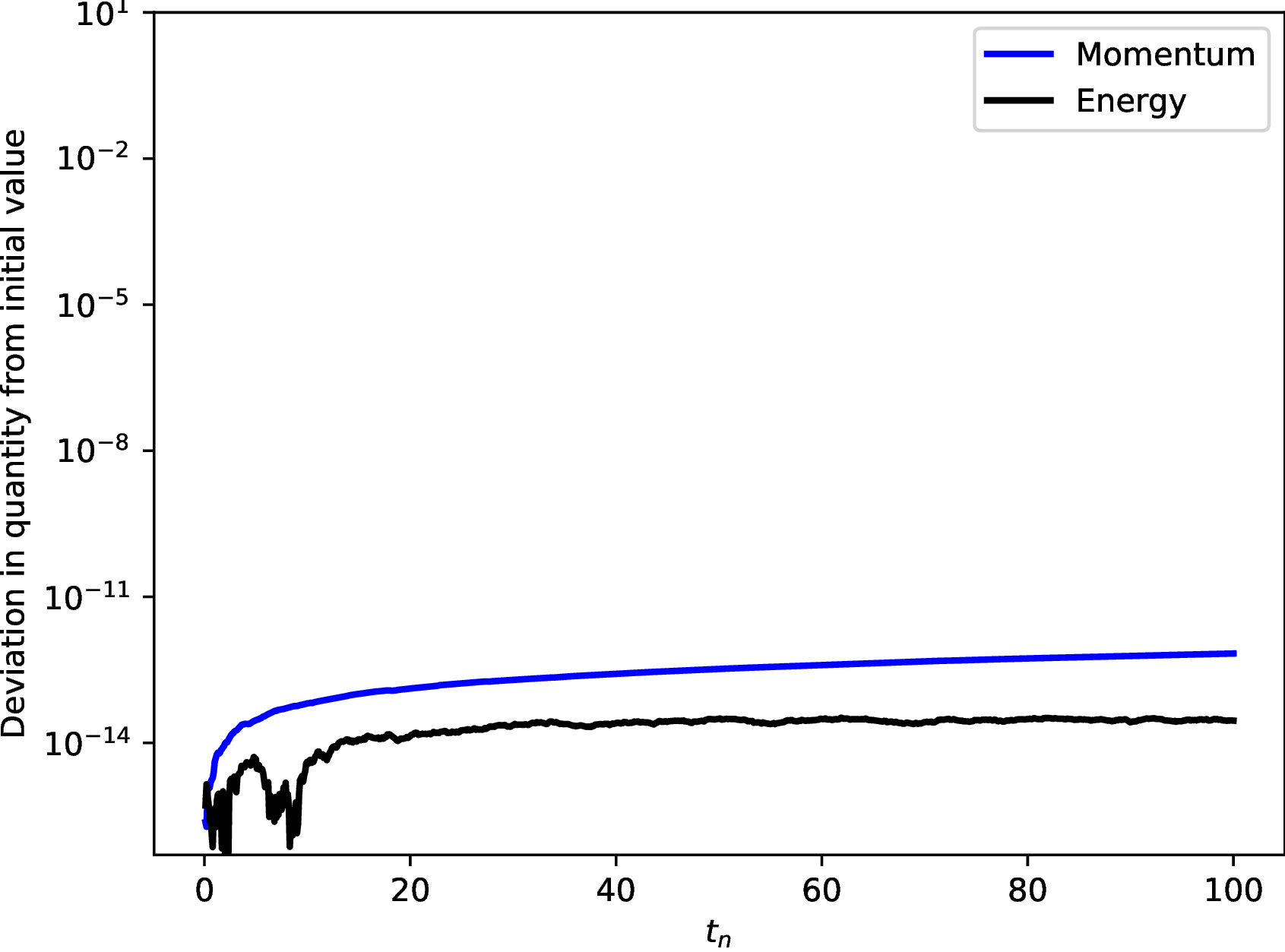}
  }
  \\
    \subfigure[][$q=1$ and $p=1$]{
    \includegraphics[
    width=0.30\textwidth]{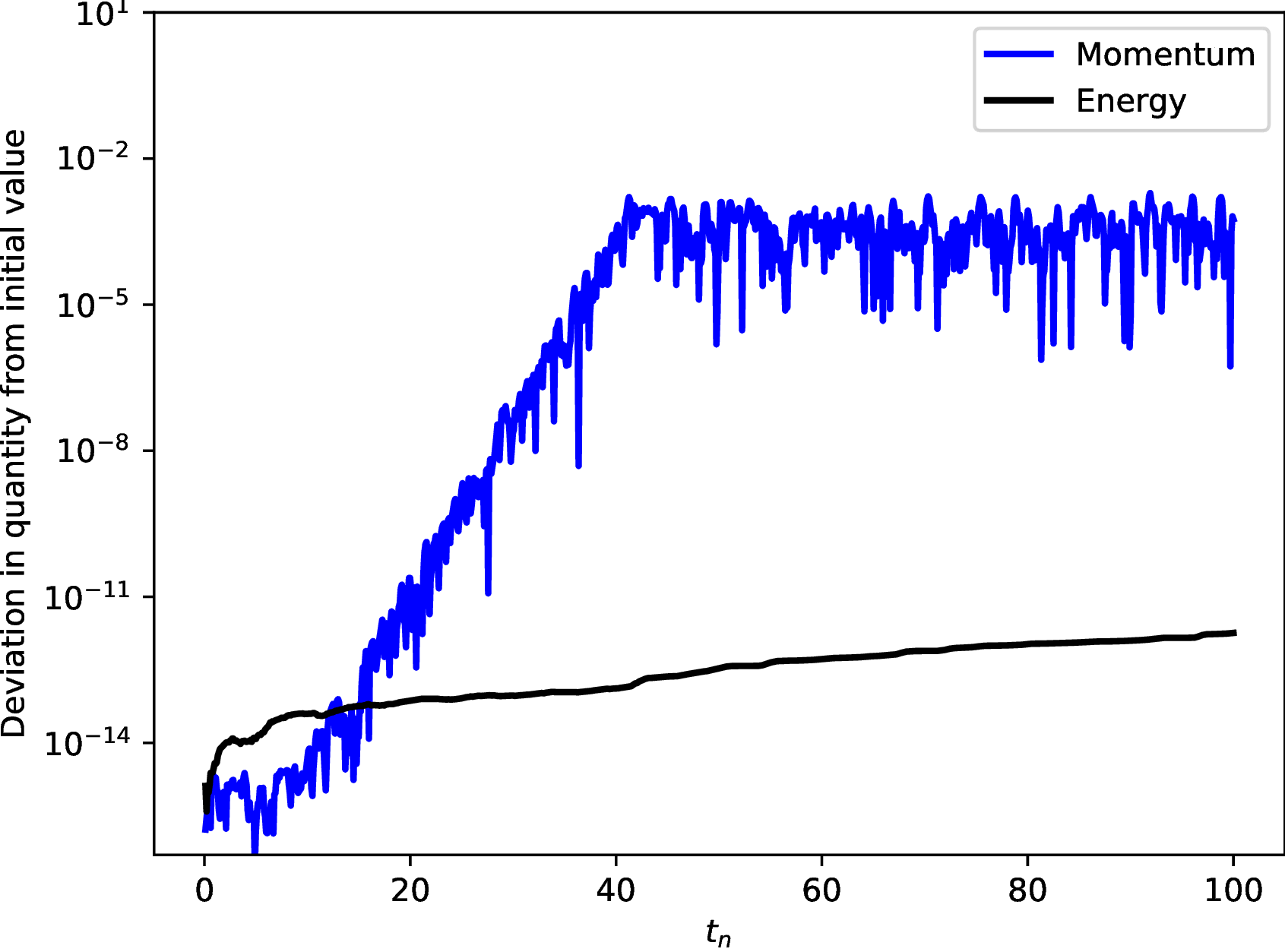}
  } \subfigure[][$q=1$ and $p=2$]{ \includegraphics[
    width=0.30\textwidth]{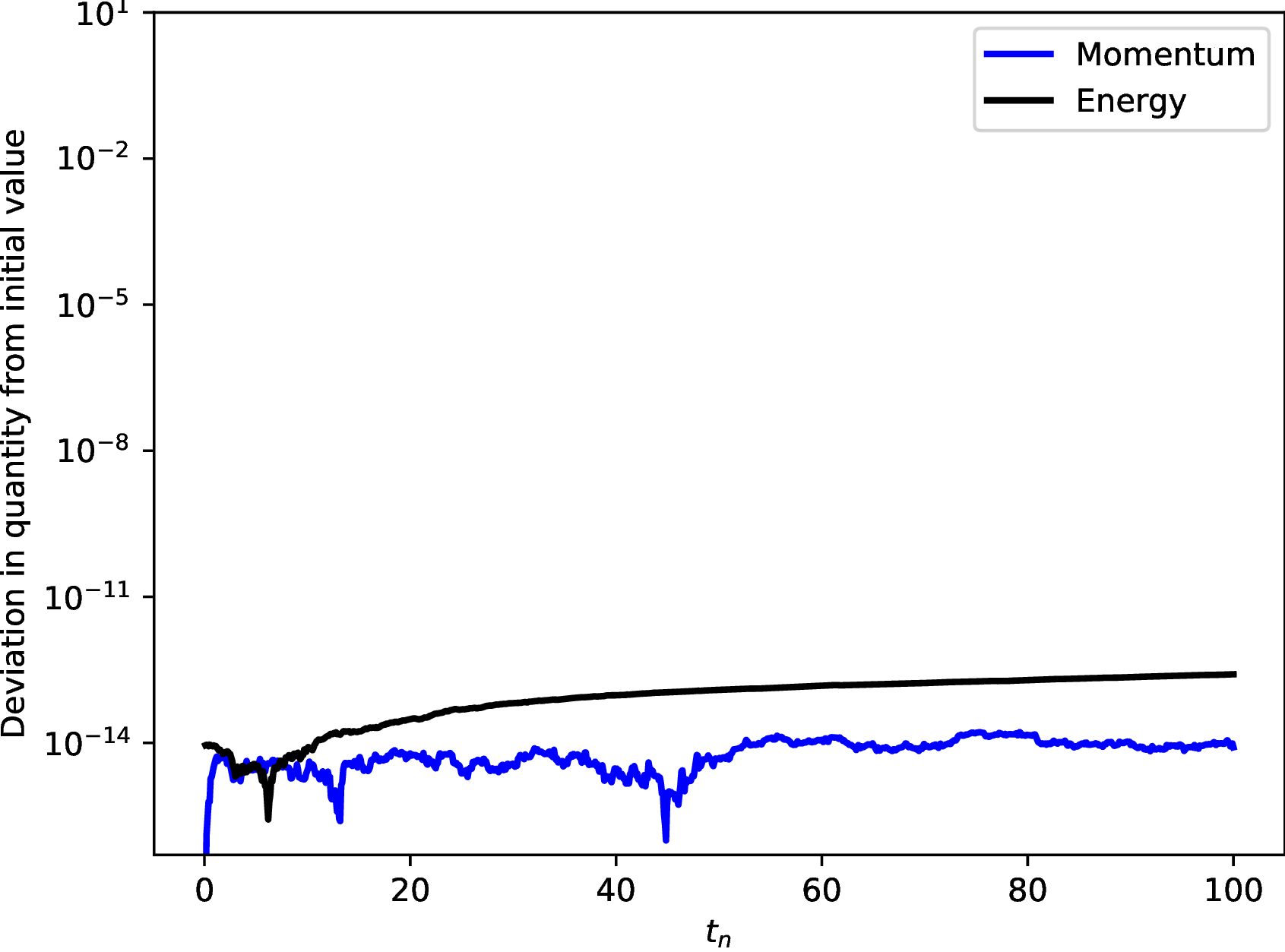}
  } \subfigure[][$q=1$ and $p=3$]{ \includegraphics[
    width=0.30\textwidth]{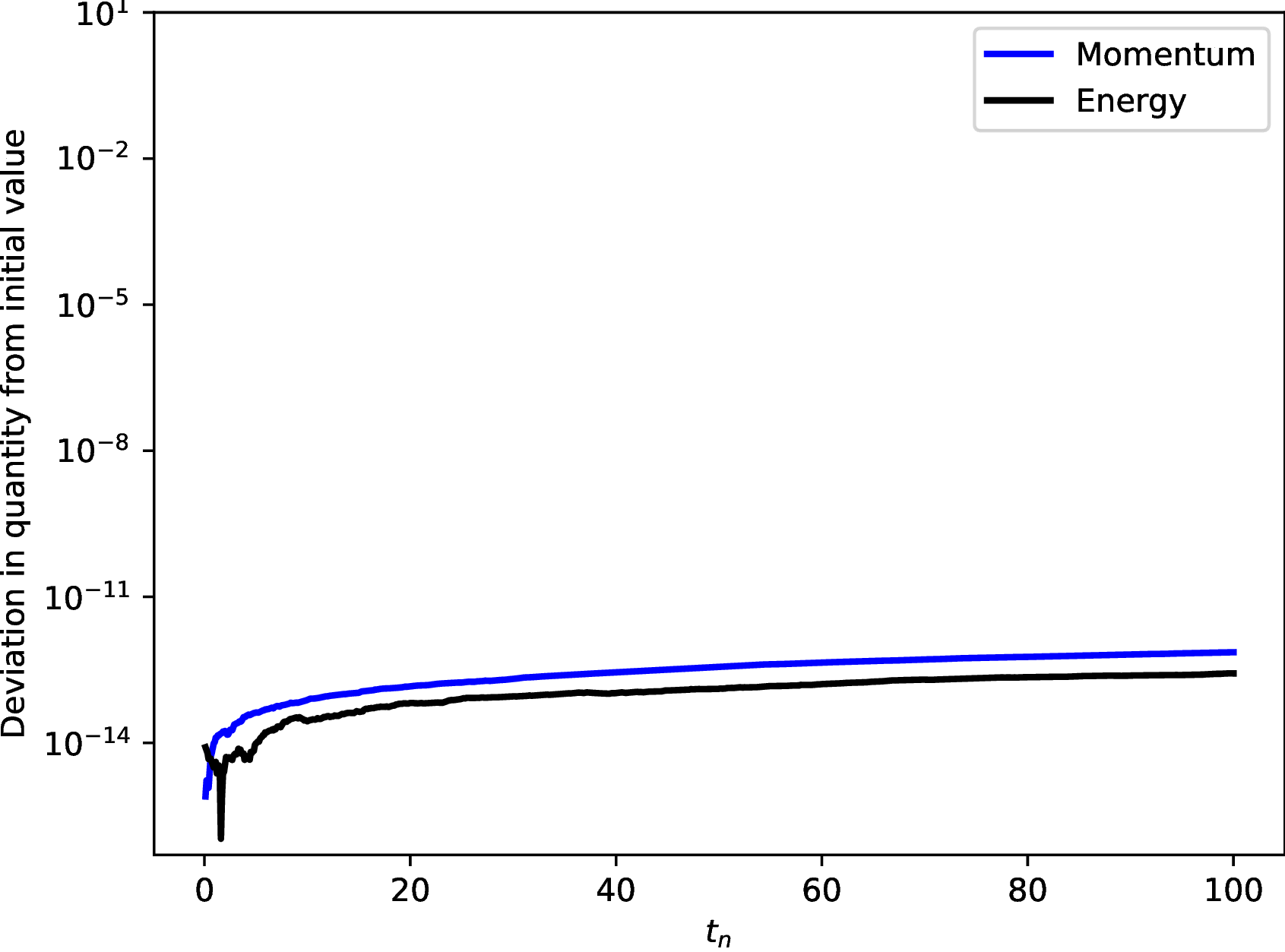}
  }
  \\
    \subfigure[][$q=2$ and $p=1$]{
    \includegraphics[
    width=0.30\textwidth]{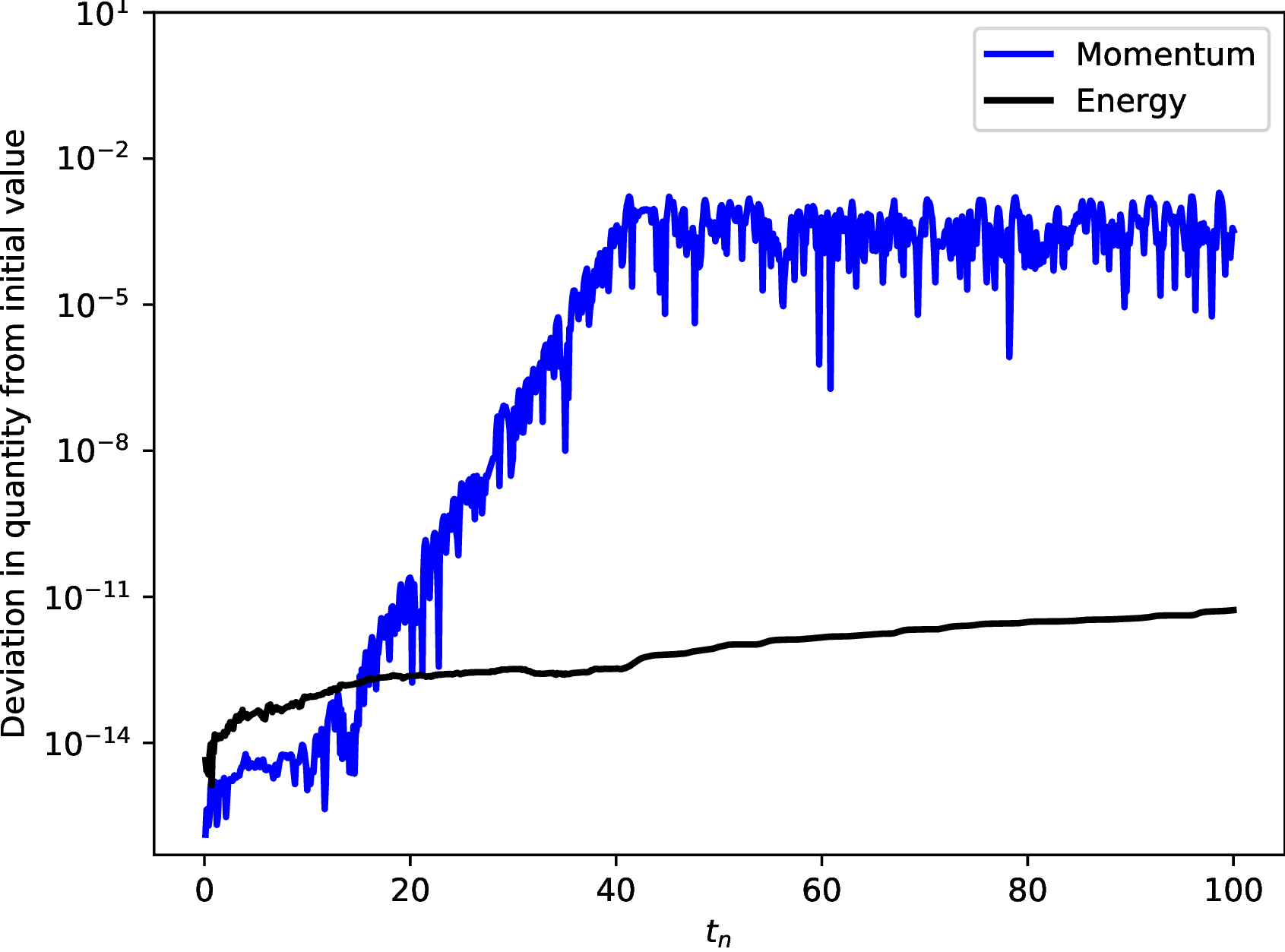}
  } \subfigure[][$q=2$ and $p=2$]{ \includegraphics[
    width=0.30\textwidth]{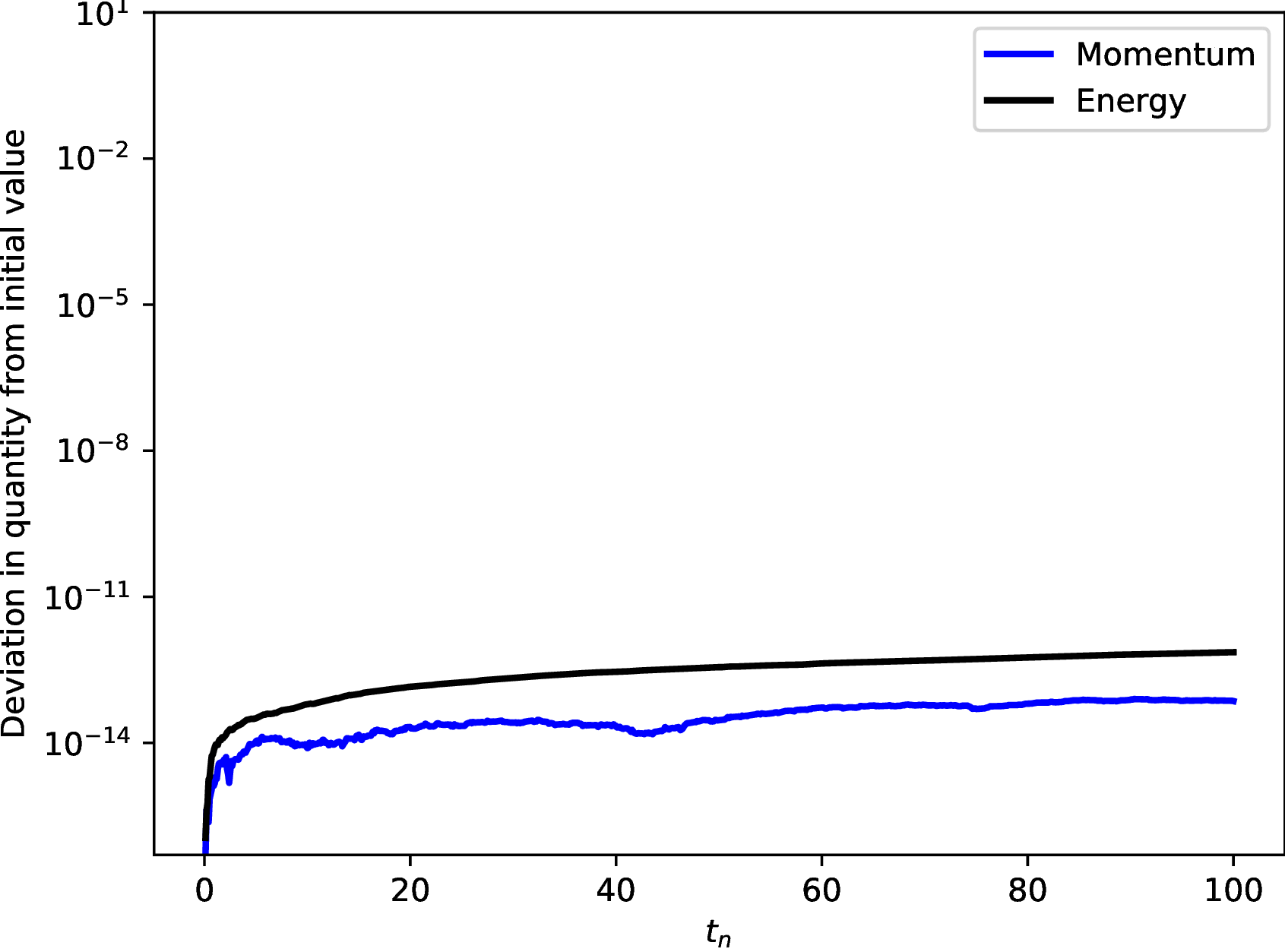}
  } \subfigure[][$q=2$ and $p=3$]{ \includegraphics[
    width=0.30\textwidth]{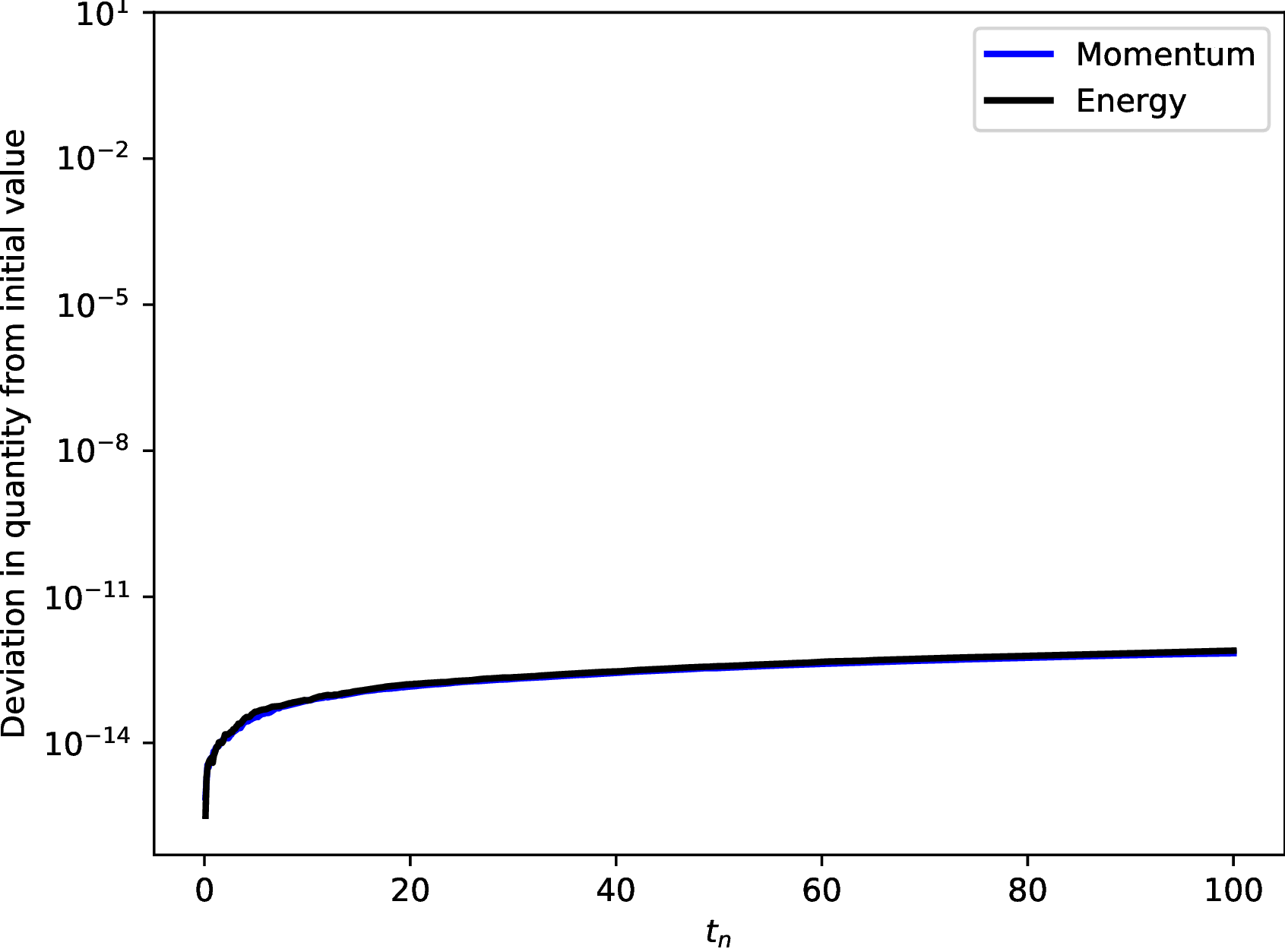}
  }

  \caption{The deviation in the conservation laws
    \eqref{eqn:momentum2} and \eqref{eqn:energy} for the
    finite element scheme \eqref{eqn:stlfem} for the
    NLS equation \eqref{eqn:nls} initialised by \eqref{eqn:nlsexact}
    with varying temporal degree $q$ and spatial degree $p$. All
    simulations have uniform temporal and spatial elements with
    $\dt{}=0.1$ and $\dx{}=0.04$. We notice that for $p>1$ all laws
    are preserved below machine precision. For $p=1$ the deviation in
    momentum propagates significantly over time. \label{fig:nls:dev:dg}}
\end{figure}
We note that the energy \eqref{eqn:energy} is exactly conserved
locally with errors propagating over time, however, the momentum
\eqref{eqn:momentum2} is not preserved locally unless $p>1$. We note
if considering the momentum conserving scheme \eqref{eqn:mstfem},
while the momentum would be well preserved the deviation in energy
would be large for all $p$.

\FloatBarrier

\section{Conclusion} \label{sec:conclusion}

We introduced space-time finite element approximations which
preserve geometric structure, namely the energy conservation law, of
an arbitrary parabolic multisymplectic PDE. Furthermore we proved
existence and uniqueness of this approximation for the wave equation,
followed by an extensive numerical study showing good long term
behaviour of the solution and convergence of the
simulations. Experimentally, the approximations converged optimally in
time, with sub-optimal spatial behaviour for both discretisations
considered. Our approximation facilitates discretisations over general
space-time tensor product domains, and may be readily implemented on
nonuniform meshes shaped by an appropriate mesh density function or a
posteriori indicator. Furthermore, our methods may be naturally
extended to multisymplectic PDEs of a higher spatial dimension,
allowing for discretisations over irregular domains. In future work,
we shall use the proposed discretisation here to inform fully adaptive
space-time algorithms which preserve geometric structure.

\section*{Acknowledgements}

This work has been partially supported through the IMA small grants
program, the Canadian Research Chairs and NSERC Discovery grant
programs (J.J.), H2020-MSCA-RISE-$691070$ CHiPS and the Simons foundation
(E.C.). Both authors would additionally like to acknowledge the
support of the Isaac Newton Institute for Mathematical Sciences,
Cambridge through the EPSRC grant EP/K032208/1.

The authors also gratefully acknowledge the suggestions of Dr Colin
Cotter which proved invaluable for the space-time finite element
implementation utilised in Section \ref{sec:numerics}.

\bibliographystyle{abbrv}
\bibliography{multisym}

\end{document}